\documentclass{amsart}
\usepackage{mathabx}
\usepackage{stmaryrd}
\usepackage{enumerate}
\usepackage[pdftex]{graphicx}
\usepackage{mathrsfs}
\usepackage{amsmath}
\usepackage{amsthm}
\usepackage{color}
\usepackage{tikz}

\newtheorem{thm}{Theorem}[section]
\newtheorem{prop}[thm]{Proposition}
\newtheorem{lem}[thm]{Lemma}
\newtheorem{cor}[thm]{Corollary}

\theoremstyle{definition}
\newtheorem{definition}[thm]{Definition}

\theoremstyle{remark}
\newtheorem{notation}[thm]{Notation}
\newtheorem{remark}[thm]{Remark}

\numberwithin{equation}{section}

\newcommand{\Span}{\mathrm{Span}}
\newcommand{\im}{\mathrm{im}}

\newcommand{\id}{\mathrm{id}}
\newcommand{\Gr}{\mathrm{Gr}}

\newcommand{\Amc}{\mathcal{A}}

\newcommand{\Cmc}{\mathcal{C}}
\newcommand{\Dmc}{\mathcal{D}}
\newcommand{\Emc}{\mathcal{E}}
\newcommand{\Fmc}{\mathcal{F}}
\newcommand{\Gmc}{\mathcal{G}}

\newcommand{\Mmc}{\mathcal{M}}

\newcommand{\Pmc}{\mathcal{P}}
\newcommand{\Qmc}{\mathcal{Q}}
\newcommand{\Smc}{\mathcal{S}}
\newcommand{\Tmc}{\mathcal{T}}
\newcommand{\Umc}{\mathcal{U}}
\newcommand{\Vmc}{\mathcal{V}}
\newcommand{\Xmc}{\mathcal{X}}

\newcommand{\Cbbb}{\mathbb{C}}
\newcommand{\Dbbb}{\mathbb{D}}

\newcommand{\Pbbb}{\mathbb{P}}
\newcommand{\Rbbb}{\mathbb{R}}
\newcommand{\Sbbb}{\mathbb{S}}
\newcommand{\Zbbb}{\mathbb{Z}}

\newcommand{\rp}{\mathbb{RP}}
\newcommand{\re}{\mathbb{R}}

\newcommand{\PGL}{\mathrm{PGL}}

\newcommand{\SL}{\mathrm{SL}}
\newcommand{\MCG}{\mathrm{MCG}}

\begin{document}

\title{Coordinates on the Augmented Moduli Space of Convex $\rp^2$ Structures}

\thanks{T.Z. was partially supported by the National Science Foundation under grant DMS-1566585, and by the NUS-MOE grant R-146-000-270-133. The authors also acknowledge support from the GEAR Network, funded by the National Science Foundation under grant numbers
DMS 1107452, 1107263, and 1107367 (``RNMS: GEometric structures And Representation
varieties''.)}

\author{John Loftin}
\address{Department of Mathematics and Computer Science, Rutgers University Newark, 101 Warren St., Newark, NJ 07102 USA}
\email{loftin@rutgers.edu}

\author{Tengren Zhang}
\address{Department of Mathematics, National University of Singapore, 21 Lower Kent Ridge Road, Singapore 119077}
\email{matzt@nus.edu.sg}

\maketitle

\begin{abstract}
Let $S$ be an orientable, finite type surface with negative Euler characteristic. The augmented moduli space of convex real projective structures on $S$ was first defined and topologized by the first author. In this article, we give an explicit description of this topology using explicit coordinates. More precisely, given every point in this augmented moduli space, we find explicit continuous coordinates on the quotient of a suitable open neighborhood about this point by a suitable subgroup of the mapping class group of $S$. Using this, we give a simpler proof of the fact that the augmented moduli space of convex real projective structures on $S$ is homeomorphic to the orbifold vector bundle of regular cubic differentials over the Deligne-Mumford compactification of the moduli space of Riemann surfaces homeomorphic to $S$.
\end{abstract}

\tableofcontents

\section{Introduction}
Let $S$ denote a smooth, connected, oriented, finite-type surface with negative Euler characteristic. A convex $\rp^2$ structure $\mu$ on $S$ is determined by a pair $(\phi,\rho)$, where $\rho:\pi_1(S)\to\PGL(3,\Rbbb)$ is a representation and $\phi:\widetilde{S}\to\Omega$ is a $\rho$-equivariant diffeomorphism onto a properly convex domain $\Omega\subset\Rbbb\Pbbb^2$. The pair $(\phi,\rho)$ is usually known as a \emph{developing pair} for $\mu$, while $\rho$ and $\phi$ are called a \emph{holonomy representation} and a \emph{developing map} of $\mu$ respectively. We will denote the deformation space of convex $\Rbbb\Pbbb^2$ structures on $S$ by $\Cmc(S)$. (See Section \ref{sec:convex} for more precise definitions.) Hyperbolic structures are then examples of convex $\rp^2$ structures via the Klein model of hyperbolic space, where $\Omega$ is a round disk.  In this article, we give local coordinates that are adapted to describe degenerations of convex $\rp^2$ structures on $S$ that converge on the complement of a multi-curve $\Dmc$ (see Remark \ref{rem: top}(1)).  Our coordinates generalize natural orbifold coordinates, based on Fenchel-Nielsen coordinates, near the boundary of the Deligne-Mumford compactification $\overline{\mathcal M(S)}$ of the moduli space of finite-area hyperbolic structures on $S$. 

In order to motivate this choice of limiting structure to study, we recall the case of hyperbolic structures on $S$ when $S$ is a closed oriented surface of genus at least $2$. The deformation space of all (marked) hyperbolic structures is the Teichm\"uller space $\mathcal T(S)$, which is homeomorphic to $\re^{6g-6}$. There are two major and essentially different ways to analyze degenerating families of hyperbolic structures on $S$. First, Thurston gives a natural compactification $\overline{\Tmc(S)}$ of $\mathcal T(S)$ which may be seen as the set of limits in the projective space of hyperbolic lengths of all closed geodesics on $S$. Second, for (unmarked) hyperbolic structures on $S$, the moduli space $\mathcal M(S)$, which is the quotient of $\mathcal T(S)$ by the mapping class group $\mathrm{MCG}(S)$, has the structure of a quasi-projective algebraic variety with orbifold singularities. 
Its most natural compactification, the Deligne-Mumford compactification $\overline{\mathcal M(S)}$, is then formed by considering all complete, finite-area, hyperbolic structures on $S\setminus \mathcal D$ for all free homotopy classes of multi-curves $\mathcal D$ in $S$. One can also define an augmentation $\Tmc(S)^{\mathrm{aug}}$ of $\Tmc(S)$ to be $\Tmc(S)$ together with all the possible limits of degenerating families of hyperbolic structures on $S$ so that the family converges on the complement of a multicurve. It is then well-known \cite{Abi,HubKoch14} that $\overline{\mathcal M(S)}=\Tmc(S)^{\mathrm{aug}}/\mathrm{MCG}(S)$.

These two spaces $\overline{\Tmc(S)}$ and $\Tmc(S)^{\mathrm{aug}}$ are fundamentally different; most rays in Teichm\"uller space that converge to points in $\overline{\Tmc(S)}$ do not project under the quotient map $\Tmc(S)\to\Mmc(S)$ to convergent rays in $\overline{\mathcal M(S)}$. The reason for this is the following: if a family of hyperbolic structures on $S$ is pinched along a multi-curve $\mathcal D$, its accumulation set in $\overline{\mathcal M(S)}$ depends on the limiting hyperbolic structures on the complement of $\Dmc$. However, if we choose a lift of this sequence to Teichm\"uller space, then its limit in $\overline{\mathcal T(S)}$ only records the relative hyperbolic lengths of closed geodesics whose lengths are growing the fastest along this sequence. In particular, the hyperbolic structure on $S\setminus\Dmc$ is forgotten in $\overline{\mathcal T(S)}$.

The present work addresses, for the case of convex $\rp^2$ structures, analogs of the geometry of $\overline{\mathcal M(S)}$. In \cite{Loftin}, the first author introduced \emph{regular convex $\rp^2$ structures} which serve to augment the deformation space $\mathcal C(S)$. These are convex $\Rbbb\Pbbb^2$ structures on $S$, together with all the possible limits of degenerating families of convex $\Rbbb\Pbbb^2$ structures, with the property that the family converges on the complement of a multi-curve. The \emph{augmented deformation space $\mathcal C(S)^{\rm aug}$} is then the set of all regular convex $\rp^2$ structures on $S$. One should think of $\Cmc(S)^{\mathrm{aug}}$ as a generalization of $\Tmc(S)^{\mathrm{aug}}$ to the setting of convex $\Rbbb\Pbbb^2$ structures on $S$. The first author also defined a natural topology on $\mathcal C(S)^{\rm aug}$. With this topology, $\mathcal C(S)^{\rm aug}$ has a stratification, where each stratum $\mathcal C(S,\Dmc)^{\rm adm}\subset\Cmc(S)^{\mathrm{aug}}$ is determined by a multi-curve $\Dmc$ on $S$. (See Section \ref{sec: aug top} for more details.)

Despite its naturality, the topology on $\mathcal C(S)^{\rm aug}$ is rather abstract in terms of the geometry of the limiting surfaces. The purpose of this paper is to elucidate the geometric properties of families of regular convex $\rp^2$ structures by using (global) Fenchel-Nielsen type coordinates on the space of holonomies of the convex $\Rbbb\Pbbb^2$ structures on $S$ to construct (local) coordinates on appropriate quotients of $\mathcal C(S)^{\rm aug}$ by certain subgroups of the mapping class group. We show that these coordinates induce the topology on the \emph{augmented moduli space} $\mathcal C(S)^{\rm aug}  / \mathrm{MCG}(S)$. More precisely, we have the following main theorem (also see Theorem \ref{thm: main}).

\begin{thm}\label{thm: main intro}
Let $\mu\in\Cmc(S)^{\mathrm{aug}}$, and let $\Dmc$ be the multi-curve on $S$ so that $\mu\in\Cmc(S,\Dmc)^{\mathrm{adm}}$. Also, let $G_\Dmc\subset\mathrm{MCG}(S)$ be the subgroup generated by Dehn twists about the simple closed curves in $\Dmc$. Then
\[\Cmc(S)^{\mathrm{aug},\Dmc}:=\bigcup_{\Dmc'\subset\Dmc}\Cmc(S,\Dmc')^{\rm adm}\]
is an open set of $\Cmc(S)^{\mathrm{aug}}$ containing $\mu$ that is invariant under $G_\Dmc$, and there is a homeomorphism
\[\Psi_\Dmc:\Cmc(S)^{\mathrm{aug},\Dmc}/G_\Dmc\to\Rbbb^{10g-10+6n+2m}\times(\Rbbb_+)^{6g-6+2n-2m},\]
where $m$ is the number of curves in $\Dmc$, $n$ is the number of punctures of $S$, and $g$ is the genus of the compactification of $S$ in which each puncture is filled in. In particular, $\Cmc(S)^{\mathrm{aug},\Dmc}/G_\Dmc$ is a cell of dimension $16g-16+8n$.
\end{thm} 

Furthermore, the coordinate functions of $\Psi_\Dmc$ are explicitly described. See Section \ref{sec: coord} for the description of $\Psi_\Dmc$.

It turns out that if $\Dmc$ is non-empty, then any open set of $\Cmc(S)^{\mathrm{aug}}$ containing $\mu$ in the above theorem does not have compact closure. On the other hand, any open set of the augmented moduli space $\Cmc(S)^{\mathrm{aug}}/\mathrm{MCG}(S)$ containing $[\mu]$ might have a complicated singular locus. (See Section \ref{sec: aug top}.) Thus, it is necessary to quotient $\Cmc(S)^{\mathrm{aug},\Dmc}$ by the appropriate subgroup $G_{\Dmc}$ of $\mathrm{MCG}(S)$ for it to have a nice set of coordinates.

Theorem \ref{thm: main intro} is a generalization of a standard result in Teichm\"uller theory describing the behavior of Fenchel-Nielsen coordinates at the boundary of $\overline{\mathcal M(S)}$, the Deligne-Mumford compactification of the the moduli space of finite-area hyperbolic structures on $S$.  The following theorem is well-known (see e.g.\ \cite{HubKoch14})
\begin{thm}
Let $\Dmc$ be a multi-curve on $S$, and let $\mathcal T(S)^{{\rm aug},\mathcal D}$ be Teichm\"uller space augmented along curves in $\Dmc$ only. Choose a pants decomposition $\Pmc\supset\Dmc$ on $S$, and let $\ell_i,\theta_i$ denote the length and twist coordinates on curves in $\Pmc$, where the $\theta_i$ are normalized so that Dehn twists are represented by $\theta_i\mapsto\theta_i+2\pi$. Then $\Tmc(S)^{{\rm aug},\Dmc}/G_\Dmc$ can be described by coordinates 
$$\left( \prod_{j=1}^m (\ell_j\cos(\theta_j),\ell_j\sin(\theta_j)) \right) \times \left(\prod_{j=m+1}^{3g-3+n}(\ell_j,\theta_j)  \right) \in\re^{2m}\times (\re_+ \times \re)^{3g-3+n-m},$$ 
and is homeomorphic to a cell of dimension $6g-6+2n$. 
\end{thm}
Note in this theorem that about each curve $c$ in $\Dmc$, the length parameter is allowed to become 0 and the twist parameter is considered modulo $2\pi$.  In our Theorem \ref{thm: main intro}, similar but more complicated constructions on generalized length and twist parameters about $c$ are needed to define $\Psi_\Dmc$.

Theorem \ref{thm: main intro} (together with Theorem \ref{thm: hol}) gives us a complete description of the behavior of the boundary of the image of the developing map in families of convex $\rp^2$ structures degenerating to a regular convex $\rp^2$ structure, including new behavior which does not occur in the study of hyperbolic structures in $\overline{\mathcal M}_g$. The new phenomenon is that the limit set of the holonomy representation (restricted to a component of $S\setminus\Dmc$) in the boundary of the image of the developing map might be a proper subset of the boundary. However, we can still describe the behavior of boundary in this case. 

One can generalize the holonomy  $\rho\!:\pi_1(S) \to \PGL(3,\re)$ of a convex $\Rbbb\Pbbb^2$ structure to representations from $\pi_1(S)$ to other split real Lie groups (in particular for Hitchin representations into $\PGL(n+1,\re)$, \cite{Hitchin92,Labourie06,Guichard05}).  We hope the detailed model of the degeneration of the convex boundary curves given in Theorem \ref{thm: main intro} will be of help to analyze families of Hitchin representations which degenerate along a multi-curve.  

Via the uniformization theorem, one may also view $\Mmc(S)$ as the moduli space of Riemann surfaces homeomorphic to $S$, which is itself naturally a complex orbifold. When $S$ is closed, there is a natural holomorphic vector bundle $\mathcal K(S)$ over $\Mmc(S)$ whose fiber above every point $X\in\Mmc(S)$ is the vector space of holomorphic cubic differentials on $X$. Labourie \cite{Labourie07} and the first author \cite{Loftin2001} independently constructed a natural homeomorphism
\begin{equation}\label{eqn}\Cmc(S)/\mathrm{MCG}(S)\simeq\mathcal{K}(S).\end{equation}

Later, the first author \cite{Loftin} defined the notion of a regular convex $\rp^2$ structure on a (not necessarily compact) Riemann surface $S$. Also, in part by using \cite{Loftin2004,BenHul13,Nie15}, he extended (\ref{eqn}) by proving that there is a natural homeomorphism $\mathcal C(S)^{\rm aug}  / \mathrm{MCG}(S)\simeq\mathcal K(S)^{\mathrm{reg}}$, where $\mathcal K(S)^{\mathrm{reg}}$ is the orbifold vector bundle of regular cubic differentials over the Deligne-Mumford compactification of the moduli space of Riemann surfaces $\overline{\Mmc(S)}$. (See Section \ref{sec: conv cubic} below.) In particular, the main theorem of \cite{Loftin} shows that there are (local) holomorphic coordinates (up to local finite group actions) on the augmented moduli space $\mathcal C(S)^{\rm aug}/\mathrm{MCG}(S)$.

As a consequence of Theorem \ref{thm: main intro}, we significantly simplify the proof of half of the main theorem in \cite{Loftin} by applying Brouwer's Invariance of Domain Theorem to the real coordinates we construct here and the holomorphic coordinates induced by the regular cubic differentials. More precisely, we arrive at the following corollary.

\begin{cor}[Loftin]\label{cor:Loftin}
There is a natural homeomorphism $\Cmc(S)^{\mathrm{aug}}/\mathrm{MCG}(S)\simeq\mathcal{K}(S)^{\mathrm{reg}}$. When $S$ is closed, this extends the homeomorphism (\ref{eqn}).
\end{cor}

The main tool used to prove Theorem \ref{thm: main intro} are Fenchel-Nielsen type coordinates on the set of holonomies of convex $\Rbbb\Pbbb^2$ structures on $S$. In the case when $S$ is closed, these kinds of coordinates were first constructed by Goldman \cite{Go90}. Choi-Goldman \cite{ChoiGo1} also showed that these holonomies form the Hitchin component of the space of representations from $\pi_1(S)$ into $\PGL(3,\re)$. In the more general setting of Hitchin representations into $\PGL(n,\Rbbb)$, Fock-Goncharov, Bonahon-Dreyer, and the second author further developed analogs of Goldman's coordinates \cite{FocGon,BonDre,Zhang16} which are more amenable to our construction than Goldman's original coordinates. Marquis \cite{Mar10} extended Goldman's coordinates to the case of finite area convex $\Rbbb\Pbbb^2$ structures on a (possibly) punctured surface $S$. In this paper, we study convex $\rp^2$ structures on punctured surfaces by extending the  coordinates of \cite{FocGon,BonDre,Zhang16} instead.  Other works studying noncompact convex $\rp^2$ surfaces  from various points of view include \cite{BenHul13,BenHul14,Choi94II,DumWolf15,Loftin2004,Mar12,Nie15}.

Here is a brief description of the structure of this paper.
In Section \ref{admissible-section}, we follow \cite{Loftin} to define $\Cmc(S)^{\mathrm{aug}}$ and its topology. Then in Section \ref{sec: hol}, we describe the (global) coordinates on the image ${\rm hol}(\mathcal C(S))$ of the holonomy map from $\Cmc(S)$ to $\Xmc(\pi_1(S),\PGL(3,\Rbbb))/\PGL(3,\Rbbb)$, and use them to construct local coordinates on appropriate quotients of $\Cmc(S)^{\mathrm{aug}}$. For this part, a key point is that unlike the case of compact $S$, the holonomy of a regular convex $\Rbbb\Pbbb^2$ structure on a non-compact surface $S$ of negative Euler characteristic does not always determine the projective structure at the ends of $S$. We then proceed to prove Theorem \ref{thm: main intro} in Section \ref{sec: coord} by showing that the coordinates constructed in Section \ref{sec: hol} describe the topology on $\Cmc(S)^{\mathrm{aug}}/\mathrm{MCG}(S)$ described in Section \ref{admissible-section}. Section \ref{sec: conv cubic} relates our description of the topology on $\Cmc(S)^{\mathrm{aug}}/\mathrm{MCG}(S)$ to the regular cubic differentials studied in \cite{Loftin}, which allows us to recover Corollary \ref{cor:Loftin}. Finally, in the Appendix, we present the proof of Theorem \ref{thm: hol}, which describes how the image of the developing map of two regular convex $\Rbbb\Pbbb^2$ structures on $S$ with the same holonomy can differ. In a result which may be of independent interest, we also give in the Appendix a description of the limit set of any convex $\rp^2$ structure on $S$. 

{\bf Acknowledgements:} The first author would like to thank Bill Goldman for many inspiring discussions about $\rp^2$ structures. 
\section{Admissible convex real projective structures} \label{admissible-section}

In this section, we define admissible convex $\Rbbb\Pbbb^2$ structures on finite type surfaces, as well as some terminology describing the holonomy of these structures about the ends of the surface.

\subsection{Convex real projective structures}\label{sec:convex}
We begin by recalling some standard definitions and properties of $\Rbbb\Pbbb^2$ structures on surfaces.

\begin{definition}\
\begin{enumerate}
\item An \emph{$\Rbbb\Pbbb^2$ surface} $\Sigma$ is a smooth, connected, closed surface with finitely many punctures, that is equipped with a maximal collection of smooth maps $\{\psi_\alpha:U_\alpha\to\Rbbb\Pbbb^2\}_\alpha$ so that the following holds.
\begin{itemize}
\item Each $U_\alpha\subset\Sigma$ is a connected and simply connected open subset,
\item For any pair of smooth maps $\psi_\alpha$ and $\psi_\beta$, $\psi_\alpha\circ\psi_\beta^{-1}:\psi_\beta(U_\alpha\cap U_\beta)\to\psi_\alpha(U_\alpha\cap U_\beta)$ is a restriction of a projective transformation on $\Rbbb\Pbbb^2$ to each connected component of $\psi_\beta(U_\alpha\cap U_\beta)$.
\end{itemize}
The smooth maps $\psi_\alpha$ are called \emph{charts} of $\Sigma$.
\item Let $\Sigma$ and $\Sigma'$ be two $\Rbbb\Pbbb^2$ surfaces with with maximal atlases $\{\psi_\alpha\}_\alpha$ and $\{\psi_\alpha'\}_\alpha$ respectively. A diffeomorphism $f:\Sigma\to\Sigma'$ is a \emph{projective isomorphism} if for any charts $\psi_\alpha:U_\alpha\to\Rbbb\Pbbb^2$ of $\Sigma$ and $\psi'_\beta:U_\beta'\to\Rbbb\Pbbb^2$  of $\Sigma'$ so that $f(U_\alpha)\cap U_\beta'$ is non-empty, the composition
\[\psi'_\beta\circ f\circ\psi_\alpha^{-1}:\psi_\alpha(U_\alpha\cap f^{-1}(U_\beta'))\to\psi'_\beta(f(U_\alpha)\cap U_\beta')\]
is the restriction of a projective transformation on $\Rbbb\Pbbb^2$ to each connected component of $\psi_\alpha(U_\alpha\cap f^{-1}(U_\beta'))$.
\end{enumerate}
\end{definition}

Let $\widetilde{\Sigma}$ be the universal cover of $\Sigma$. Then $\widetilde{\Sigma}$ is naturally an $\Rbbb\Pbbb^2$ surface. For any choice of chart $\widetilde{\psi}_\alpha:\widetilde{U}_\alpha\to\Rbbb\Pbbb^2$ of $\widetilde{\Sigma}$, one can construct via analytic continuation, a unique local diffeomorphism $\phi_\alpha:\widetilde{\Sigma}\to\Rbbb\Pbbb^2$ so that
\begin{itemize}
\item $\phi_\alpha|_{\widetilde{U}_\alpha}=\widetilde{\psi}_\alpha$,
\item for any point $p\in\widetilde{\Sigma}$, there is an open set $\widetilde{U}\subset\widetilde{\Sigma}$ so that $p\in \widetilde{U}$ and $\phi_\alpha|_{\widetilde{U}}:\widetilde{U}\to\Rbbb\Pbbb^2$ is a chart for $\widetilde{\Sigma}$.
\end{itemize}
The local diffeomorphism $\phi_\alpha$ is usually known as a \emph{developing map} for $\Sigma$. It induces a group homomorphism $\rho_\alpha:\pi_1(\Sigma)\to\PGL(3,\Rbbb)$ with the defining property that $\phi_\alpha$ is $\rho_\alpha$-equivariant. This homomorphism is usually called a \emph{holonomy representation} of $\Sigma$, and the pair $(\phi_\alpha,\rho_\alpha)$ is a \emph{developing pair} for $\Sigma$.

Observe that for any pair of charts $\widetilde{\psi}_\alpha:\widetilde{U}_\alpha\to\Rbbb\Pbbb^2$ and $\widetilde{\psi}_\beta:\widetilde{U}_\beta\to\Rbbb\Pbbb^2$ of $\widetilde{\Sigma}$, there is some $g\in\PGL(3,\Rbbb)$ so that
\[(\phi_\alpha,\rho_\alpha)=(g\circ\phi_\beta,c_g\circ\rho_\beta),\]
where $c_g:\PGL(3,\Rbbb)\to \PGL(3,\Rbbb)$ is conjugation by $g$. In particular, the holonomy representation of $\Sigma$ is unique up to conjugation.

\begin{definition} \
\begin{enumerate}
\item A domain $\Omega\subset\Rbbb\Pbbb^2$ is \emph{properly convex} if its closure in $\Rbbb\Pbbb^2$ does not contain any projective lines, and for any pair of distinct points $p,q\in\Omega$, there is a projective line segment with endpoints $p,q$ that lies entirely in $\Omega$.
\item A connected $\Rbbb\Pbbb^2$ surface $\Sigma$ is \emph{convex} if any (equivalently, some) developing map of $\Sigma$ is a diffeomorphism onto a properly convex domain in $\Rbbb\Pbbb^2$.
\item The \emph{deformation space of convex $\Rbbb\Pbbb^2$ structures on $S$} is
\[\Cmc(S):=\left\{(f,\Sigma):\begin{array}{l}
\Sigma\text{ is a convex }\Rbbb\Pbbb^2\text{ surface}\\
f:S\to\Sigma\text{ is a diffeomorphism}\end{array}\right\}\Bigg/\sim,\]
where $(f,\Sigma)\sim(f',\Sigma')$ if $f'\circ f^{-1}:\Sigma\to\Sigma'$ is homotopic to a projective isomorphism from $\Sigma$ to $\Sigma'$. An equivalence class $[f,\Sigma]\in\Cmc(S)$ is a \emph{(marked) convex $\Rbbb\Pbbb^2$ structure} on $S$.
\end{enumerate}
\end{definition}

Let $\mu\in\Cmc(S)$, let $(f,\Sigma)$ be a representative of $\mu$, and let $(\phi_\alpha,\rho_\alpha)$ be any developing pair of $\Sigma$. The diffeomorphism $f:S\to\Sigma$ induces an isomorphism $f_*:\pi_1(S)\to\pi_1(\Sigma)$, and also lifts to a map $\widetilde{f}:\widetilde{S}\to\widetilde{\Sigma}$. Define the pair
\[(\phi,\rho):=(\widetilde{f}\circ\phi_\alpha,f_*\circ\rho_\alpha).\]
Note that $\rho:\pi_1(S)\to\PGL(3,\Rbbb)$ is injective, and that $\phi:\widetilde{S}\to\Rbbb\Pbbb^2$ is a $\rho$-equivariant diffeomorphism onto a properly convex domain. Furthermore, since $\rho(\pi_1(S))$ acts properly discontinuously on $\Omega$, it is a discrete subgroup of $\PGL(3,\Rbbb)$. We refer to $(\phi,\rho)$ as a \emph{developing pair} for the convex $\mathbb{RP}^2$ structure $\mu=[(f,\Sigma)]\in\mathcal{C}(S)$.

The following well-known theorem (see e.g.\ Goldman \cite{Go90} Section 2.2) states that any developing pair for any representative $(f,\Sigma)$ of $\mu\in\Cmc(S)$ determines $\mu$.

\begin{thm}\label{thm:dev}
Let $\rho:\pi_1(S)\to\PGL(3,\Rbbb)$ be an injective representation, and $\phi:\widetilde{S}\to\Rbbb\Pbbb^2$ be a $\rho$-equivariant diffeomorphism onto a properly convex domain $\Omega\subset\Rbbb\Pbbb^2$. Then there is a unique $\mu\in\Cmc(S)$ so that $(\phi,\rho)$ is the developing pair for $\mu$. Furthermore, $(\phi,\rho)$ and $(\phi',\rho')$ are developing pairs for $\mu\in\Cmc(S)$ if and only if there is some $g\in\PGL(3,\Rbbb)$ so that $\rho'=c_g\circ\rho$ and $\phi'$ is homotopic to $g\circ\phi$ via a $\rho'$-equivariant homotopy.
\end{thm}

The domain $\Omega$ in the above proposition will be referred to as the $\rho$-\emph{equivariant developing image} of $\mu$. 

Given a properly convex domain $\Omega\subset\Rbbb\Pbbb^2$, we can define the \emph{Hilbert metric} on $\Omega$ as follows.

\begin{definition}
Let $\Omega\subset\Rbbb\Pbbb^2$ be a properly convex domain and let $x,y\in\Omega$. Let $\ell$ be a projective line in $\Rbbb\Pbbb^2$ through $x$ and $y$, and let $a,b\in\ell$ so that $\{a,b\}=\partial\Omega\cap\ell$ and $a,x,y,b$ lie in $\ell$ in this order. Define
\[d_\Omega(x,y):=\frac{1}{2}\log[a,x,y,b],\]
where $[a,x,y,b]$ is the cross ratio of four points $a,x,y,b$ on the projective line $\ell$.
\end{definition}

One can verify that $d_\Omega$ defines a metric on $\Omega$, and $(\Omega,d_\Omega)$ is a complete, proper, path metric space. Furthermore, since the cross ratio is a projective invariant, $d_\Omega$ is invariant under the projective transformations that preserve $\Omega$. Thus, if $\Omega$ is the $\rho$-equivariant developing image of some $\mu\in\Cmc(S)$, $d_\Omega$ descends to a metric on $\Omega/\rho(\pi_1(S))$. Projective line segments are geodesics of $d_\Omega$.

Using this, we show that any holonomy representation of $\mu$, together with the $\rho$-equivariant developing image, determines $\mu$.

\begin{prop}\label{prop:devimage}
Let $\mu_0,\mu_1\in\Cmc(S)$ with developing pairs $(\phi_0,\rho_0)$ and $(\phi_1,\rho_1)$. If $\rho_0=\rho_1:\pi_1(S)\to\PGL(3,\Rbbb)$ and $\phi_0(\widetilde{S})=\phi_1(\widetilde{S})\subset\Rbbb\Pbbb^2$, then $\mu_0=\mu_1$.
\end{prop}

\begin{proof}
By Theorem \ref{thm:dev}, it is sufficient to show that $\phi_0$ is homotopic to $\phi_1$ via a $\rho$-equivariant homotopy. Set $\Omega:=\phi_0(\widetilde{S})=\phi_1(\widetilde{S})$, and define $F:[0,1]\times\Omega\to\Omega$ by declaring $F(t,p)$ to be the unique point in the projective line segment in $\Omega$ between $\phi_0(p)$ and $\phi_1(p)$, so that $d_\Omega(\phi_0(p),F(t,p))=td_\Omega(\phi_0(p),\phi_1(p))$. It is clear that $F$ is a homotopy between $\phi_0$ and $\phi_1$, and it is easy to verify from the projective invariance of $d_\Omega$ that $F(t,\gamma\cdot p)=\rho(\gamma)\cdot F(t,p)$ for all $(t,p)\in[0,1]\times\Omega$ and all $\gamma\in\pi_1(S)$.
\end{proof}

As a consequence of Theorem \ref{thm:dev}, we may define the \emph{holonomy map}
\[\mathrm{hol}:\Cmc(S)\to\Xmc_3(S):=\mathrm{Hom}(\pi_1(S),\PGL(3,\Rbbb))/\PGL(3,\Rbbb),\]
where $\PGL(3,\Rbbb)$ acts on $\mathrm{Hom}(\pi_1(S),\PGL(3,\Rbbb))$ by conjugation. We will also denote $\mathrm{hol}_\mu:=\mathrm{hol}(\mu)$ when convenient.

We now describe natural topologies on $\Cmc(S)$ and $\Xmc_3(S)$. Let us start with the topology on $\Xmc_3(S)$. Given a finite generating set $\Smc$ of $\pi_1(S)$, $\mathrm{Hom}(\pi_1(S),\PGL(3,\Rbbb))$ is realized as a subset of $\PGL(3,\Rbbb)^{|\Smc|}$ by evaluating every $\rho\in\mathrm{Hom}(\pi_1(S),\PGL(3,\Rbbb))$ on $\Smc$. The standard topology on $\PGL(3,\Rbbb)$ thus induces a subspace topology on $\mathrm{Hom}(\pi_1(S),\PGL(3,\Rbbb))\subset\PGL(3,\Rbbb)^{|\Smc|}$, which can be verified to be independent of the choice of $\Smc$. This in turn induces the quotient topology on $\Xmc_3(S)$.

Next, we define a topology in $\Cmc(S)$. Choose a Riemannian metric on $\Rbbb\Pbbb^2$. For any properly convex domain $\Omega\subset\Rbbb\Pbbb^2$, let $U_{\Omega,\epsilon}$ denote the set of properly convex domains in $\Rbbb\Pbbb^2$ whose Hausdorff distance from $\Omega$ is less than $\epsilon$. The set $\Pmc\Cmc:=\{\text{properly convex domains in }\Rbbb\Pbbb^2\}$ can then be equipped with the topology generated by the basis
\[\left\{U_{\Omega,\epsilon}:
\Omega\subset\Rbbb\Pbbb^2\text{ is a properly convex domain and }\epsilon>0.\right\}\]
This topology does not depend on the choice of Riemannian metric on $\Rbbb\Pbbb^2$.

Using this, we can topologize $\Cmc(S)$ in the following way. Recall that any chart $\widetilde{\psi}_\alpha:\widetilde{U}_\alpha\to\Rbbb\Pbbb^2$ of $\widetilde{\Sigma}$ determines a developing pair $(\phi_\alpha,\rho_\alpha)$ for $\Sigma$. Equip 
\[\widetilde{\Cmc}(S):=\left\{(f,\Sigma,\widetilde{\psi}_\alpha):\begin{array}{l}
\Sigma\text{ is a convex }\Rbbb\Pbbb^2\text{ surface}\\
f:S\to\Sigma\text{ is a diffeomorphism}\\
\widetilde{\psi}_\alpha:\widetilde{U}_\alpha\to\Rbbb\Pbbb^2\text{ is a chart of }\widetilde{\Sigma}\end{array}\right\}.\]
with the topology generated by
\[\left\{U_{(f,\Sigma,\widetilde{\psi}_\alpha),U,V}:\begin{array}{l}(f,\Sigma,\widetilde{\psi}_\alpha)\in\widetilde{\Cmc}(S),\,\, U\subset\Pmc\Cmc\text{ is an open set containing }\phi_\alpha(\widetilde\Sigma)\\
V\subset\mathrm{Hom}(\pi_1(S),\PGL(3,\Rbbb))\text{ is an open set containing }f_*\circ\rho_\alpha,
\end{array}\right\},\]
where
\[U_{(f,\Sigma,\widetilde{\psi}_\alpha),U,V}:=\{(f',\Sigma',\widetilde{\psi}_\beta')\in\widetilde{\Cmc}(S):\psi'_\beta(\widetilde{\Sigma}')\in U, \,\,f'_*\circ \rho'_\beta\in V\}.\]
Since $\Cmc(S)$ can be realized as a quotient of $\widetilde{\Cmc}(S)$, the topology on $\widetilde{\Cmc}(S)$ induces a quotient topology on $\Cmc(S)$. It is clear from the definition of this topology that the holonomy map $\mathrm{hol}:\Cmc(S)\to\Xmc_3(S)$ is continuous.

\subsection{The admissibility condition}
To build an augmentation of the deformation space of convex $\Rbbb\Pbbb^2$ structures, we need to consider a particular subset $\Cmc(S)^{\mathrm{adm}}\subset\Cmc(S)$ that satisfy an admissibility condition (see Definition \ref{def:admissible}). The admissibility condition is a condition on the ends of convex $\Rbbb\Pbbb^2$ surfaces, which arises naturally in the course of studying degenerations of convex $\Rbbb\Pbbb^2$ structures. 

In the case when $S$ is closed, we have the following theorem. The first statement is due to Choi-Goldman \cite{ChoiGo1} while the second is due to Kuiper \cite{Kuiper54}.

\begin{thm}\label{thm: determined by holonomy}
Suppose that $S$ is a closed oriented surface. Then
\begin{enumerate}
\item $\mathrm{hol}:\Cmc(S)\to\Xmc_3(S)$ is a homeomorphism onto a connected component of $\Xmc_3(S)$, usually known as the $\PGL(3,\Rbbb)$-Hitchin component.
\item for every $\gamma\in\pi_1(S)\setminus\{\id\}$ and for every $\mu\in\Cmc(S)$, the conjugacy class $\mathrm{hol}_\mu(\gamma)$ contains a diagonal representative with pairwise distinct eigenvalues of the same sign.
\end{enumerate}
\end{thm}

Both statements in this theorem fail if $S$ has punctures. However, we still have the following results of Marquis \cite{Mar10}.

\begin{thm}[Marquis]\label{thm: holonomy1}
Let $\mu\in\Cmc(S)$.
\begin{enumerate}
\item If $\gamma\in\pi_1(S)\setminus\{\id\}$ is a non-peripheral element, then the conjugacy class $\mathrm{hol}_\mu(\gamma)$ contains a diagonal representative with pairwise distinct and positive eigenvalues.
\item If $\gamma\in\pi_1(S)$ is a peripheral element, then the conjugacy class $\mathrm{hol}_\mu(\gamma)$ has to contain
\[\left[\begin{array}{ccc}
1&1&0\\
0&1&1\\
0&0&1
\end{array}\right],\,\,\,\left[\begin{array}{ccc}
\lambda_1&0&0\\
0&\lambda_2&1\\
0&0&\lambda_2
\end{array}\right]\,\,\,\text{or}\,\,\,\left[\begin{array}{ccc}
\lambda_1&0&0\\
0&\lambda_2&0\\
0&0&\lambda_3
\end{array}\right]\]
for some pairwise distinct and positive $\lambda_1,\lambda_2,\lambda_3$ .
\end{enumerate}
\end{thm}

The three group elements in $\PGL(3,\Rbbb)$ listed in (2) of Theorem \ref{thm: holonomy1} are known as the \emph{standard parabolic element}, the \emph{standard quasi-hyperbolic element} and the \emph{standard hyperbolic element} respectively. Any group element in $\PGL(3,\Rbbb)$ is \emph{parabolic}, \emph{quasi-hyperbolic} or \emph{hyperbolic} if it is conjugate to the standard parabolic, quasi-hyperbolic or hyperbolic element respectively.

Motivated by the previous theorem, we will now restrict ourselves to convex $\Rbbb\Pbbb^2$ structures on $S$ whose ends are either parabolic type, quasi-hyperbolic type, or bulge $\pm\infty$ type, as described below.

\subsubsection{Parabolic type} Let $g\in\PGL(3,\Rbbb)$ be the standard parabolic element. Note that $g$ has a unique fixed point $q=[1,0,0]^T\in\Rbbb\Pbbb^2$ and stabilizes a projective line $l=l(g)=[0,0,1]$ through $q$. Choose any point $p\in\Rbbb\Pbbb^2\setminus l$ and let $l'$ be the projective line segment with endpoints $p$ and $g\cdot p$ that does not intersect $l$. Then the closed curve
\[q\cup\left(\bigcup_{i=-\infty}^\infty g^i\cdot l'\right)\]
bounds a properly convex domain $\Omega_p\subset\Rbbb\Pbbb^2$ that is invariant under the $\langle g\rangle$-action, where $\langle g\rangle$ denotes the cyclic group generated by $g$ (see Figure \ref{parabolic-figure}). Let $\Sigma_{g,p}:=\langle g\rangle\backslash\Omega_p$, and observe that $\Sigma_{g,p}$ is a convex $\Rbbb\Pbbb^2$ surface. Using $\Sigma_{g,p}$, we describe the first type of end that we allow the convex $\Rbbb\Pbbb^2$ surfaces we consider to have.

\begin{definition}
A puncture on a convex $\Rbbb\Pbbb^2$ surface is of \emph{parabolic type} if there is some $p\in\Rbbb\Pbbb^2\setminus l(g)$ so that $\Sigma_{g,p}$ is projectively isomorphic to a neighborhood of the puncture.
\end{definition}

One can verify that if $\Sigma$ is a convex $\Rbbb\Pbbb^2$ surface with a puncture and $\gamma\in\pi_1(\Sigma)$ is a peripheral element corresponding to this puncture, then this puncture is of parabolic type if and only if some (equiv. any) holonomy representation of $\Sigma$ evaluated at $\gamma$ is parabolic.

\begin{figure}
\begin{center}
\includegraphics{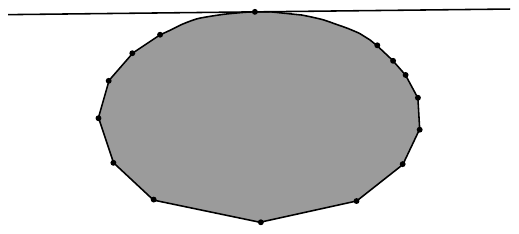}
\LARGE
\put (-125,50){$\Omega_p$}
\small
\put (-125, -2){$p$}
\put (-75, 9){$g\cdot p$}
\put (-20, 98){$l$}
\put (-125, 110){$q$}
\put (-98, 1){$l'$}
\end{center}
\caption{Parabolic}
\label{parabolic-figure}
\end{figure}

\subsubsection{Quasi-hyperbolic type}
Let $g\in\PGL(3,\Rbbb)$ be a standard quasi-hyperbolic element. Observe that $g$ stabilizes two projective lines $l_1=l_1(g)$ and $l_2=l_2(g)$, and has two fixed points $q_0,q_1\in\Rbbb\Pbbb^2$, where $q_1\in l_1\cap l_2$ and $q_0\in  l_2\setminus l_1$.

For any $p\in\Rbbb\Pbbb^2\setminus (l_1\cup l_2)$, let $l_3'$ be the projective line segment between $p$ and $g\cdot p$ that does not intersect $l_1$, and observe that there is a projective line segment $l_2'$ with endpoints $q_0$ and $q_1$ so that the closed curve
\[\left(\bigcup_{i=-\infty}^\infty g^i\cdot l_3'\right)\cup l_2'\]
bounds a properly convex domain $\Omega_{p}\subset\Rbbb\Pbbb^2$ that is invariant under the $\langle g\rangle$-action (see Figure \ref{quasi-hyp-figure}). This then defines a convex $\Rbbb\Pbbb^2$ surface $\Sigma_{g,p}:=\langle g\rangle\backslash\Omega_{p}$.

\begin{definition}
A puncture on a convex $\Rbbb\Pbbb^2$ surface is of \emph{quasi-hyperbolic type} if there is some standard quasi-hyperbolic element $g\in\PGL(3,\Rbbb)$ and some point $p\in\Rbbb\Pbbb^2\setminus (l_1(g)\cup l_2(g))$ so that $\Sigma_{g,p}$ is projectively isomorphic to a neighborhood of the puncture.
\end{definition}

One can verify that if $\Sigma$ is a convex $\Rbbb\Pbbb^2$ surface with a puncture and $\gamma\in\pi_1(\Sigma)$ is a peripheral element corresponding to this puncture, then this puncture is of quasi-hyperbolic type if and only if some (equiv. any) holonomy representation of $\Sigma$ evaluated at $\gamma$ is quasi-hyperbolic.

\begin{figure}
\begin{center}
\includegraphics{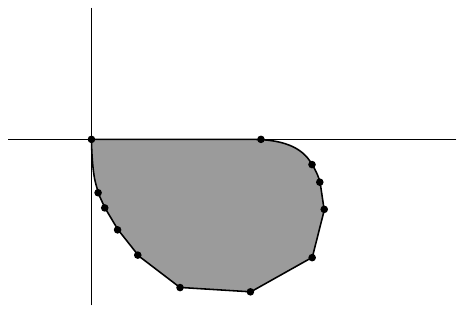}
\LARGE
\put (-125,50){$\Omega_p$}
\small
\put (-138, 4){$p$}
\put (-102, 2){$g\cdot p$}
\put (-20, 86){$l_2$}
\put (-185, 30){$l_1$}
\put (-176, 87){$q_1$}
\put (-95, 87){$q_0$}
\put (-123, 2){$l_3'$}
\end{center}
\caption{Quasi-Hyperbolic}
\label{quasi-hyp-figure}
\end{figure}

\subsubsection{Bulge $\pm\infty$ type}
Let $g\in\PGL(3,\Rbbb)$ be a standard hyperbolic element. Observe that $g$ has an attracting fixed point $q_1\in\Rbbb\Pbbb^2$, a repelling fixed point $q_3\in\Rbbb\Pbbb^2$, and a third fixed point $q_2\in\Rbbb\Pbbb^2$, which we refer to as the \emph{saddle fixed point}. Let $l_1=l_1(g)$, $l_2=l_2(g)$, $l_3=l_3(g)$ be the projective lines through $q_1$ and $q_2$, $q_2$ and $q_3$, $q_3$ and $q_1$ respectively.

Choose any point $p\in\Rbbb\Pbbb^2\setminus(l_1\cup l_2\cup l_3)$, and let $l_4'$ be the projective line segment with endpoints $p,g\cdot p$ that does not intersect $l_1$. Then $l_4'$ lies in an open triangle $\Delta_p$ which is a connected component of $\Rbbb\Pbbb^2\setminus(l_1\cup l_2\cup l_3)$. For $i=1,2,3$, let $l_i':=\partial\Delta_p\cap l_i$, and let $l_i''$ be the closure of $l_i\setminus l_i'$. Note that the closed curves
\[l_3'\cup\bigcup_{i=-\infty}^\infty g^i\cdot l_4'\,\,\,\,\text{and}\,\,\,\,l_1''\cup l_2''\cup\bigcup_{i=-\infty}^\infty g^i\cdot l_4'\]
both bound properly convex domains $\Omega_{p}'$ (see Figure \ref{bulge-minus-figure}) and $\Omega_{p}''$ (see Figure \ref{bulge-plus-figure}) respectively, which are invariant under the $\langle g\rangle$-action. Let $\Sigma_{g,p}':=\langle g\rangle\backslash\Omega_{p}'$ and $\Sigma_{g,p}'':=\langle g\rangle\backslash\Omega_{p}''$.

\begin{definition}
A puncture on a convex $\Rbbb\Pbbb^2$ surface $\Sigma$ is of \emph{bulge $-\infty$ type} (resp. \emph{bulge $+\infty$ type}) if there is some standard hyperbolic element $g\in\PGL(3,\Rbbb)$ and some $p\in\Rbbb\Pbbb^2\setminus (l_1(g)\cup l_2(g)\cup l_3(g))$ so that $\Sigma_{g,p}'$ (resp. $\Sigma_{g,p}''$) is projective isomorphic to a neighborhood of the puncture.
\end{definition}

\begin{figure}
\begin{center}
\includegraphics{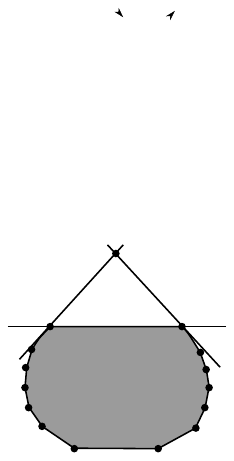}
\LARGE
\put (-63,30){$\Omega_p'$}
\small
\put (-80, -4){$p$}
\put (-39, -4){$g\cdot p$}
\put (-81, 82){$l_1$}
\put (-41, 82){$l_2$}
\put (-3, 60){$l_3$}
\put (-98, 66){$q_1$}
\put (-55, 96){$q_2$}
\put (-28, 67){$q_3$}
\put (-60, -5){$l_4'$}
\end{center}
\caption{Bulge $-\infty$}
\label{bulge-minus-figure}
\end{figure}

\begin{figure}
\begin{center}
\includegraphics{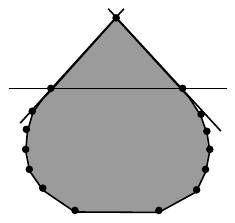}
\LARGE
\put (-63,30){$\Omega_p''$}
\small
\put (-78, -3){$p$}
\put (-37, -3){$g\cdot p$}
\put (-79, 83){$l_1$}
\put (-39, 83){$l_2$}
\put (-1, 61){$l_3$}
\put (-96, 67){$q_1$}
\put (-53, 97){$q_2$}
\put (-26, 69){$q_3$}
\put (-58, -4){$l_4'$}
\end{center}
\caption{Bulge $+\infty$}
\label{bulge-plus-figure}
\end{figure}

\begin{remark} If $\Sigma$ is a convex $\Rbbb\Pbbb^2$ surface with a puncture and $\gamma\in\pi_1(\Sigma)$ is a peripheral element corresponding to this puncture, then the holonomy representation evaluated at $\gamma$ does not determine the type of the puncture. This is unlike the case when the puncture is of parabolic type or quasi-hyperbolic type.
\end{remark}

Let $[f,\Sigma]\in\Cmc(S)$. For every puncture $p$ of $S$, the diffeomorphism $f$ sends a neighborhood of $p$ to a neighborhood of some puncture $q$ of $\Sigma$. We will abuse notation and denote $q=f(p)$.

\begin{definition}\label{def:admissible}\
\begin{enumerate}
\item A convex $\Rbbb\Pbbb^2$ surface $\Sigma$ is \emph{admissible} if every puncture of $\Sigma$ is either of parabolic type, quasi-hyperbolic type, bulge $+\infty$ type, or bulge $-\infty$ type. Similarly, $[f,\Sigma]\in\Cmc(S)$ is \emph{admissible} if $\Sigma$ is admissible.
\item Let $\Cmc(S)^{\mathrm{adm}}:=\{\mu\in\Cmc(S):\mu\text{ is admissible}\}$. For any $\mu=[f,\Sigma]\in\Cmc(S)^\mathrm{adm}$ and any puncture $p$ of $S$, the \emph{$\mu$-type of $p$} is the type of $f(p)$.
\end{enumerate}
\end{definition}

We equip $\Cmc(S)^\mathrm{adm}\subset\Cmc(S)$ with the subspace topology.

\subsection{The augmented deformation space.}\label{sec: aug}

Next, we describe the \emph{augmented deformation space of admissible convex $\Rbbb\Pbbb^2$ structures}, which was previously introduced and studied by the first author \cite{Loftin}.

\begin{remark}
In \cite{Loftin}, this augmented deformation space was constructed for closed $S$. However, the construction extends easily to general $S$ as long as we only consider admissible convex $\Rbbb\Pbbb^2$ structures.
\end{remark}

We begin by describing multi-curves $\Dmc$ in $S$.

\begin{definition} 
A \emph{multi-curve} is a collection of simple closed curves in $S$ that are pairwise non-intersecting, pairwise non-homotopic, non-contractible and non-peripheral. We allow the empty set as a multi-curve.
\end{definition}

\begin{remark}
A multi-curve in $S$ is a collection of non-peripheral closed curves corresponding to the vertices of a simplex in the curve complex equipped with its usual simplicial structure.
\end{remark}

Let $\Dmc$ be a multi-curve in $S$, and let $S_1,\dots,S_k$ be the connected components of $S\setminus\Dmc$.
Consider a closed curve $c\in\Dmc$, and choose an orientation on $c$. Let $S_i,S_j$ be the connected components of $S\setminus\Dmc$ that that lie on the left and right of $c$ respectively (it is possible that $S_i=S_j$). The orientation on $c$ determines two conjugacy classes of group elements $[\gamma_L]\in[\pi_1(S_i)]$ and $[\gamma_R]\in[\pi_1(S_j)]$ on the left and right of $c$ respectively. Let $p_L$ and $p_R$ be the punctures of $S_i$ and $S_j$ respectively that correspond to $c$.

\begin{definition}
A tuple $(\mu_1,\dots,\mu_k)\in\prod_{j=1}^k\Cmc(S_j)^{\mathrm{adm}}$ is \emph{compatible across $c$} if $\mathrm{hol}_{\mu_i}(\gamma_L)$ is conjugate to  $\mathrm{hol}_{\mu_j}(\gamma_R)$, and the $\mu_i$-type of $p_L$ is bulge $\pm\infty$ if and only if the $\mu_j$-type of $p_R$ is bulge $\mp\infty$.
\end{definition}

Denote
\[\Cmc(S,\Dmc)^\mathrm{adm}:=\left\{\mu\in\prod_{j=1}^k\Cmc(S_j)^{\mathrm{adm}}:\mu\text{ is compatible across }c \text{ for all }c\in\Dmc\right\}\]
and
\[\Cmc(S)^\mathrm{aug}:=\bigcup_{\Dmc\text{ a multi-curve}}\Cmc(S,\Dmc)^\mathrm{adm}.\]
We will refer to each $\Cmc(S,\Dmc)^\mathrm{adm}\subset\Cmc(S)^\mathrm{aug}$ as the $\Dmc$-\emph{stratum} of $\Cmc(S)^\mathrm{aug}$.

\begin{definition}
A \emph{regular} (marked) convex $\rp^2$ structure on $S$ is a point in $\Cmc(S)^{\mathrm{aug}}$.
\end{definition}

The terminology of a regular convex $\rp^2$ structure was introduced in \cite{Loftin}, and is so named because they correspond to holomorphic cubic differentials with regular singularities at the punctures, see Section \ref{sec: conv cubic}. Intuitively, one can think of $\Cmc(S)^\mathrm{aug}$ as an ``augmentation" of $\Cmc(S)^\mathrm{adm}$ where we include all possible degenerations of the convex $\Rbbb\Pbbb^2$ structure on $S$ that converge on the complement of a multi-curve (see Theorem 2.5.1 in \cite{Loftin}). This is an analog of the augmented Teichm\"uller space, denoted $\Tmc(S)^\mathrm{aug}$, for the deformation space of convex real projective structures. In fact, this construction of $\Cmc(S)^\mathrm{aug}$, when restricted only to regular structures $(\mu_1,\dots,\mu_k)$ where the developing image of each $\mu_i$ is the Klein model of hyperbolic plane, gives the usual construction of $\Tmc(S)^\mathrm{aug}$. From this, it is clear that $\Tmc(S)^\mathrm{aug}$ naturally embeds in $\Cmc(S)^\mathrm{aug}$. The case of $\Tmc(S)^\mathrm{aug}$ is much simpler though, since the type of each puncture of $\mu_i$ is parabolic.

\subsection{The topology on the augmented deformation space.}\label{sec: aug top}

We will now define a topology on $\Cmc(S)^\mathrm{aug}$. To do so, it is convenient to introduce the following terminology.

\begin{definition}
Let $\mu\in\Cmc(S,\Dmc)^\mathrm{adm}$ and choose orientations on the closed curves $c\in\Dmc$. Let $S'$ be the connected component of $S\setminus\Dmc$ that lies to the left of $c$, and let $p$ be the puncture of $S'$ corresponding to $c$. Then the $\mu$-\emph{type} of $c$ is the $\mu$-type of $p$.
\end{definition}

Observe that the $\mu$-type of the oriented simple closed curve $c$ is parabolic, quasi-hyperbolic or bulge $\pm\infty$ if and only if the $\mu$-type of $c^{-1}$ is parabolic, quasi-hyperbolic or bulge $\mp\infty$ respectively.

The key to defining the topology on $\Cmc(S)^\mathrm{aug}$ are the pulling maps that were introduced in \cite{Loftin}. As before, let $\Dmc$ be any multi-curve in $S$ and let $S_1,\dots,S_k$ be the connected components of $S\setminus\Dmc$. By making appropriate choice of base points, the inclusion $S_i\subset S$ identifies $\pi_1(S_i)$ as a subgroup of $\pi_1(S)$. Observe that there is a homeomorphism $g_i:S_i\to\widetilde{S}/\pi_1(S_i)$ which is unique up to homotopy, and the induced map on fundamental groups $(g_i)_*:\pi_1(S_i)\to\pi_1(\widetilde{S}/\pi_1(S_i))$ is an isomorphism. Let $\widetilde{g}_i:\widetilde{S}_i\to\widetilde{S}$ denote the lift of $g_i$. For any $\nu\in\Cmc(S)$, if $(\phi,\rho)$ is a developing pair for $\nu$, then $\phi\circ\widetilde{g}_i:\widetilde{S}_i\to\mathbb{RP}^2$ is $\rho\circ (g_i)_*$-equivariant. Theorem \ref{thm:dev} then implies that $(\phi\circ\widetilde{g}_i,\rho\circ (g_i)_*)$ is the developing pair for some $\mu_i\in\Cmc(S_i)$. Since $g_i$ is unique up to homotopy, $\mu_i$ does not depend on the choice of $g_i$. Hence, this allows us to define the $\Dmc$-\emph{pulling map}
\[\mathrm{Pull}_\Dmc:\Cmc(S)\to\prod_{i=1}^k\Cmc(S_i)\]
by $\mathrm{Pull}_\Dmc(\mu)=(\mu_1,\dots,\mu_k)$. It is straighforward to verify that $\mathrm{Pull}_\Dmc$ is continuous. Also, it is important to emphasize that each $\mu_i$ here has a representative $\rho_i\in\mathrm{hol}(\mu_i)$ so that the $\rho_i$-equivariant developing image of $\mu_i$ agree for all $i$. Note that if $\Dmc$ is non-empty, then every $\mu_i$ is not admissible even if $\mu$ is admissible.

Next, if $\Dmc'\subset\Dmc$ are multi-curves  in $S$, let $S_1,\dots,S_k$ be the connected components of $S\setminus\Dmc$ and $S_1',\dots,S_{k'}'$ be the connected components of $S\setminus\Dmc'$. Note that for all $i=1,\dots,k$, there is some $j=1,\dots,k'$ so that $S_i\subset S_j'$. Let $\Dmc(S_i')$ denote the curves in $\Dmc$ that lie in $S_i'$ but are non-peripheral in $S_i'$. This allows us to define the $(\Dmc',\Dmc)$-\emph{pulling map}
\[\mathrm{Pull}_{\Dmc',\Dmc}:\prod_{i=1}^{k'}\Cmc(S_i')\to\prod_{i=1}^k\Cmc(S_i)\]
by $\mathrm{Pull}_{\Dmc',\Dmc}(\mu_1,\dots,\mu_{k'}):=\Big(\textrm{Pull}_{\Dmc(S_1')}(\mu_1),\dots,\mathrm{Pull}_{\Dmc(S_{k'}')}(\mu_{k'})\Big)$. Clearly, $\mathrm{Pull}_{\Dmc',\Dmc}$ is continuous.

Using the pulling maps, we can now define a basis for the topology on $\Cmc(S)^\mathrm{aug}$. For any open $U\subset\prod_{i=1}^k\Cmc(S_i)$ that intersects $\Cmc(S,\Dmc)^\mathrm{adm}\subset\prod_{i=1}^k\Cmc(S_i)$, let
\[\Umc(U,\Dmc):=\bigcup_{\Dmc'\subset\Dmc}\Bigg(\mathrm{Pull}_{\Dmc',\Dmc}^{-1}(U)\cap\Cmc(S,\Dmc')^\mathrm{adm}\Bigg).\]
Note that $\Umc(U,\Dmc)\subset\Cmc(S)^\mathrm{aug}$, and define
\[\Amc(S):=\left\{
\Umc(U,\Dmc):\begin{array}{l}\Dmc\text{ is a multi-curve in }S,\\
S_1,\dots,S_k\text{ are the connected components of }S\setminus\Dmc,\\
U\subset\prod_{i=1}^k\Cmc(S_i)$ is an open set that intersects $\Cmc(S,\Dmc)^\mathrm{adm}
\end{array}\right\}.\]
If $\Umc(U,\Dmc)$ and $\Umc(U',\Dmc')$ are sets in $\Amc(S)$, let $\Dmc''$ denote the maximal multi-curve that lies in both $\Dmc$ and $\Dmc'$ (it is possible that $\Dmc''$ is empty), and let $S_1'',\dots,S_{k''}''$ be the connected components of $S\setminus\Dmc''$. Then 
\[U'':=\mathrm{Pull}_{\Dmc'',\Dmc}^{-1}(U)\cap \mathrm{Pull}_{\Dmc'',\Dmc'}^{-1}(U')\subset\prod_{i=1}^{k''}\Cmc(S_i'')\]
is open, and $\Umc(U'',\Dmc'')\subset\Umc(U,\Dmc)\cap\Umc(U',\Dmc')$. This implies that $\Amc(S)$ is a basis on $\Cmc(S)^{\rm aug}$. Equip $\Cmc(S)^{\rm aug}$ with the topology generated by $\Amc(S)$. 

The topology on $\Cmc(S)^\mathrm{aug}$ as defined is rather abstract. The philosophical purpose of this paper of this paper is to understand this topology in a concrete way. We begin by observing that this topology has several important features that we will record as the following preliminary remarks.

\begin{remark}\label{rem: top}\
\begin{enumerate}
\item For any $\mu\in\Cmc(S)^\mathrm{aug}$, let $\Dmc$ be the multi-curve so that $\mu\in\Cmc(S,\Dmc)^\mathrm{adm}$, and let $S_1,\dots,S_k$ be the connected components of $S\setminus\Dmc$. Then
\[\Cmc(S)^{\mathrm{aug},\Dmc}:=\Umc\left(\prod_{i=1}^k\Cmc(S_i),\Dmc\right)=\bigcup_{\Dmc'\subset\Dmc}\Cmc(S,\Dmc')^\mathrm{adm}\]
is an open set in $\Cmc(S)^\mathrm{aug}$ that contains $\mu$. In particular, if a sequence $\{\mu^j\}_{j=1}^\infty\subset\Cmc(S)^\mathrm{aug}$ converges to $\mu$, then by removing finitely many points from this sequence, we may assume that $\{\mu^j\}_{j=1}^\infty\subset \Cmc(S)^{\mathrm{aug},\Dmc}$. Hence, for any $j\in\mathbb{Z}^+$, there is some $\Dmc^j\subset\Dmc$ so that $\mu^j\in\Cmc(S,\Dmc^j)$, so $\mu$ and $\mathrm{Pull}_{\Dmc^j,\Dmc}(\mu^j)$ are of the form
\[\mu=(\mu_1,\dots,\mu_k)\,\,\text{ and }\,\,\mathrm{Pull}_{\Dmc^j,\Dmc}(\mu^j)=(\mu^j_1,\dots,\mu^j_k)\]
for some $\mu_i,\mu^j_i\in\Cmc(S_i)$. From the definition of the topology on $\Cmc(S)^\mathrm{aug}$, one observes that $\lim_{j\to\infty}\mu^j=\mu$ in $\Cmc(S)^\mathrm{aug}$ if and only if $\lim_{j\to\infty}\mu^j_i=\mu_i$ in $\Cmc(S_i)$ for all $i=1,\dots,k$.
\item Let $\Dmc$ be a multi-curve on $S$ and $S_1,\dots,S_k$ be the connected components of $S\setminus\Dmc$. The holonomy map $\mathrm{hol}:\Cmc(S_j)\to\Xmc_3(S_j)$ extends to the map
\begin{eqnarray*}
\mathrm{hol}:\prod_{j=1}^k\Cmc(S_j)&\to&\prod_{i=1}^k\Xmc_3(S_j)\\
(\mu_1,\dots,\mu_k)&\mapsto&(\mathrm{hol}(\mu_1),\dots,\mathrm{hol}(\mu_k)).
\end{eqnarray*}
Restricting this to $\Cmc(S,\Dmc)^\mathrm{adm}$ defines a continuous map 
\[\mathrm{hol}:\Cmc(S,\Dmc)^\mathrm{adm}\to\prod_{i=1}^k\Xmc_3(S_j).\]
\item This topology on $\Cmc(S)^\mathrm{aug}$ is first countable. This was verified by the first author \cite{Loftin}.
\item From the definition of the usual topology on the augmented Teichm\"uller space $\Tmc(S)^\mathrm{aug}$ (see \cite{Abi}), it is easy to see that the natural inclusion of $\Tmc(S)^\mathrm{aug}$ into $\Cmc(S)^\mathrm{aug}$ as described above is a homeomorphism onto its image.
\item Abikoff \cite{Abi} observed that $\Tmc(S)^{\mathrm{aug}}$ is not locally compact at any point in $\Tmc(S)^\mathrm{aug}\setminus\Tmc(S)$. The same argument proves that $\Cmc(S)^\mathrm{aug}$ is not locally compact at any point in $\Cmc(S)^\mathrm{aug}\setminus\Cmc(S)$. More precisely, let $\Dmc$ be any non-empty multicurve in $S$, let $S_1,\dots,S_k$ be the connected components of $S\setminus\Dmc$, and let $\mu=(\mu_1,\dots,\mu_k)\in\Cmc(S,\Dmc)^\mathrm{adm}$. By the definition of the topology on $\Cmc(S)^\mathrm{aug}$, for every open set $V\subset\Cmc(S)^\mathrm{aug}$ containing $\mu$, there is some nonempty open set $U\subset\prod_{i=1}^k\Cmc(S_i)$ so that $\Umc(U,\Dmc)$ is non-empty and lies in $V$. Observe that the intersection of any fiber of the map $\mathrm{Pull}_{\Dmc',\Dmc}$ with $\Cmc(S,\Dmc')^\mathrm{adm}$ is either empty or has non-compact closure (because of the action of Dehn twists). This implies that the closure of $\Umc(U,\Dmc)$ in $\Cmc(S)^{\mathrm{aug}}$ is not compact, so the same holds for the closure of $V$ in $\Cmc(S)^{\mathrm{aug}}$. Hence, $\Cmc(S)^\mathrm{aug}$ is not locally compact at $\mu$.
\item The mapping class group
\[\mathrm{MCG}(S):=\mathrm{Diffeo(S)^+}/\mathrm{Diffeo(S)^+_0}\] acts naturally on the set of multi-curves on $S$. For any multi-curve $\Dmc$ on $S$ and any $[g]\in\MCG(S)$, we may define $\Dmc':=[g]\cdot\Dmc$. Let $S_1,\dots,S_k$ be the connected components of $S\setminus\Dmc$, and let $S_1',\dots,S_k'$ be the connected components of $S\setminus\Dmc'$. One can verify that the map
\[[g]:\prod_{i=1}^k\Cmc(S_i)\to\prod_{i=1}^k\Cmc(S_i')\]
defined by $\big([f_1,\Sigma_1],\dots,[f_k,\Sigma_k]\big)\mapsto\big([f_1\circ g,\Sigma_1],\dots,[f_k\circ g,\Sigma_k]\big)$ is a homeomorphism, so it restricts to a homeomorphism $[g]:\Cmc(S,\Dmc)^\mathrm{adm}\to\Cmc(S,\Dmc')^\mathrm{adm}$. This in turn defines a homeomorphism $[g]:\Cmc(S)^\mathrm{aug}\to\Cmc(S)^\mathrm{aug}$.
Thus, $\MCG(S)$ acts on $\Cmc(S)^\mathrm{aug}$ by homeomorphisms. The first author \cite{Loftin} showed that the quotient $\overline{\Mmc(S)}:=\Cmc(S)^\mathrm{aug}/\mathrm{MCG}(S)$ is topologically an orbifold, albeit with a complicated singular locus.
\end{enumerate}
\end{remark}

\section{Describing $\mathrm{hol}(\Cmc(S))$} \label{sec: hol}
There is a well-known coordinate system on $\mathrm{hol}(\Cmc(S))$ that was originally due to Goldman \cite{Go90}, and later modified by the second author \cite{Zhang16} using the work of Fock-Goncharov \cite{FocGon} and Bonahon-Dreyer \cite{BonDre}. These coordinates play a key role in our description of the topology on $\Cmc(S)^\mathrm{aug}$, so we will devote this section to carefully constructing them.

\subsection{Projective invariants}
We can think of $\Rbbb\Pbbb^2$ as the set of projective classes of vectors in $\Rbbb^3$. Similarly, $(\Rbbb\Pbbb^2)^*$ can be thought of either as the set of projective classes of linear functionals in $(\Rbbb^3)^*$, the set of projective lines in $\Rbbb\Pbbb^2$, or the set of projectivized planes through the origin in $\Rbbb^3$. Throughout the rest of this paper, we will assume these identifications without further comment.

Given any $L_1,L_2\in(\Rbbb\Pbbb^2)^*$ and $p_1,p_2\in\Rbbb\Pbbb^2$ so that $p_i\notin L_j$ for all $i,j=1,2$, we can define the \emph{cross ratio}
\[C(L_1,p_1,p_2,L_2):=\frac{L_1(p_2)L_2(p_1)}{L_1(p_1)L_2(p_2)}.\]
Here, we choose a linear functional representative for each $L_i$ and a vector representative for each $p_j$ to evaluate $L_i(p_j)$. One can verify from the definition of the cross ratio that the choice of representatives is irrelevant. Furthermore, the cross ratio is a projective invariant, and satisfies the symmetries
\[C(L_1,p_1,p_2,L_2)=C(L_2,p_2,p_1,L_1)=\frac{1}{C(L_1,p_2,p_1,L_2)}. \]

\begin{remark}\label{rem:C+}
Note that if $L_1=L_2$ or $p_1=p_2$ (or both), then $C(L_1,p_1,p_2,L_2)=1$. On the other hand, if $L_1\neq L_2$, then $\Rbbb\Pbbb^2\setminus(L_1\cup L_2)$ has two connected components, and $C(L_1,p_1,p_2,L_2)$ is positive (resp. negative) if and only if $p_1$ and $p_2$ lie in the same (resp. different) connected component of $\Rbbb\Pbbb^2\setminus(L_1\cup L_2)$.
\end{remark}

Similarly, given any triple $L_1,L_2,L_3\in(\Rbbb\Pbbb^2)^*$ and $p_1,p_2,p_3\in\Rbbb\Pbbb^2$ so that $p_i\notin L_{i-1}\cup L_{i+1}$ for all $i=1,2,3$ (arithmetic in the subscripts is done modulo 3), we can define the triple ratio
\[T(L_1,L_2,L_3,p_1,p_2,p_3):=\frac{L_1(p_2)L_2(p_3)L_3(p_1)}{L_1(p_3)L_3(p_2)L_2(p_1)}.\]
As before, we choose linear functional representatives of the $L_i$ and vector representatives of each $p_j$ to evaluate $L_i(p_j)$, and one can verify the independence of the triple ratio from these choices. The triple ratio is also a projective invariant satisfying the symmetries
\[T(L_1,L_2,L_3,p_1,p_2,p_3)=T(L_2,L_3,L_1,p_2,p_3,p_1)=\frac{1}{T(L_2,L_1,L_3,p_2,p_1,p_3)}.\]

\begin{remark}\label{rem:T+}
Note that if $L_1$, $L_2$ and $L_3$ do not intersect at a common point, then $\Rbbb\Pbbb^2\setminus(L_1\cup L_2\cup L_3)$ has four connected components, each of which is a triangle. Further suppose that $p_i\in L_i$ for $i=1,2,3$, i.e. $(p_i,L_i)$ is a flag (see Section \ref{sec: limit maps}). In this situation, it is straightforward to check that $T(L_1,L_2,L_3,p_1,p_2,p_3)$ is positive if and only if one of these four connected components contains all of $p_1$, $p_2$, $p_3$ in its boundary.
\end{remark}

\subsection{Ideal triangulations and pants decompositions} \label{sec: triangulation}

Next, we will describe a particular ideal triangulation on $S$ that one can associate to any pants decomposition of $S$. We begin by precisely defining the notion of an ideal triangulation on the topological surface $S$.

Since $S$ has negative Euler characteristic, it is well-known that $\pi_1(S)$ is a hyperbolic group, and its Gromov boundary $\partial\pi_1(S)$ has a natural cyclic order induced by the orientation on $S$. More concretely, if we choose a convex cocompact hyperbolic metric on $S$, then the universal cover $\widetilde{S}$ of $S$ can be identified with the Poincar\'e disc $\Dbbb$ as oriented Riemannian metric spaces. For any $p\in\widetilde{S}=\Dbbb$, the orbit map $\pi_1(S)\to\Dbbb$ defined by $\gamma\mapsto\gamma\cdot p$ is a quasi-isometric embedding, so it extends to an embedding of $\partial\pi_1(S)$ into $\partial\Dbbb^2$. The orientation on $\Dbbb\simeq\widetilde{S}$ then induces a counter-clockwise cyclic ordering on $\partial\Dbbb^2$, which restricts to a cyclic ordering on $\partial\pi_1(S)$. One can then verify that this cyclic ordering on $\partial\pi_1(S)$ does not depend on any of the choices made.

\begin{definition}
A \emph{geodesic} on $\widetilde{S}$ is an (unordered) pair of distinct points $\{x,y\}\subset\partial\pi_1(S)$, so that $\{x,y\}$ is not the set of fixed points of some peripheral $\gamma\in\pi_1(S)$.
\end{definition}

Denote the space of geodesics on $\widetilde{S}$ by $\Gmc(\widetilde{S})$, and note that the natural $\pi_1(S)$ action on $\partial\pi_1(S)$ induces a $\pi_1(S)$ action on $\Gmc(\widetilde{S})$. Also, we say that two geodesics $\{x,y\}$ and $\{x',y'\}$ \emph{intersect transversely} if $x<x'<y<y'<x$ or $x<y'<y<x'<x$ in $\partial\pi_1(S)$.

\begin{definition}
An \emph{ideal triangulation} $\widetilde{\Tmc}\subset{\Gmc}(\widetilde S)$ on $\widetilde{S}$ is then a maximal, $\pi_1(S)$-invariant collection of geodesics that pairwise do not intersect transversely, with the property that every $\{x,y\}\in\widetilde{\Tmc}$ satisfies one of the following:
\begin{itemize}
\item $\{x,y\}$ is the set of fixed points of some $\gamma\in\pi_1(S)$
\item there are points $z,w\in\partial\pi_1(S)$ so that $\{z,x\},\{z,y\},\{w,x\},\{w,y\}\in\widetilde{\Tmc}$.
\end{itemize}
If the former holds, then $\{x,y\}$ is a \emph{closed edge} of $\widetilde{\Tmc}$. On the other hand, if the latter holds, then $\{x,y\}$ is an \emph{isolated edge} of $\widetilde{\Tmc}$. Denote the set of closed edges in $\widetilde{\Tmc}$ by $\widetilde{\Pmc}=\widetilde{\Pmc}_{\widetilde{\Tmc}}$, and denote the set of isolated edges in $\widetilde{\Tmc}$ by $\widetilde{\Qmc}=\widetilde{\Qmc}_{\widetilde{\Tmc}}$. Every $\{x,y\}\in\widetilde{\Tmc}$ is an \emph{edge}, and $x,y$ are the \emph{vertices} of the edge $\{x,y\}$. 
\end{definition}

\begin{remark}
For every $\gamma\in\pi_1(S)$, let $\gamma^-$, $\gamma^+\in\partial\pi_1(S)$ denote the attracting and repelling fixed points of $\gamma$ respectively. It is important to emphasize that by our definition of geodesics, if $\gamma\in\pi_1(S)$ is a peripheral group element, then $\{\gamma^-,\gamma^+\}$ is not a geodesic, and hence not an edge in $\widetilde\Tmc$. We use this convention as it will be convenient for our purposes later.
\end{remark}

\begin{definition}
An \emph{ideal triangle} of the ideal triangulation $\widetilde{\Tmc}$ is a triple $\{x,y,z\}\subset\partial\pi_1(S)$ so that $\{x,y\},\{x,z\},\{y,z\}\in\widetilde{\Tmc}$. We will refer to $x,y,z$ as the \emph{vertices} of the ideal triangle $\{x,y,z\}$. Denote the set of ideal triangles of $\widetilde{\Tmc}$ by $\widetilde{\Theta}_{\widetilde{\Tmc}}=\widetilde{\Theta}$, and let $\widetilde{\Vmc}=\widetilde{\Vmc}_{\widetilde{\Tmc}}\subset\partial\pi_1(S)$ denote the set of vertices of the ideal triangles in $\widetilde{\Theta}$.
\end{definition}

We can then define an \emph{ideal triangulation} of $S$ to be the quotient $\Tmc:=\widetilde{\Tmc}/\pi_1(S)$ of some ideal triangulation of $\widetilde{\Tmc}$ of $\widetilde{S}$, and let $[x,y]$ denote the element in $\Tmc$ with $\{x,y\}$ as a representative. If one chooses a convex cocompact hyperbolic metric $\Sigma$ on $S$, then every ideal triangulation on $S$ is realized geometrically as an ideal triangulation of the convex core of the hyperbolic surface $\Sigma$ in the classical sense. Denote $\Pmc:=\widetilde{\Pmc}/\pi_1(S)$, $\Qmc:=\widetilde{\Qmc}/\pi_1(S)$, $\Theta:=\widetilde{\Theta}/\pi_1(S)$ and $\Vmc:=\widetilde{\Vmc}/\pi_1(S)$. Let $[x,y,z]$ denote the element in $\Theta$ with representative $\{x,y,z\}\in\widetilde\Theta$, and let $[x]$ denote the element in $\Vmc$ with representative $x\in\widetilde{\Vmc}$. We will refer to the elements in $\Tmc$, $\Pmc$, $\Qmc$, $\Theta$, $\Vmc$ respectively as edges, closed edges, isolated edges, ideal triangles, and vertices of the ideal triangulation $\Tmc$ of $S$. 

Next, we specialize to particular ideal triangulations coming from pants decompositions of $S$, i.e.\ maximal multi-curves. For every pants decomposition $\Pmc$ of $S$, observe that each connected component of $S\setminus\Pmc$ is a pair of pants. Denote this collection of pairs of pants by $\Pbbb$. For each $P\in\Pbbb$, let $\alpha_P,\beta_P,\gamma_P\in\pi_1(S)$ be the group elements satisfying the following:
\begin{itemize}
\item $\alpha_P\beta_P\gamma_P=\id$.
\item $[\alpha_P]$, $[\beta_P]$, $[\gamma_P]$ correspond to the three boundary components of $P$, oriented so that $P$ lies to the left of the boundary component.
\end{itemize}
Then let
\begin{align*}
\Tmc_\Pmc=\Tmc:=&\{[\alpha_P^-,\beta_P^-]:P\in\Pbbb\}\cup\{[\beta_P^-,\gamma_P^-]:P\in\Pbbb\}\cup\{[\gamma_P^-,\alpha_P^-]:P\in\Pbbb\}\\
&\cup\{[\gamma^-,\gamma^+]:\gamma\in\pi_1(S)\text{ corresponds to some closed curve in }\Pmc\}.
\end{align*}
One may verify that $\Tmc$ is an ideal triangulation. The pants decomposition $\Pmc$ is naturally identified with
\[\{[\gamma^-,\gamma^+]:\gamma\in\pi_1(S)\text{ corresponds to some closed curve in }\Pmc\},\]
which is the set of closed edges in $\Tmc$, and
\[\Qmc:=\{[\alpha_P^-,\beta_P^-]:P\in\Pbbb\}\cup\{[\beta_P^-,\gamma_P^-]:P\in\Pbbb\}\cup\{[\gamma_P^-,\alpha_P^-]:P\in\Pbbb\}\] 
is the set of isolated edges in $\Tmc$. Also, the set of ideal triangles of $\Tmc$ is
\[\Theta:=\{[\alpha_P^-,\beta_P^-,\gamma_P^-]:P\in\Pbbb\}\cup\{[\beta_P^-,\beta_P\cdot \alpha_P^-,\gamma_P^-]:P\in\Pbbb\},\]
and the set of vertices of $\Tmc$ is
\[\Vmc:=\{[\alpha_P^-]:P\in\Pbbb\}\cup\{[\beta_P^-]:P\in\Pbbb\}\cup\{[\gamma_P^-]:P\in\Pbbb\}\]
Note that from the way we defined $\Tmc$, every vertex of every closed edge in $\widetilde{\Tmc}$ is a vertex of an ideal triangle of $\widetilde{\Tmc}$, see Figure \ref{ideal-triangulation-figure}. 
\begin{figure}
\begin{center}
\includegraphics{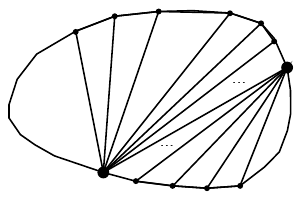}
\small
\put (-113, 82){$w_1$}
\put (-92, 89){$w_2$}
\put (-70, 91){$w_3$}
\put (-38, 90){$w_4$}
\put (-22, 86){$w_5$}
\put (-10, 73){$w_6$}
\put (-1, 61){$x$}
\put (-31, -2){$z_1$}
\put (-48, -3){$z_2$}
\put (-64, -2){$z_3$}
\put (-83, 0){$z_5$}
\put (-100, 3){$y$}
\end{center}
\caption{Ideal Triangulation: $[x,y]\in\Pmc,\,\,[x,z_i],[y,w_j]\in\Qmc$}
\label{ideal-triangulation-figure}
\end{figure}

\subsection{Flag maps} \label{sec: limit maps}
Let $\mathcal F$ be the space of flags in $\rp^2$, i.e. $\mathcal F = \{(v,\ell)\in \rp^2 \times (\rp^2)^* : v\in\ell\}$. Also, let $\mu\in\Cmc(S)$ and let $\rho:\pi_1(S)\to\PGL(3,\Rbbb)$ be a representative of the conjugacy class $\mathrm{hol}(\mu)$. Let $\Vmc$ denote the set of vertices of the ideal triangulation $\Tmc$ whose closed edges are a pants decomposition of $S$ as defined above, and let $\widetilde\Vmc\subset\partial\pi_1(S)$ denote the lift of $\Tmc$.

Observe that every point $x\in\widetilde{\Vmc}$ is a repelling fixed point $\gamma^-$ of some $\gamma\in\pi_1(S)\setminus\{\id\}$. This allows us to construct a \emph{flag map} $\xi_\rho:\widetilde{\Vmc}\to \mathcal F$ in the following way. By (1) of Theorem \ref{thm: holonomy1}, we see that one of the following holds for $\rho(\gamma)$. 
\begin{enumerate}
\item $\rho(\gamma)$ has exactly three fixed points in $\Rbbb\Pbbb^2$, one of which is attracting and another is repelling. The third fixed point that is neither attracting nor repelling is called the \emph{saddle fixed point}.
\item $\rho(\gamma)$ has exactly two fixed points in $\Rbbb\Pbbb^2$, one of which is repelling. We will refer to the fixed point that is not repelling as the \emph{quasi-attracting fixed point}. $\rho(\gamma)$ also stabilizes a unique line that contains the quasi-attracting fixed point, but not the repelling fixed point.
\item $\rho(\gamma)$ has exactly two fixed points in $\Rbbb\Pbbb^2$, one of which is attracting. We will refer to the fixed point that is not attracting as the \emph{quasi-repelling fixed point}. $\rho(\gamma)$ also stabilizes a unique line that contains the quasi-repelling fixed point, but not the attracting fixed point.
\item $\rho(\gamma)$ has a unique fixed point in $\Rbbb\Pbbb^2$, and stabilizes a unique line through that fixed point.
\end{enumerate}
(1) holds when $\rho(\gamma)$ is hyperbolic, (2) or (3) holds when $\rho(\gamma)$ is quasi-hyperbolic, and (4) holds when $\rho(\gamma)$ is parabolic. Theorem \ref{thm: holonomy1} implies that (2), (3) and (4) can happen only when $\gamma$ is a peripheral element.

Using this, define $\xi_\rho(x):=\left(\xi_\rho^{(1)}(x),\xi_\rho^{(2)}(x)\right)\in\mathcal F\subset \Rbbb\Pbbb^2\times(\Rbbb\Pbbb^2)^*$ as follows.
\begin{itemize}
\item[(I)] If (1) holds, define $\xi_\rho^{(1)}(x)$ to be the repelling fixed point of $\rho(\gamma)$, and $\xi_\rho^{(2)}(x)$ to be the projective line containing the repelling and saddle fixed points of $\rho(\gamma)$.  See Figure \ref{hyperbolic-flag-figure}. 
\item[(II)] If (2) holds, define $\xi_\rho^{(1)}(x)$ to be the repelling fixed point of $\rho(\gamma)$, and $\xi_\rho^{(2)}(x)$ to be the projective line containing both fixed points of $\rho(\gamma)$. See Figure  \ref{quasi-hyp-flag1-figure}. 
\item[(III)] If (3) holds, define $\xi_\rho^{(1)}(x)$ to be the quasi-repelling fixed point of $\rho(\gamma)$, and $\xi_\rho^{(2)}(x)$ to be projective line stabilized by $\rho(\gamma)$ that contains its quasi-repelling fixed point but not its attracting fixed point. See Figure \ref{quasi-hyp-flag2-figure}. 
\item[(IV)] If (4) holds, define $\xi_\rho^{(1)}(x)$ to be the unique fixed point of $\rho(\gamma)$, and define $\xi_\rho^{(2)}(x)$ to be the unique projective line stabilized by $\rho(\gamma)$. See Figure \ref{parabolic-flag-figure}. 
\end{itemize}

\begin{figure}
\begin{center}
\includegraphics{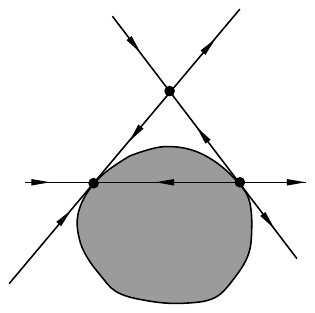}
\LARGE
\put (-70,30){$\Omega$}
\small
\put (-32, 64){$\xi_\rho^{(1)}(x)$}
\put (-4, 20){$\xi_\rho^{(2)}(x)$}
\end{center}
\caption{Flag for hyperbolic case}
\label{hyperbolic-flag-figure}
\end{figure}

\begin{figure}
\begin{center}
\includegraphics{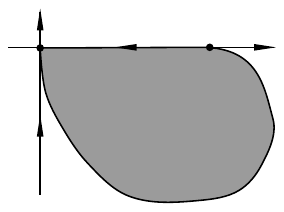}
\LARGE
\put (-55,35){$\Omega$}
\small
\put (-37, 83){$\xi_\rho^{(1)}(x)$}
\put (-158, 75){$\xi_\rho^{(2)}(x)$}
\end{center}
\caption{Flag for quasi-attracting case}
\label{quasi-hyp-flag1-figure}
\end{figure}

\begin{figure}
\begin{center}
\includegraphics{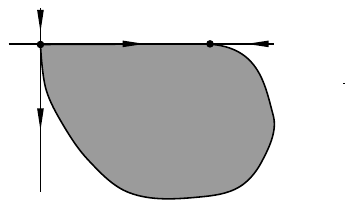}
\LARGE
\put (-55,35){$\Omega$}
\small
\put (-112, 82){$\xi_\rho^{(1)}(x)$}
\put (-142, 30){$\xi_\rho^{(2)}(x)$}
\end{center}
\caption{Flag for quasi-repelling case}
\label{quasi-hyp-flag2-figure}
\end{figure}

\begin{figure}
\begin{center}
\includegraphics{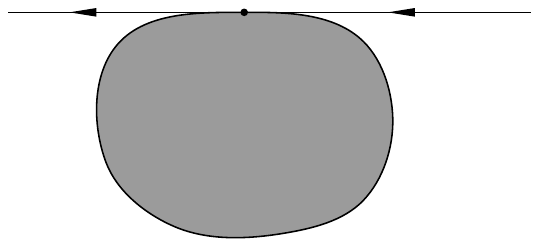}
\LARGE
\put (-90,60){$\Omega$}
\small
\put (-100, 117){$\xi_\rho^{(1)}(x)$}
\put (-205, 110){$\xi_\rho^{(2)}(x)$}
\end{center}
\caption{Flag for parabolic case}
\label{parabolic-flag-figure}
\end{figure}

\begin{remark}
Observe that the flag map can actually be defined on all fixed points of all non-identity elements in $\partial\pi_1(S)$. However, we will only consider the map restricted to $\widetilde{\Vmc}$.
\end{remark}

It is easy to verify that the flag map $\xi_\rho$ is $\rho$-equivariant. Also, for any $x\in\widetilde{\Vmc}$, $\xi^{(1)}_\rho(x)$ is either a repelling fixed point, a quasi-repelling fixed point, or the unique fixed point of $\rho(\gamma)$ for some $\gamma\in\pi_1(S)$. It follows that $\xi^{(1)}_\rho(\widetilde{\Vmc})$ lies in $\partial\Omega$, where $\Omega$ is the $\rho$-equivariant developing image of $\mu$. On the other hand, $\xi^{(2)}_\rho(\widetilde{\Vmc})$ satisfy the following proposition.

\begin{prop}\label{prop:transverse1}
Let $\Omega$ be the $\rho$-equivariant developing image of $\mu$. Then for all $x\in\widetilde{\Vmc}$, $\xi_\rho^{(2)}(x)$ does not intersect $\Omega$.
\end{prop}

\begin{proof}
Since $\Omega$ is properly convex, there is an affine chart $\mathbb{A}$ containing $\overline{\Omega}$. Fix $x\in\Vmc$, and let $\gamma\in\pi_1(S)\setminus\{\id\}$ be an element whose repelling fixed point is $x$. We will consider the four cases (I) -- (IV) separately. 

Suppose that (III) or (IV) holds. Then $\xi_\rho^{(1)}(x)$ is the unique fixed point of $\rho(\gamma)$ in $\xi_\rho^{(2)}(x)$. Since $\Omega$ is open, if $\xi_\rho^{(2)}(x)$ intersects $\Omega$, then this intersection contains a point that is not $\xi_\rho^{(1)}(x)$. The convexity of $\Omega$ then implies that $\xi_\rho^{(2)}(x)\cap\mathbb{A}$ lies in $\Omega$, which contradicts the fact that $\overline{\Omega}$ lies in $\mathbb{A}$. Hence, $\xi_\rho^{(2)}(x)$ does not intersect $\Omega$.

Suppose that (I) or (II) holds. In this case, $\rho(\gamma)$ has two fixed points in $\xi_\rho^{(2)}(x)$, one of which is $\xi_\rho^{(1)}(x)$. Suppose for contradiction that $\Omega$ intersects $\xi_\rho^{(2)}(x)$. Then there are some $p,q\in\partial\Omega$ that lie in the two different components of $\mathbb{A}\setminus\xi_\rho^{(2)}(x)$. Observe that $\lim_{n\to\infty}\rho(\gamma^{-n})\cdot p=\xi_\rho^{(1)}(x)$. Also, if $r_n$ denotes the ray in $\mathbb{A}$ that passes through $\rho(\gamma^{-n})\cdot p$ and has source $\xi_\rho^{(1)}(x)$, then $r_n$ lies in the component of $\mathbb{A}\setminus\xi_\rho^{(2)}(x)$ that contains $p$ for all $n\in\mathbb{Z}$, and $\lim_{n\to\infty}r_n$ is a ray in $\xi_\rho^{(2)}(x)$ with source $\xi_\rho^{(1)}(x)$. From this, it follows that for sufficiently large $n\in\mathbb{Z}^+$, the triangle in $\mathbb{A}$ with vertices $q$, $\rho(\gamma)^{-n}\cdot p$, and $\xi_\rho^{(1)}(x)$ contains a segment of $\partial\Omega$ with endpoints $\rho(\gamma)^{-n}\cdot p$ and $\xi_\rho^{(1)}(x)$. However, the convexity of $\Omega$ implies that this triangle also lies in $\Omega$, which is impossible.
\end{proof}

Next, we prove that the map $\xi_\rho$ is transverse in the following sense.

\begin{prop}\label{prop:transverse}
Suppose that $x,y\in\widetilde{\Vmc}$ are distinct. Then
\begin{enumerate}
\item $\xi^{(1)}_\rho(x)\neq\xi^{(1)}_\rho(y)$. 
\item $\xi^{(1)}_\rho(x)$ does not lie in $\xi^{(2)}_\rho(y)$. In particular, $\xi^{(2)}_\rho(x)\neq \xi^{(2)}_\rho(y)$.
\end{enumerate}
\end{prop}

\begin{remark}
Proposition \ref{prop:transverse} is false if we replace $\widetilde{\Vmc}$ with $\partial\pi_1(S)$. This is the main reason why we excluded peripheral elements in our definition of geodesics.
\end{remark}

If $(\phi,\rho)$ is a developing pair for $\mu$, then the orientation on $S$ induces an orientation on $\widetilde{S}$, which induces an orientation on $\Omega$ via $\phi$. This orientation on $\Omega$ does not depend on the choice $(\phi,\rho)$, so it defines a counter-clockwise cyclic ordering on $\partial\Omega$. The following is a preliminary lemma to prove Proposition \ref{prop:transverse}.

\begin{lem}\label{lem:op}
If $(a,b,c)$ is a pairwise distinct triple of points in $\widetilde{\Vmc}$ so that $a<b<c<a$ in the cyclic ordering on $\widetilde{\Vmc}$, then $\xi^{(1)}_\rho(a)\leq \xi^{(1)}_\rho(b)\leq \xi^{(1)}_\rho(c)\leq \xi^{(1)}_\rho(a)$ in that cyclic order along $\partial\Omega$.
\end{lem}

\begin{remark}\label{rem: op}
(1) of Proposition \ref{prop:transverse} in fact implies that all the inequalities in Lemma \ref{lem:op} are strict.
\end{remark}

\begin{proof}[Proof of Lemma \ref{lem:op}]
Observe that if $\xi^{(1)}_\rho(a)=\xi^{(1)}_\rho(b)$ or $\xi^{(1)}_\rho(b)=\xi^{(1)}_\rho(c)$ or $\xi^{(1)}_\rho(c)=\xi^{(1)}_\rho(a)$, then the proposition holds trivially. Hence, we only need to consider the case when $\xi^{(1)}_\rho(a)$, $\xi^{(1)}_\rho(b)$, and $\xi^{(1)}_\rho(c)$ are pairwise distinct.

Choose a convex cocompact hyperbolic metric on $S$, then $\widetilde{S}\simeq\Dbbb$ as oriented Riemannian metric spaces and $\widetilde{\Vmc}$ is a subset of $\partial\Dbbb$. Let $\gamma_a$, $\gamma_b$ and $\gamma_c\in\pi_1(S)$ be group elements whose attracting fixed points are $a,b,c$ respectively. For any $p\in\widetilde{S}$ and any $\gamma\in\pi_1(S)$, let $l_{p,\gamma\cdot p}$ be the closed line segment between $p$ and $\gamma\cdot p$. Then define
\[L_{a,p}:=\bigcup_{i=1}^\infty l_{\gamma_a^{i-1}\cdot p,\gamma_a^i\cdot p},\,\,\,\,L_{b,p}:=\bigcup_{i=1}^\infty l_{\gamma_b^{i-1}\cdot p,\gamma_b^i\cdot p},\,\,\,\,L_{c,p}:=\bigcup_{i=1}^\infty l_{\gamma_c^{i-1}\cdot p,\gamma_c^i\cdot p}.\]
Observe that $L_{a,p}$, $L_{b,p}$ and $L_{c,p}$ are simple curves starting at $p$ and going towards $a,b,c$ respectively. By choosing $p$ appropriately, we can further ensure that $L_{a,p}$, $L_{b,p}$ and $L_{c,p}$ are pairwise non-intersecting. Since $a<b<c<a$ in $\widetilde{\Vmc}$, this ensures that if we take a small disc centered at $p$, then the boundary of this disc, when oriented counter-clockwise, intersects $L_{a,p}$, $L_{b,p}$ and $L_{c,p}$ in that order.

Suppose for contradiction that the proposition is false. Since $\xi^{(1)}_\rho(a)$, $\xi^{(1)}_\rho(b)$, and $\xi^{(1)}_\rho(c)$ are pairwise distinct, this implies that $\xi^{(1)}_\rho(a)<\xi^{(1)}_\rho(c)<\xi^{(1)}_\rho(b)<\xi^{(1)}_\rho(a)$ in this cyclic order around $\partial\Omega$. Let $\phi$ be the $\rho$-equivariant developing map for $\mu$, and recall that the orientation on $\Omega$ was chosen so that $\phi$ is orientation preserving. The $\rho$-equivariance of $\phi$ ensures that then $\phi(L_{a,p})$, $\phi(L_{b,p})$ and $\phi(L_{c,p})$ are three pairwise non-intersecting curves starting at $\phi(p)$ at going towards $\xi^{(1)}_\rho(a)$, $\xi^{(1)}_\rho(b)$ and $\xi^{(1)}_\rho(c)$ respectively. But this means that if we take a small disc centered at $\phi(p)$, then the boundary of this disc, when oriented counter-clockwise, intersects $\phi(L_{a,p})$, $\phi(L_{c,p})$ and $\phi(L_{b,p})$ in that order. This contradicts the assumption that $\phi$ is orientation preserving.
\end{proof}

\begin{proof}[Proof of Proposition \ref{prop:transverse}]
By the definition of $\widetilde{\Vmc}$, there are some $z,w\in\widetilde{\Vmc}$ so that $x<z<y<w<x$ in this order along $\partial\pi_1(S)$. Let $\gamma\in\pi_1(S)$ be a non-peripheral element, and let $\gamma_z,\gamma_w\in\pi_1(S)$ be elements so that $\gamma_z^-=z$ and $\gamma_w^-=w$. Then for sufficiently large $n$,
\[x<\gamma_z^{-n}\cdot \gamma^\pm<y<\gamma_w^{-n}\cdot \gamma^\pm<x.\]
Also, let $\Omega$ be the $\rho$-equivariant developing image of $\mu$. By Lemma \ref{lem:op}, we see that
\begin{equation}\label{eqn: chain}
\xi^{(1)}_\rho(x)\leq \rho(\gamma_z)^{-n}\cdot \xi^{(1)}_\rho(\gamma^\pm)\leq \xi^{(1)}_\rho(y)\leq\rho(\gamma_w)^{-n}\cdot  \xi^{(1)}_\rho(\gamma^\pm)\leq \xi^{(1)}_\rho(x).
\end{equation}

\emph {(1)} Suppose for contradiction that $\xi_\rho^{(1)}(x)=\xi_\rho^{(1)}(y)$. Then \eqref{eqn: chain} implies that either $\xi_\rho^{(1)}(x)=\rho(\gamma_w)^{-n}\cdot\rho(\gamma)^\pm=\xi_\rho^{(1)}(y)$ or $\xi_\rho^{(1)}(x)=\rho(\gamma_z)^{-n}\cdot\rho(\gamma)^\pm=\xi_\rho^{(1)}(y)$. In either case, we have that $\xi^{(1)}_\rho(\gamma^+)=\xi^{(1)}_\rho(\gamma^-)$, which is impossible.

\emph{(2)} Suppose for contradiction that $\xi_\rho^{(1)}(x)\in\xi_\rho^{(2)}(y)$ for some distinct $x,y\in\widetilde{\Vmc}$. By Proposition \ref{prop:transverse1} and the convexity of $\Omega$, one deduces that there is an open line segment $L$ in $\Rbbb\Pbbb^2$ with endpoints $\xi_\rho^{(1)}(x)$ and $\xi_\rho^{(1)}(y)$, so that $L\subset\partial\Omega$ and either $\rho(\gamma_z)^{-n}\cdot\rho(\gamma)^\pm\in L$ or $\rho(\gamma_w)^{-n}\cdot\rho(\gamma)^\pm\in L$. Assume without loss of generality that the former holds. Then observe that $\lim_{k\to\infty}\rho(\gamma_z^{-n}\cdot\gamma\cdot\gamma_z^n)^k\cdot L$ is the entire projective line in $\Rbbb\Pbbb^2$ containing $L$. Since $\partial\Omega$ contains $L$ and is invariant under $\rho(\gamma_z^{-n}\cdot\gamma\cdot\gamma_z^n)$, we deduce that $\partial\Omega$ is a projective line, but this contradicts the properness of $\Omega$.
\end{proof}

If we choose a different representative $\rho'=g\cdot\rho\cdot g^{-1}\in\mathrm{hol}(\mu)$, then $\xi_{\rho'}=g\cdot \xi_\rho$. Furthermore, if $\mathrm{hol}(\mu)=\mathrm{hol}(\mu')$, then the $\PGL(3,\Rbbb)$-orbit of flag maps associated to $\mu$ and $\mu'$ agree. Hence, the map $\rho\mapsto\xi_\rho$ associates to the conjugacy class $\mathrm{hol}(\mu)$ a $\PGL(3,\Rbbb)$-orbit of maps from $\widetilde{\Vmc}\to\Rbbb\Pbbb^2\times(\Rbbb\Pbbb^2)^*$, which we denote by $\xi_{\mathrm{hol}(\mu)}:=[\xi_\rho]$. The next proposition tells us that $\xi_{\mathrm{hol}(\mu)}$ varies continuously with $\mathrm{hol}(\mu)$.

\begin{prop}\label{prop:conv-flag}
Let $F:(-\epsilon,\epsilon)\to\Xmc_3(S)$ given by $F:t\mapsto[\rho_t]$ be a map whose image lies in $\mathrm{hol}(\Cmc(S))$. Then $F$ is continuous if and only if $t\mapsto\xi_{\rho_t}(x)$ is a continuous path in $\Fmc$ for every $x\in\widetilde{\Vmc}$.
\end{prop}

This proposition is a consequence of the following elementary fact.

\begin{lem}  \label{conv-flag}
Let $L_i\to L$ be a convergent sequence of endomorphisms of $\Rbbb^n$, and let $\lambda^j_i,\lambda^j$ be the (generalized) eigenvalues of $L_i$ and $L$ respectively. Assume $\lambda^j_i,\lambda^j$ are real for all $j=1,\dots,n$. If there is some $m\in\{1,\dots,n\}$ so that for all $i$,
\[\dim \ker \prod_{j=1}^m (L_i-\lambda_i^jI) = m= \dim \ker \prod_{j=1}^m (L-\lambda^jI),\]
where $I$ is the identity endomorphism on $\Rbbb^n$, then
\[\ker \prod_{j=1}^m (L_i-\lambda_i^jI) \to \ker \prod_{j=1}^m (L-\lambda^jI)\]
in the Grassmannian $\Gr_m(\Rbbb^n)$.
\end{lem}
\begin{proof}
Let $x_i^1,\dots, x_i^m$ be an orthonormal basis of $\ker \prod_{j=1}^m (L_i-\lambda_i^jI)$.  By taking subsequences, we may assume that $x_i^j$ converges to $x^j$. Since $L_i\to L$, we see that $\lambda^j_i\to\lambda^j$, so $x^1,\dots, x^m$ all lie in $\ker \prod_{j=1}^m (L-\lambda^jI)$. By the dimension hypothesis, $x^1,\dots, x^m$ is an orthonormal basis of $\ker \prod_{j=1}^m (L-\lambda^jI)$. Hence, up to taking subsequences, $\ker \prod_{j=1}^m (L_i-\lambda_i^jI)$ converges to $\ker \prod_{j=1}^m (L-\lambda^jI)$. Repeating this argument for all subsequences of $\{L_i\}_{i=1}^\infty$ proves the lemma.
\end{proof}

\begin{proof}[Proof of Proposition \ref{prop:conv-flag}]
For any $x\in\widetilde{\Vmc}$, let $\gamma\in\pi_1(S)$ be the unique primitive group element so that $x$ is the repelling fixed point of $\gamma$. For any $\mu\in\Cmc(S)$ and any representative $\rho\in\mathrm{hol}(\mu)$, let $L\in\SL(3,\Rbbb)$ be a representative of $\rho(\gamma)$. Then let $\lambda_1\ge \lambda_2\ge \lambda_3$ be the generalized eigenvalues of $L$. Observe that as defined, $\xi_\rho^{(1)}(x)=\ker(L-\lambda_3)$ and $\xi_\rho^{(2)}(x)=\ker\big((L-\lambda_3)(L-\lambda_2)\big)$. To prove the forward direction, we simply apply Lemma \ref{conv-flag}.

For the backward direction, pick any $\gamma\in\pi_1(S)$, any triple of pairwise distinct points $x,y,z\in\widetilde{\Vmc}$, and let $(x',y',z'):=\gamma\cdot (x,y,z)$. Also, let
\[(a,b,c,d):=(\xi_\rho^{(1)}(x),\xi_\rho^{(1)}(y),\xi_\rho^{(1)}(z),\xi_\rho^{(2)}(x)\cap\xi_\rho^{(2)}(y)),\]
and let
\[(a',b',c',d'):=(\xi_\rho^{(1)}(x'),\xi_\rho^{(1)}(y'),\xi_\rho^{(1)}(z'),\xi_\rho^{(2)}(x')\cap\xi_\rho^{(2)}(y')).\]
The $\rho$-equivariance of $\xi$ implies that $\rho(\gamma)\cdot (a,b,c,d)=(a',b',c',d')$.

We will now argue that the quadruple of points $a,b,c,d$ are in general position, i.e. no three of them lie in a line in $\Rbbb\Pbbb^2$. By (2) of Proposition \ref{prop:transverse}, we see that the triple $a,b,d$ do not lie in a line in $\Rbbb\Pbbb^2$. For the same reason, the same is true for the triples $a,c,d$ and $b,c,d$. On the other hand, suppose for contradition that $a,b,c$ lie in a line. By the same argument as the first part of the proof of Proposition \ref{prop:transverse}, there is some $\eta\in\pi_1(S)$ so that $x<\eta^-<\eta^+<y<z$. The convexity of $\Omega$ then implies that there is a projective (open) line segment $L$ with endpoints $a,b$ so that $L\subset\partial\Omega$ and  $\rho(\eta)^-,\rho(\eta)^+\in L$. Since $\partial\Omega$ is $\rho$-equivariant, this means that
\[\bigcup_{i=-\infty}^\infty \rho(\eta)^i\cdot L\subset\partial\Omega,\]
and is an entire projective line in $\Rbbb\Pbbb^2$. This violates the properness of $\Omega$, so $a,b,c$ cannot lie in a line. We have thus proven that $a,b,c,d$ are in general position, so $(a',b',c',d')=\rho(\gamma)\cdot (a,b,c,d)$ is also in general position.

If we normalize $a=[1:0:0]^T$, $b=[0:1:0]^T$, $c=[0:0:1]^T$ and $d=[1:1:1]^T$, then it is a straightforward exercise to explicitly write down a matrix representative for $\rho(\gamma)$ in terms of the coordinates of $a',b',c',d'$. From this, it is clear that $\rho(\gamma)$ varies continuously with $a',b',c',d'$.
\end{proof}

\subsection{The holonomy map.}
It will be important later that we understand the image of the map $\mathrm{hol}:\Cmc(S)^\mathrm{adm}\to\Xmc_3(S)$. To do so, we set up the following notation. For every $\mu\in\Cmc(S)^\mathrm{adm}$, let $\Amc_{\mathrm{hol}(\mu)}$ denote the set of punctures of $S$ whose $\mu$-type is bulge $\pm\infty$, and let $\Gamma_{\mathrm{hol}(\mu)}\subset\pi_1(S)$ be the set of peripheral group elements corresponding to $\Amc_{\mathrm{hol}(\mu)}$.

For any representative $\rho\in\mathrm{hol}(\mu)$, let $\Omega$ denote the $\rho$-equivariant developing image of $\mu$. Let $\gamma\in\Gamma_{\mathrm{hol}(\mu)}$ and let $\rho(\gamma)^+=\xi^{(1)}(\gamma^+)$, $\rho(\gamma)^0$, $\rho(\gamma)^-=\xi^{(1)}(\gamma^-)$ be the attracting, saddle and repelling points of $\rho(\gamma)$. Note that $\rho(\gamma)^+$ and $\rho(\gamma)^-$ necessarily lie on the boundary of $\Omega$. Also, by Proposition \ref{prop:transverse1} and (1) of Proposition \ref{prop:transverse}, we see that $\xi^{(1)}(\widetilde{\Vmc}\setminus\{\gamma^+,\gamma^-\})$ lies entirely in one of the two connected components of $\Rbbb\Pbbb^2\setminus(\xi^{(2)}_\rho(\gamma^+)\cup\xi^{(2)}_\rho(\gamma^-))$, call it $A$. The projective line through $\rho(\gamma)^+$ and $\rho(\gamma)^-$ cuts $A$ into two open triangles, and the fact that $\gamma$ is a peripheral group element ensures that $\xi^{(1)}(\widetilde{\Vmc}\setminus\{\gamma^+,\gamma^-\})$ lies entirely in one of these two open triangles, call it $\Delta'$. 

\begin{definition}
Let $\rho\in\mathrm{hol}(\mu)$ and $\gamma\in\Gamma_{\mathrm{hol}(\mu)}$. The \emph{principal triangle} of $\rho(\gamma)$ is the open triangle $\Delta=\Delta_\rho$ that is the connected component of $A\setminus\Span(\rho(\gamma)^+,\rho(\gamma)^-)$ that is not $\Delta'$.
\end{definition}

It is clear that $\rho(\gamma)$ has a unique principal triangle, which depends only on $\rho$. Let $\Gmc_\rho$ denote the set of principal triangles of the group elements in $\rho(\Gamma_{\mathrm{hol}(\mu)})$. Observe that there is a natural $\pi_1(S)$-action on $\Gmc_\rho$ induced by $\rho$. The next theorem tells us to what extent different points in $\Cmc(S,\Dmc)^{\mathrm{adm}}$ can have the same holonomy.

\begin{thm}\label{thm: hol}
Let $\Dmc$ be any multi-curve in $S$, let $S_1,\dots,S_k$ be the connected components of $S\setminus\Dmc$, and let $(\mu_1,\dots,\mu_k),(\mu_1',\dots,\mu_k')\in\Cmc(S,\Dmc)^\mathrm{adm}$ so that
\[\mathrm{hol}(\mu_1,\dots,\mu_k)=\mathrm{hol}(\mu_1',\dots,\mu_k').\]
For all $i=1,\dots,k$, let $\rho_i$ be representatives of $\mathrm{hol}(\mu_i)$, and let $\Omega_i, \Omega_i'\subset\Rbbb\Pbbb^2$ be the $\rho_i$-equivariant developing images of $\mu_i,\mu_i'\in\Cmc(S_i)^{\mathrm{adm}}$ respectively. Then the interior of the symmetric difference $\Omega_i\, \triangle\, \Omega_i'$ is the union of a $\pi_1(S_i)$-invariant subset of triangles in $\Gmc_{\rho_i}$.
\end{thm}

The proof of Theorem \ref{thm: hol} is in the Appendix. As an immediate consequence of Theorem \ref{thm: hol}, Proposition \ref{prop:devimage}, and the compatibility of $\mu$ across the closed curves in $\Dmc$, we have the following corollary. 

\begin{cor}\label{cor: hol}
Let $\mathrm{hol}:\Cmc(S,\Dmc)^\mathrm{adm}\to\prod_{i=1}^k\Xmc_3(S_j)$, and choose an orientation for each $c\in\Dmc$. For any $\mu\in\Cmc(S,\Dmc)^{\mathrm{adm}}$, let
\[\Dmc_{\mathrm{hol}(\mu)}:=\{c\in\Dmc:\mu\text{-type of c is bulge }\pm\infty\}.\]
Then $|\mathrm{hol}^{-1}(\mathrm{hol}(\mu))|=2^{|\Dmc_{\mathrm{hol}(\mu)}\cup\Amc_{\mathrm{hol}(\mu)}|}$. Furthermore, each element in $\mathrm{hol}^{-1}(\mathrm{hol}(\mu))$ corresponds to the choice of whether the $\mu$-type of $c$ is bulge $+\infty$ or bulge $-\infty$ for each $c \in\Dmc_{\mathrm{hol}(\mu)}$, and whether the $\mu$-type of $p$ is bulge $+\infty$ or bulge $-\infty$ for each $p \in\Amc_{\mathrm{hol}(\mu)}$.
\end{cor}

\subsection{Edge and triangle invariants}\label{sec:edgetriangle}
When $S$ is a closed surface, Goldman gave a parameterization for $\mathrm{hol}(\Cmc(S))$ that generalizes the Fenchel-Nielsen coordinates on $\Tmc(S)$ \cite{Go90}. Briefly, he did this by first parameterizing $\mathrm{hol}(\Cmc(P))$, where $P$ is the oriented thrice punctured sphere. Then, he  generalized this parameterization to all surfaces of negative Euler characteristic by parameterizing the space of ways to assemble convex $\Rbbb\Pbbb^2$ structures on pairs of pants together. This parameterization was later extended by Marquis to the setting of where $S$ is not closed \cite{Mar10}.

By modifying previous parameterizations of $\mathrm{hol}(\Cmc(S))$ given by Bonahon-Dreyer \cite{BonDre} and Fock-Goncharov \cite{FocGon}, the second author \cite{Zhang16} also gave a continuous (in fact real-analytic) coordinate system of $\mathrm{hol}(\Cmc(S))$ that is similar in flavor to the one given by Goldman. However, this coordinate system has the additional advantage that the parameters are naturally projective invariants, and thus have easier geometric interpretations. We will give a brief description of this coordinate system. To do so, one first needs to define the edge and triangle invariants.

Choose a pants decomposition $\Pmc$ of $S$ and let $\Tmc=\Tmc_\Pmc$ be the associated ideal triangulation described in Section \ref{sec: triangulation}. For any closed edge $c=[x,y]\in\Pmc$, choose a lift $\{x,y\}\in\widetilde{\Pmc}$ of $c$, and let $\widetilde{T}_x,\widetilde{T}_y$ be triangles in $\widetilde{\Theta}$ with $x,y$ as a vertex respectively. We refer to the orbit $\pi_1(S)\cdot \{\widetilde{T}_x,\widetilde{T}_y\}$ as a \emph{bridge} across $c$. See Figure \ref{bridge-figure}.

\begin{figure}
\begin{center}
\includegraphics[scale=1.1]{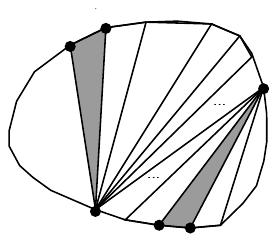}
\small
\put (-52, 11){$\widetilde{T}_x$}
\put (-103, 95){$\widetilde{T}_y$}
\put (-1, 78){$x$}
\put (-97, 6){$y$}
\put (-114, 107){$z_2$}
\put (-90, 116){$w_2$}
\put (-50, -2){$z_1$}
\put (-70, 0){$w_1$}
\end{center}
\caption{Bridge}
\label{bridge-figure}
\end{figure}

For each edge $[x,y]\in\Tmc$, choose a representative $\{x,y\}\in\widetilde{\Tmc}$ of $[x,y]$. If $\{x,y\}\in\widetilde{\Qmc}$, let $z_1,z_2\in\widetilde{\Vmc}$ so that $\{x,z_1\}, \{y,z_1\}, \{x,z_2\}, \{y,z_2\}\in\widetilde{\Tmc}$ and $x<z_2<y<z_1<x$. On the other hand, if $\{x,y\}\in\widetilde{\Pmc}$, let $\{\widetilde{T}_x,\widetilde{T}_y\}$ be an element in the bridge across $c$ so that $\widetilde{T}_{x}=\{x,z_1,w_1\}$ where $x<z_1<w_1<x$, and $\widetilde{T}_{y}=\{y,z_2,w_2\}$ where $y<z_2<w_2<y$. See Figures \ref{edge-inv-nonbridge-figure}
and \ref{edge-inv-bridge-figure}, noting that for $i=1,2$, $z_i=\xi^{(1)}(z_i)$ in these figures. Define
\[\begin{array}{rrcl}
&s_{x,y}:\mathrm{hol}(\Cmc(S))&\to&\Rbbb\\
 &[\rho]&\mapsto& C\left(\xi_\rho^{(2)}(x),\xi_\rho^{(1)}(z_2),\xi_\rho^{(1)}(z_1),\xi_\rho^{(1)}(x)+\xi_\rho^{(1)}(y)\right).
\end{array}\]
This is well-defined because the projective invariance of the cross ratio implies that all the choices we made are irrelevant (except for the choice of a bridge across each closed curve in $\Pmc$). Hence, for each edge $[x,y]\in\Tmc$, we have defined two invariants, $s_{x,y}$ and $s_{y,x}$. Observe using Remark \ref{rem:C+}, Lemma \ref{lem:op}, and Remark \ref{rem: op} that $s_{x,y}<0$, so one can define $\sigma_{x,y}:=\log (-s_{x,y})$. These are called the \emph{edge invariants} along $[x,y]$.

\begin{figure}
\begin{center}
\includegraphics[scale=1.2]{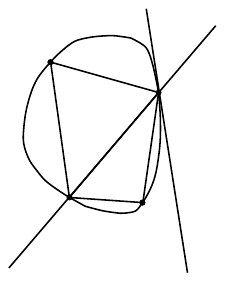}
\small
\put (-32, 78){$\xi^{(1)}_\rho(x)$}
\put (-116, 17){$\xi^{(1)}_\rho(y)$}
\put (-127, 103){$\xi^{(1)}_\rho(z_2)$}
\put (-53, 4){$\xi^{(1)}_\rho(z_1)$}
\put (-15, 122){$\xi^{(1)}_\rho(x)+\xi^{(1)}_\rho(y)$}
\put (-53, 130){$\xi^{(2)}_\rho(x)$}
\end{center}
\caption{Edge invariant for $\{x,y\}\in\mathcal Q$}
\label{edge-inv-nonbridge-figure}
\end{figure}

\begin{figure}
\begin{center}
\includegraphics[scale=1.1]{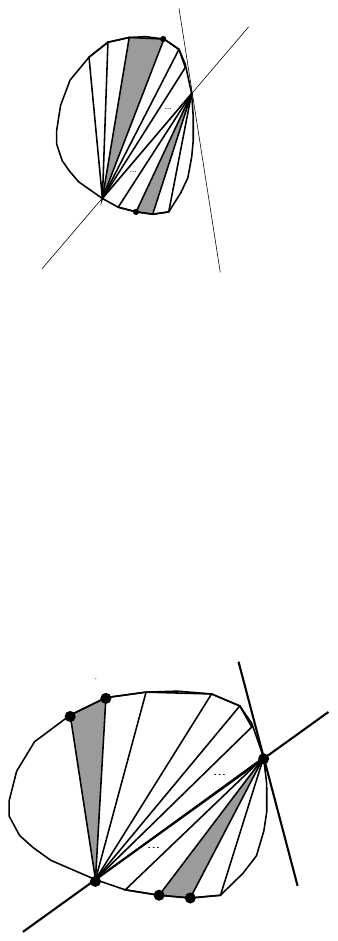}
\small
\put (-72, 9){$\widetilde{T}_x$}
\put (-123, 93){$\widetilde{T}_y$}
\put (-21, 76){$x$}
\put (-117, 4){$y$}
\put (-134, 105){$z_2$}
\put (-110, 114){$w_2$}
\put (-70, -4){$z_1$}
\put (-90, -2){$w_1$}
\end{center}
\caption{Edge invariant for $\{x,y\}\in\mathcal P$}
\label{edge-inv-bridge-figure}
\end{figure}

Similarly, for every ideal triangle $[x,y,z]\in\Theta$, choose a representative $\{x,y,z\}\in\widetilde{\Theta}$ so that $x<y<z<x$ in $\partial\pi_1(S)$. Define
\[\begin{array}{rrcl}
&t_{x,y,z}:\mathrm{hol}(\Cmc(S))&\to&\Rbbb\\
 &[\rho]&\mapsto& T\left(\xi_\rho^{(2)}(x),\xi_\rho^{(2)}(y), \xi_\rho^{(2)}(z),\xi_\rho^{(1)}(x),\xi_\rho^{(1)}(y),\xi_\rho^{(1)}(z)\right).
\end{array}\]
See Figure \ref{triangle-inv-figure}.
The projective invariance of the triple ratio again guarantees that the $t_{x,y,z}$ do not depend on any of the choices made. Furthermore, the symmetry of the triple ratio implies that $t_{x,y,z}=t_{y,z,x}=t_{z,x,y}$, so we only have one such function for each $[x,y,z]\in\Theta$. Again, observe using Remark \ref{rem:T+} and Proposition \ref{prop:transverse} that $t_{x,y,z}>0$, so one can define the \emph{triangle invariants} for $[x,y,z]$ to be $\tau_{x,y,z}:=\log (t_{x,y,z})$.

\begin{figure}
\begin{center}
\includegraphics{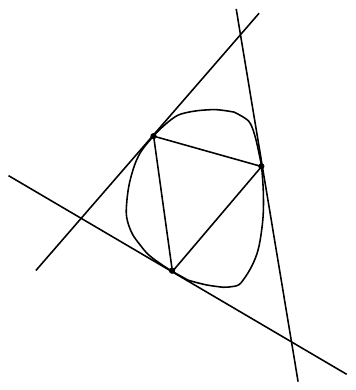}
\small
\put (-100, 46){$\xi_\rho^{(1)}(z)$}
\put (-165, 56){$\xi_\rho^{(2)}(y)$}
\put (-32, 103){$\xi_\rho^{(1)}(x)$}
\put (-75, 180){$\xi_\rho^{(2)}(x)$}
\put (-115, 120){$\xi_\rho^{(1)}(y)$}
\put (1, 7){$\xi_\rho^{(2)}(z)$}
\end{center}
\caption{Triangle invariant}
\label{triangle-inv-figure}
\end{figure}

\subsection{Coordinates on $\mathrm{hol}(\Cmc(P))$}

For now, we specialize to the case when $S=P$, the oriented thrice punctured sphere. Let $\alpha,\beta,\gamma\in\pi_1(P)$ be three group elements corresponding to oriented peripheral curves in $P$, so that $\gamma\beta\alpha=\id$, and $P$ lies to the left of each oriented peripheral curve. Then observe that
\[\widetilde{\Tmc}:=\big\{\eta\cdot \{\alpha^-,\beta^-\}:\eta\in\pi_1(S)\big\}\cup\big\{\eta\cdot \{\beta^-,\gamma^-\}:\eta\in\pi_1(S)\big\}\cup\big\{\eta\cdot \{\gamma^-,\alpha^-\}:\eta\in\pi_1(S)\big\},\]
is an ideal triangulation of $\widetilde{P}$. With this ideal triangulation, $\widetilde{\Pmc}=\emptyset$, $\widetilde{\Qmc}=\widetilde{\Tmc}$,
\[\widetilde{\Theta}:=\big\{\eta\cdot \{\alpha^-,\beta^-,\gamma^-\}:\eta\in\pi_1(S)\big\}\cup\big\{\eta\cdot \{\alpha^-,\alpha\cdot \beta^-,\gamma^-\}:\eta\in\pi_1(S)\big\},\]
and
\[\widetilde{\Vmc}:= \{\eta\cdot\alpha^-:\eta\in\pi_1(S)\}\cup\{\eta\cdot\beta^-:\eta\in\pi_1(S)\}\cup\{\eta\cdot\gamma^-:\eta\in\pi_1(S)\}.\]
In particular, $\Tmc$ consists of three edges, $\Theta$ consists of two triangles, and $\Vmc$ consists of three vertices. See Figure \ref{pants-figure}.

\begin{figure}
\begin{center}
\includegraphics[width=0.9\textwidth]{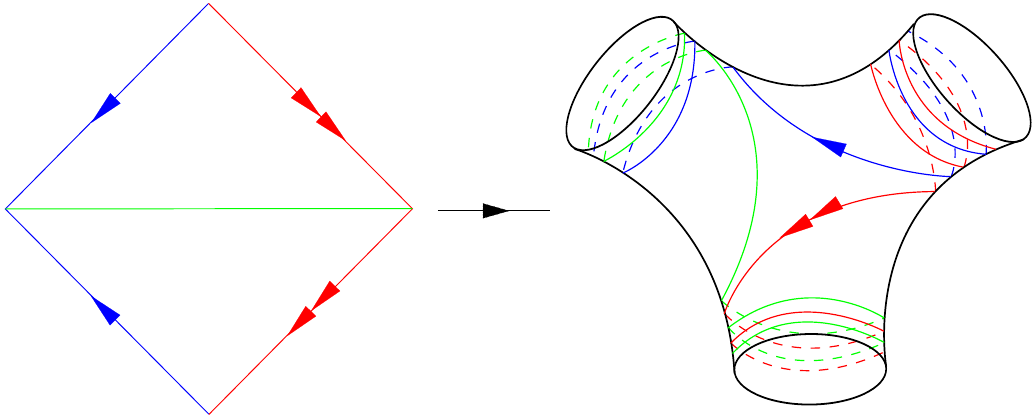}
\end{center}
\caption{Ideal triangulation of pair of pants}
\label{pants-figure}
\end{figure}

Bonahon-Dreyer \cite{BonDre} computed an expression for the eigenvalues of $\rho(\alpha)$, $\rho(\beta)$ and $\rho(\gamma)$ in terms of the edge and triangle invariants associated to $\Tmc$. Explicitly, if we denote the (generalized) eigenvalues of $\rho(\eta)$ by $\lambda_{1,\eta}(\rho)\geq\lambda_{2,\eta}(\rho)\geq\lambda_{3,\eta}(\rho)$ for any $\eta\in\pi_1(P)\setminus\{\id\}$, then for $i=1,2$, define $\ell_{i,[\eta]}:\mathrm{hol}(\Cmc(P))\to\Rbbb$ by
\[\ell_{i,[\eta]}(\rho):=\log\left(\frac{\lambda_{i,\eta}(\rho)}{\lambda_{i+1,\eta}(\rho)}\right).\]
Note that the $\ell_{i,[\eta]}$ depends only on the conjugacy class $[\eta]$, and not on the choice of $\eta\in\pi_1(P)$.
Then Bonahon-Dreyer showed that
\begin{eqnarray*}
\ell_{1,[\alpha]}&=&\sigma_{\alpha^-,\beta^-}+\sigma_{\alpha^-,\gamma^-}\\
\ell_{2,[\alpha]}&=&\sigma_{\beta^-,\alpha^-}+\sigma_{\gamma^-,\alpha^-}+\tau_{\alpha^-,\gamma^-,\beta^-}+\tau_{\alpha^-,\alpha\cdot\gamma^-,\beta^-}\\
\ell_{1,[\beta]}&=&\sigma_{\beta^-,\gamma^-}+\sigma_{\beta^-,\alpha^-}\\
\ell_{2,[\beta]}&=&\sigma_{\gamma^-,\beta^-}+\sigma_{\alpha^-,\beta^-}+\tau_{\alpha^-,\gamma^-,\beta^-}+\tau_{\alpha^-,\alpha\cdot\gamma^-,\beta^-}\\
\ell_{1,[\gamma]}&=&\sigma_{\gamma^-,\alpha^-}+\sigma_{\gamma^-\beta^-}\\
\ell_{2,[\gamma]}&=&\sigma_{\alpha^-,\gamma^-}+\sigma_{\beta^-\gamma^-}+\tau_{\alpha^-,\gamma^-,\beta^-}+\tau_{\alpha^-,\alpha\cdot\gamma^-,\beta^-}
\end{eqnarray*}

In particular, the expressions on the right have to be at least $0$. These six inequalities are known as the (weak) \emph{closed leaf inequalities}. Bonahon-Dreyer then showed that these are the only relations satisfied by these parameters.

\begin{thm} [Bonahon-Dreyer]\label{thm: Bonahon-Dreyer}
The map $\Phi:\mathrm{hol}(\Cmc(P))\to\Rbbb^8$ given by
\[\begin{array}{l}
\Phi:[\rho]\mapsto \left(\sigma_{\alpha^-,\beta^-}(\rho),\sigma_{\beta^-,\alpha^-}(\rho),\sigma_{\alpha^-,\gamma^-}(\rho),\sigma_{\gamma^-,\alpha^-}(\rho),\right.\\
\hspace{4.5cm}\left.\sigma_{\beta^-,\gamma^-}(\rho),\sigma_{\gamma^-,\beta^-}(\rho),\tau_{\alpha^-,\gamma^-,\beta^-}(\rho),\tau_{\alpha^-,\alpha\cdot\gamma^-,\beta^-}(\rho)\right)
\end{array}\]
is a homeomorphism onto the closed convex polytope in $\Rbbb^8$ cut out by the closed leaf inequalities.
\end{thm}

Solving the six linear equations above then proves the following.
\begin{cor}\label{cor:coordinate}
The map $\Phi:\mathrm{hol}(\Cmc(P))\to(\Rbbb_{\geq 0})^6\times\Rbbb^2$ given by
\[\Phi:[\rho]\mapsto \left(\ell_{1,[\alpha]}(\rho),\ell_{2,[\alpha]}(\rho),\ell_{1,[\beta]}(\rho),\ell_{2,[\beta]}(\rho),\ell_{1,[\gamma]}(\rho),\ell_{2,[\gamma]}(\rho),\sigma_{\alpha^-,\beta^-}(\rho),\tau_{\alpha^-,\gamma^-,\beta^-}(\rho)\right)\]
is a homeomorphism.
\end{cor}

\begin{remark}
Bonahon-Dreyer \cite{BonDre} and the second author \cite{Zhang16} were working in the more general setting of Hitchin representations, so they only stated their results for representations where the holonomy about each boundary component was required to be hyperbolic. However, in the case of convex $\Rbbb\Pbbb^2$ structures, Proposition \ref{prop:conv-flag} extends their arguments verbatim to the cases where the holonomy about the boundary component is quasi-hyperbolic or parabolic.
\end{remark}

In the coordinate system given in Corollary \ref{cor:coordinate}, the invariants $\sigma_{\alpha^-,\beta^-}$ and $\tau_{\alpha^-,\gamma^-,\beta^-}$ are called the \emph{internal parameters} of $P$, and the six invariants $\ell_{1,[\alpha]}$, $\ell_{2,[\alpha]}$, $\ell_{1,[\beta]}$, $\ell_{2,[\beta]}$, $\ell_{1,[\gamma]}$, $\ell_{2,[\gamma]}$ are called the \emph{length parameters}. We will simplify notation and denote $\sigma_{\alpha^-,\beta^-}$ and $\tau_{\alpha^-,\gamma^-,\beta^-}$ by $i_{1,P}$ and $i_{2,P}$ respectively.

\subsection{Coordinates on $\mathrm{hol}(\Cmc(S))$}\label{sec: coordhol}

Now, we will use the parameterization of $\mathrm{hol}(\Cmc(P))$ to parameterize $\mathrm{hol}(\Cmc(S))$. To do so, choose once and for all
\begin{itemize}
\item a pants decomposition $\Pmc$ on $S$,
\item a bridge across each closed curve in $\Pmc$,
\item an orientation for every closed curve in $\Pmc$
\item an orientation about each puncture of $S$.
\end{itemize}
For every $c\in\Pmc$, let $p_1,p_2$ be the punctures of $S\setminus c$ corresponding to $c$. If $c$ is non-separating, then $c$ (equipped with its chosen orientation) determines two conjugacy classes $[\gamma_1],[\gamma_2]\in[\pi_1(S\setminus c)]$, so that $[\gamma_i]$ corresponds to the puncture $p_i$. On the other hand, if $c$ is separating, let $S_1$ and $S_2$ be the two connected components of $S\setminus c$, so that $p_i$ is a puncture of $S_i$. Then $c$ determines a conjugacy class $[\gamma_i]\in[\pi_1(S_i)]$ corresponding to $p_i$.

Goldman \cite{Go90} proved that if $c$ is non-separating, then for any $\mu'\in\Cmc(S\setminus c)^\mathrm{adm}$, there is some $\mu\in\Cmc(S)^\mathrm{adm}$ so that $\mu|_{S\setminus c}=\mu'$ if and only if the $\mu'$-type of $p_i$ is bulge $-\infty$ for both $i=1,2$, and $\mathrm{hol}_{\mu'}(\gamma_1)=\mathrm{hol}_{\mu'}(\gamma_2)$. Similarly, if $c$ is separating, then for any $(\mu_1,\mu_2)\in\Cmc(S_1)^\mathrm{adm}\times\Cmc(S_2)^\mathrm{adm}$, there is some $\mu\in\Cmc(S)^\mathrm{adm}$ so that $\mu|_{S_i}=\mu_i$ for $i=1,2$ if and only if the $\mu_i$-type of $p_i$ is bulge $-\infty$ for both $i=1,2$, and $\mathrm{hol}_{\mu_1}(\gamma_1)=\mathrm{hol}_{\mu_2}(\gamma_2)$. Furthermore, regardless of whether $c$ is separating or not, Goldman \cite{Go90} also showed that the set of all such $\mu\in\Cmc(S)^\mathrm{adm}$ is parameterized by two parameters, called \emph{bulge} and \emph{twist parameters} $b_c,t_c:\Cmc(S)^\mathrm{adm}\to\Rbbb$. These are defined by
\[b_c(\mu):=\sigma_{\gamma_c^-,\gamma_c^+}(\mathrm{hol}_\mu)-\sigma_{\gamma_c^+,\gamma_c^-}(\mathrm{hol}_\mu)\,\,\,\text{ and }\,\,\,t_c(\mu):= \sigma_{\gamma_c^-,\gamma_c^+}(\mathrm{hol}_\mu)+\sigma_{\gamma_c^+,\gamma_c^-}(\mathrm{hol}_\mu) .\]
In particular, $b_c$ and $t_c$ depend only on $\mathrm{hol}_\mu$, so we also denote $b_c(\mathrm{hol}_\mu):=b_c(\mu)$ and $t_c(\mathrm{hol}_\mu):=t_c(\mu)$.
See Figure \ref{twist-bulge-figure}. 
\begin{figure}
\begin{center}
\includegraphics{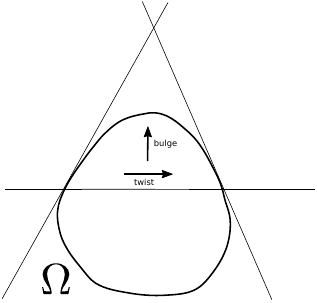}
\end{center}
\caption{Twist and bulge}
\label{twist-bulge-figure}
\end{figure}

\begin{remark}
Given a simple closed curve $c$ in $S$, Goldman defined an $\Rbbb^2$ action on $\Cmc(S)$ by \emph{bulge and shearing deformations} along $c$. The bulge and twist parameters $b_c$ and $t_c$ were designed to precisely capture these deformations; performing a bulging deformation changes the bulge parameter while keeping the twist parameter fixed, while performing a twist deformation changes the twist parameter while keeping the bulge parameter fixed.
\end{remark}

\begin{remark}
Goldman \cite{Go90} stated his results in the case when the $\mu$-type of all the punctures of $S$ are bulge $-\infty$, since he was mainly interested in the closed surface case. However, his arguments work in this more general setting as well.
\end{remark}

Combining this together with Corollary \ref{cor:coordinate} proves the following theorem.

\begin{thm}\label{thm: par}
Let $S$ be a connected, orientable surface with negative Euler characteristic, genus $g$ with $n$ punctures. Make the choices that we did at the start of this section, and let $\Pmc=\{c_1,\dots,c_{3g-3+n}\}$, let $\{d_1,\dots,d_n\}$ be the punctures of $S$, and let $\Pbbb=\{P_1,\dots,P_{2g-2+n}\}$. Then
\[\mathring{\Phi}:\mathrm{hol}(\Cmc(S))\to(\Rbbb^2)^{3g-3+n}\times(\Rbbb_+^2)^{3g-3+n}\times(\Rbbb_{\geq 0}^2)^n\times(\Rbbb^2)^{2g-2+n}\]
is a homeomorphism, where
\[\mathring{\Phi}:=\prod_{i=1}^{3g-3+n}(b_{c_i},t_{c_i})\prod_{i=1}^{3g-3+n}(\ell_{1,c_i},\ell_{2,c_i})\prod_{i=1}^{n}(\ell_{1,d_i},\ell_{2,d_i})\prod_{j=1}^{2g-2+n}(i_{1,P_j},i_{2,P_j}).\]
\end{thm}

\begin{remark}
Again, the second author \cite{Zhang16} proved Theorem \ref{thm: par} for convex $\Rbbb\Pbbb^2$ structures where the holonomy about each boundary component was required to be hyperbolic. Proposition \ref{prop:conv-flag} extends his proof verbatim to the cases where the holonomy about the boundary component is quasi-hyperbolic or parabolic.
\end{remark}

The homeomorphism $\mathring{\Phi}$ is not ideal for our purposes because it does not behave well under Dehn twists about the curves in $\Pmc$. We will thus further modify $\mathring{\Phi}$ to get a new homeomorphism that has that property. For each $c\in\Pmc$, let
\[r_c:=\frac{\ell_{1,c}\cdot\sigma_{\gamma_c^-,\gamma_c^+}-\ell_{2,c}\cdot\sigma_{\gamma_c^+,\gamma_c^-}}{3}\]
be the \emph{reparameterized bulge parameters}. Observe that if we replace the parameters $b_{c_i}$ with $r_{c_i}$ for all $i$ in the homeomorphism $\mathring{\Phi}$, then this defines a new homeomorphism
\[\Phi:\mathrm{hol}(\Cmc(S))\to(\Rbbb^2)^{3g-3+n}\times(\Rbbb_+^2)^{3g-3+n}\times(\Rbbb_{\geq 0}^2)^n\times(\Rbbb^2)^{2g-2+n}\]
given by
\[\Phi:=\prod_{i=1}^{3g-3+n}(r_{c_i},t_{c_i})\prod_{i=1}^{3g-3+n}(\ell_{1,c_i},\ell_{2,c_i})\prod_{i=1}^{n}(\ell_{1,d_i},\ell_{2,d_i})\prod_{j=1}^{2g-2+n}(i_{1,P_j},i_{2,P_j}).\]
The next proposition describes how $\Phi$ behaves under Dehn twists about the curves in $\Pmc$.

\begin{prop}\label{prop:Dehntwistcoords}
Let $D_c$ be the Dehn twist along the (oriented) closed curve $c\in\Pmc$. Then all the coordinate functions of $\Phi$ agree at $\mathrm{hol}_\mu$ and $D_c\cdot \mathrm{hol}_\mu$, except for $t_c$ which satisfies $t_c(D_c\cdot\mathrm{hol}_\mu)=t_c(\mathrm{hol}_\mu)+\ell_{1,c}(\mathrm{hol}_\mu)+\ell_{2,c}(\mathrm{hol}_\mu)$.
\end{prop}

\begin{proof}
It is clear from the projective invariance of the coordinate functions that the only possible coordinate functions of $\Phi$ that might differ at $\mathrm{hol}_\mu$ and $D_c\cdot\mathrm{hol}_\mu$ are $t_c$ and $r_c$. Let $\pi_1(S)\cdot\{\widetilde{T}_{\gamma_c^-},\widetilde{T}_{\gamma_c^+}\}$ be the bridge across $c$ that we chose to define $\Phi$. For $i=1,2$, let $z_i,w_i\in\widetilde{\Vmc}$ be the points such that $\widetilde{T}_{\gamma_c^-}=\{\gamma_c^-,z_1,w_1\}$ and $\gamma_c^-<z_1<w_1<\gamma_c^-$, and $\widetilde{T}_{\gamma_c^+}=\{\gamma_c^+,z_2,w_2\}$ where $\gamma_c^+<z_2<w_2<\gamma_c^+$. Observe that by choosing a basepoint in $S$, $D_c$ induces a group homomorphism, $D_c:\pi_1(S)\to\pi_1(S)$, which sends the bridge $\pi_1(S)\cdot\{\widetilde{T}_{\gamma_c^-},\widetilde{T}_{\gamma_c^+}\}$ across $c$ to the bridge $\pi_1(S)\cdot\{\widetilde{T}_{\gamma_c^-},\gamma_c\cdot\widetilde{T}_{\gamma_c^+}\}$ across $c$. See Figure \ref{dehn-twist-figure}.
\begin{figure}
\begin{center}
\includegraphics{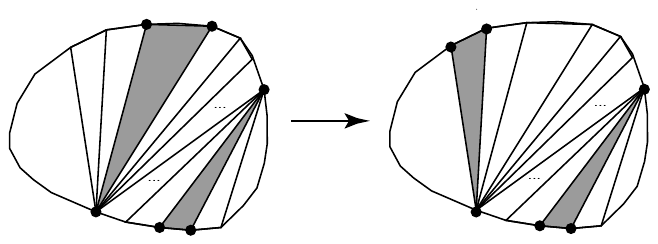}
\small
\put (-185, 74){$x$}
\put (-276, 7){$y$}
\put (-1, 74){$x$}
\put (-92, 7){$y$}
\put (-162, 62){$D_c$}
\end{center}
\caption{Dehn twist along $c=[x,y]$}
\label{dehn-twist-figure}
\end{figure}

Choose representatives $\rho^1\in\mathrm{hol}_\mu$ and $\rho^2\in D_c\cdot \mathrm{hol}_\mu$ so that
\[\xi_{\rho^j}^{(1)}(\gamma_c^+)=[0:0:1]^T,\,\,\, \xi_{\rho^j}^{(1)}(\gamma_c^-)=[1:0:0]^T,\,\,\, \xi_{\rho^j}^{(1)}(z_1)=[1:1:1]^T,\]
\[\xi_{\rho^j}^{(2)}(\gamma_c^+)=[1:0:0],\,\,\, \text{ and }\,\,\,\xi_{\rho^j}^{(2)}(\gamma_c^-)=[0:0:1]\]
for $j=1,2$. Then it follows that
\[\xi^{(1)}_{\rho^2}(z_2)=\rho(\gamma_c)\cdot\xi^{(1)}_{\rho^1}(z_2)=
\left[\begin{array}{ccc}
\exp\left(\ell_{1,c}(\mathrm{hol}_\mu)\right)&0&0\\
0&1&0\\
0&0&\exp\left(-\ell_{2,c}(\mathrm{hol}_\mu)\right)
\end{array}\right]
\cdot\xi^{(1)}_{\rho^1}(z_2).\]

Let $a,d\in\Rbbb$ so that $\xi^{(1)}_{\rho^1}(z_2)=[a:1:d]^T$. Observe that $a,d<0$, and
\[\xi^{(1)}_{\rho^2}(z_2)=[\exp\left(\ell_{1,c}(\mathrm{hol}_\mu)\right) a:1:\exp\left(-\ell_{2,c}(\mathrm{hol}_\mu)\right)d]^T.\]
One can then compute that
\begin{eqnarray*}
r_c(D_c\cdot\mathrm{hol}_\mu)&=&\frac{\ell_{1,c}(D_c\cdot\mathrm{hol}_\mu)\sigma_{\gamma_c^-,\gamma_c^+}(D_c\cdot\mathrm{hol}_\mu)-\ell_{2,c}(D_c\cdot\mathrm{hol}_\mu)\sigma_{\gamma_c^+,\gamma_c^-}(D_c\cdot\mathrm{hol}_\mu)}{3}\\
&=&\frac{\ell_{1,c}(\mathrm{hol}_\mu)(-\log(-d)+\ell_{2,c}(\mathrm{hol}_\mu))-\ell_{2,c}(\mathrm{hol}_\mu)(\log(-a)+\ell_{1,c}(\mathrm{hol}_\mu))}{3}\\
&=&\frac{-\ell_{1,c}(\mathrm{hol}_\mu)\log(-d)-\ell_{2,c}(\mathrm{hol}_\mu)\log(-a)}{3}\\
&=&\frac{\ell_{1,c}(\mathrm{hol}_\mu)\sigma_{\gamma_c^-,\gamma_c^+}(\mathrm{hol}_\mu)-\ell_{2,c}(\mathrm{hol}_\mu)\sigma_{\gamma_c^+,\gamma_c^-}(\mathrm{hol}_\mu)}{3}\\
&=&r_c(\mathrm{hol}_\mu).
\end{eqnarray*}
On the other hand,
\begin{eqnarray*}
t_c(D_c\cdot\mathrm{hol}_\mu)&=&\sigma_{\gamma_c^-,\gamma_c^+}(D_c\cdot\mathrm{hol}_\mu)+\sigma_{\gamma_c^+,\gamma_c^-}(D_c\cdot\mathrm{hol}_\mu)\\
&=&-\log(-d)+\ell_{2,c}(\mathrm{hol}_\mu)+\log(-a)+\ell_{1,c}(\mathrm{hol}_\mu)\\
&=&\sigma_{\gamma_c^-,\gamma_c^+}(\mathrm{hol}_\mu)+\sigma_{\gamma_c^+,\gamma_c^-}(\mathrm{hol}_\mu)+\ell_{1,c}(\mathrm{hol}_\mu)+\ell_{2,c}(\mathrm{hol}_\mu)\\
&=&t_c(\mathrm{hol}_\mu)+\ell_{1,c}(\mathrm{hol}_\mu)+\ell_{2,c}(\mathrm{hol}_\mu).
\end{eqnarray*}
\end{proof}

Furthermore, the reparametrized bulge parameter going to $+\infty$ or $-\infty$ has the following interpretation.

\begin{lem}\label{lem:bulge}
Let $z_1,z_2\in\widetilde{\Vmc}$ so that $\gamma_c^-<z_2<\gamma_c^+<z_1<\gamma_c^-$ be the two points used to define $b_c$ (they correspond to a bridge across $c$). For all $j\in\Zbbb^+$, let $\mu^j\in\Cmc(S)$ and choose $\rho^j\in\mathrm{hol}(\mu^j)$ so that for all $i,j$,
\[\xi_{\rho^j}(\gamma_c^+)=\xi_{\rho^i}(\gamma_c^+),\,\,\,\xi_{\rho^j}(\gamma_c^-)=\xi_{\rho^i}(\gamma_c^-),\,\,\,\text{ and }\,\,\,\xi_{\rho^j}^{(1)}(z_1)=\xi_{\rho^i}^{(1)}(z_1).\]
Also, let $\Omega^j$ be the $\rho^j$-equivariant developing image of $\mu^j$, and let $\Delta$ be the (open) triangle containing $\xi_{\rho^j}^{(1)}(z_2)$ whose vertices are $\xi_{\rho_j}^{(1)}(\gamma_c^+)$, $\xi_{\rho_j}^{(1)}(\gamma_c^-)$, and $\xi_{\rho_j}^{(2)}(\gamma_c^-)\cap \xi_{\rho_j}^{(2)}(\gamma_c^+)$. (Note that $\Delta$ does not depend on $j$.) If there is some $C>0$ so that
\[-C<t_c(\mathrm{hol}_{\mu^j})<C\,\,\text{ and }\,\,\frac{1}{C}<\ell_{1,c}(\mathrm{hol}_{\mu^j}),\,\,\ell_{2,c}(\mathrm{hol}_{\mu^j})<C\]
for all $j$, then
\begin{enumerate}
\item $\lim_{j\to\infty}r_c(\mu^j)=+\infty$ if and only if $\lim_{j\to\infty}\Omega^j\cap\Delta=\Delta$.
\item $\lim_{j\to\infty}r_c(\mu^j)=-\infty$ if and only if $\lim_{j\to\infty}\Omega^j\cap\Delta$ is empty.
\end{enumerate}
\end{lem}

\begin{proof}
We will only prove (1); the proof of (2) is similar. By transforming everything by a projective transformation, we may assume that for all $j$,
\[\xi_{\rho^j}^{(1)}(\gamma_c^+)=[0:0:1]^T,\,\,\, \xi_{\rho^j}^{(1)}(\gamma_c^-)=[1:0:0]^T,\,\,\, \xi_{\rho^j}^{(1)}(z_1)=[1:1:1]^T,\]
\[\xi_{\rho^j}^{(2)}(\gamma_c^+)=[1:0:0],\,\,\, \text{ and }\,\,\,\xi_{\rho^j}^{(2)}(\gamma_c^-)=[0:0:1].\]
Let $\xi^{(1)}_{\rho^j}(z_2)=[a_j:1:d_j]^T$. Then $a_j,d_j<0$, and a straightforward computation shows that
\[r_c(\mathrm{hol}_{\mu_j})=\frac{-\ell_{1,c}(\mathrm{hol}_{\mu_j})\log(-d_j)-\ell_{2,c}(\mathrm{hol}_{\mu_j})\log(-a_j)}{3}\]
and
\[t_c(\mathrm{hol}_{\mu_j})=\log(-a_j)-\log(-d_j).\]
Hence, $e^{-C}<\frac{a_j}{d_j}<e^C$ for all $j$. Since $\frac{1}{C}<\ell_{1,c}(\mathrm{hol}_{\mu^j}),\,\,\ell_{2,c}(\mathrm{hol}_{\mu^j})<C$ for all $j$, $\lim_{j\to\infty}r_c(\mu_j)=+\infty$ if and only if $\lim_{j\to\infty}a_j=0=\lim_{j\to\infty}d_j$. This implies that $\lim_{j\to\infty}r_c(\mu_j)=+\infty$ if and only if
\[\lim_{j\to\infty}[a_j:1:d_j]^T=[0:1:0]^T,\]
which happens if and only if $\lim_{j\to\infty}\Omega^j\cap\Delta=\Delta$ by the convexity of $\Omega_j$.
\end{proof}

From our construction, it is clear that for any multi-curve $\Dmc\subset\Pmc$ and any connected component $S'$ of $S\setminus\Dmc$, the oriented pants decomposition $\Pmc$ on $S$ restricts to an oriented pants decomposition $\Pmc'$ on $S'$. Also, the ideal triangulation $\Tmc_{\Pmc'}$ on $S'$ is naturally a subset of the ideal triangulation $\Tmc_\Pmc$ on $S$, so the choice of bridge across each closed edge in $\Tmc_\Pmc$ induces a choice of bridge across each closed edge of $\Tmc_{\Pmc'}$. Thus, $\Pmc$ together with the choice of a bridge across each closed edge in $\Tmc_\Pmc$ determines a coordinate system on $\mathrm{hol}(\Cmc(S'))$. Furthermore,
\begin{enumerate}
\item The length parameters on $\mathrm{hol}(\Cmc(S'))$ are exactly the length parameters on $\mathrm{hol}(\Cmc(S))$ associated to the closed curves in $\Pmc'\subset\Pmc$ and the boundary curves of $S'$.
\item The internal parameters on $\mathrm{hol}(\Cmc(S'))$ are exactly the internal parameters on $\mathrm{hol}(\Cmc(S))$ associated to the pairs of pants in $\Pbbb_\Pmc$ that lie in $S'$.
\item The twist, bulge, and reparameterized bulge parameters on $\mathrm{hol}(\Cmc(S'))$ are exactly the twist, bulge, and reparameterized bulge parameters on $\mathrm{hol}(\Cmc(S))$ associated to the closed curves in $\Pmc'\subset\Pmc$.
\end{enumerate}
This immediately implies the following remark.

\begin{remark}\label{rem: coordinates}
For all $\Dmc\subset\Pmc$, the coordinate system on $\mathrm{hol}(\Cmc(S))$ determines a coordinate system on $\mathrm{hol}(\Cmc(S,\Dmc)^\mathrm{adm})\subset\prod_{i=1}^k\mathrm{hol}(\Cmc(S_i))$. More precisely, if $\{d_1,\dots,d_n\}$ are the punctures of $S$, $\Pmc\setminus\Dmc=\{c_1,\dots,c_{3g-3+n-k}\}$, $\Dmc=\{e_1,\dots,e_k\}$, and $\Pbbb=\{P_1,\dots,P_{2g-2+n}\}$, then this coordinate system is given explicitly by the homeomorphism
\[\Phi_\Dmc:\mathrm{hol}(\Cmc(S,\Dmc)^\mathrm{adm})\to(\Rbbb^2)^{3g-3+n-k}\times(\Rbbb_+^2)^{3g-3+n-k}\times(\Rbbb_{\geq 0}^2)^{k+n}\times(\Rbbb^2)^{2g-2+n},\]
where
\begin{eqnarray*}
\Phi_\Dmc&:=&\prod_{i=1}^{3g-3+n-k}(r_{c_i},t_{c_i})\prod_{i=1}^{3g-3+n-k}(\ell_{1,c_i},\ell_{2,c_i})\\
&&\hspace{2cm}\prod_{i=1}^{k}(\ell_{1,e_i},\ell_{2,e_i})\prod_{i=1}^{n}(\ell_{1,d_i},\ell_{2,d_i})\prod_{j=1}^{2g-2+n}(i_{1,P_j},i_{2,P_j}).
\end{eqnarray*}
\end{remark}

\section{Coordinate description of the topology of $\Cmc(S)^\mathrm{aug}$} \label{sec: coord}

As we observed in (4), (5) and (6) of Remark \ref{rem: top}, $\Cmc(S)^\mathrm{aug}$ is not locally compact and $\Cmc(S)^\mathrm{aug}/\MCG(S)$ is an orbifold with a complicated singular locus. As such, it is not easy to give local coordinates for either $\Cmc(S)^\mathrm{aug}$ or $\Cmc(S)^\mathrm{aug}/\MCG(S)$.

This however, tells us that taking the quotient of $\Cmc(S)^\mathrm{aug}$ by the trivial group is ``too big", while taking the quotient of $\Cmc(S)^\mathrm{aug}$ by all of $\MCG(S)$ is ``too small". The naive dream is then to find a subgroup $G\subset\MCG(S)$ so that the topology on $\Cmc(S)^\mathrm{aug}/G$ admits local coordinates about every point $p$. It turns out that this too is impossible. However, the next theorem tells us that if we allow $G$ to change depending on the stratum of $\Cmc(S)^\mathrm{aug}$ where $p$ lies, then there is an explicit description of the topology on a neighborhood of $p$ in $\Cmc(S)^\mathrm{aug}/G$.

More precisely, let $\Dmc$ be any multi-curve and let $G_\Dmc\subset\MCG(S)$ be the subgroup generated by Dehn twists about $\Dmc$. Recall that we previously defined
\[\Cmc(S)^{\mathrm{aug},\Dmc}:=\bigcup_{\Dmc'\subset\Dmc}\Cmc(S,\Dmc')^\mathrm{adm}.\]
For the purpose of simplifying notation, set $V_\Dmc:=\Cmc(S)^{\mathrm{aug},\Dmc}$. Observe that there is a natural $G_\Dmc$-action on $X_{\Dmc'}:=\Cmc(S,\Dmc')^\mathrm{adm}$ for all $\Dmc'\subset\Dmc$, so we can define $W_{\Dmc',\Dmc}:=X_{\Dmc'}/G_\Dmc$ and $U_\Dmc:=V_\Dmc/G_\Dmc$. Equip $W_{\Dmc',\Dmc}$ and $U_\Dmc$ with the quotient topology. By (1) of Remark \ref{rem: top}, $V_\Dmc\subset\Cmc(S)^\mathrm{aug}$ is an open set about any point in $X_\Dmc\subset\Cmc(S)^\mathrm{aug}$.

The goal of this section is to prove the following theorem.

\begin{thm}\label{thm: main}
Let $\Dmc$ be any multi-curve, let $\Pmc\supset\Dmc$ be an oriented pants decomposition on $S$ and choose orientations about the punctures of $S$. Let $\Tmc_\Pmc$ be the induced ideal triangulation as described in Section \ref{sec: triangulation}, and choose a bridge across every closed edge of $\Tmc_\Pmc$. Then there is an explicit homeomorphism
\[\Psi_\Dmc:U_\Dmc\to(\Rbbb^2)^{3g-3+n-k}\times(\Rbbb_+^2)^{3g-3+n-k}\times(\Rbbb^4)^k\times (\Rbbb^2)^n\times(\Rbbb^2)^{2g-2+n}.\]
In particular, $U_\Dmc$ is homeomorphic to a cell.
\end{thm}

\subsection{The homeomorphism $\Theta_{\Dmc',\Dmc}$}
 As a preliminary step to define $\Psi_\Dmc$, we first define a continuous parameterization $\Theta_{\Dmc',\Dmc}$ of $\mathrm{hol}(W_{\Dmc',\Dmc}):=\mathrm{hol}(X_{\Dmc'})/G_\Dmc$ for any $\Dmc'\subset\Dmc$. Per the hypothesis of Theorem \ref{thm: main}, choose an oriented pants decomposition $\Pmc\supset\Dmc$, an orientation on every boundary component of $S$, and a bridge across every closed edge in $\Tmc_\Pmc$. Remark \ref{rem: coordinates} tells us that for any $\Dmc'\subset\Dmc$, these choices determine a coordinate system on $\mathrm{hol}(X_{\Dmc'})$.

Let $\{S_1',\dots,S_k'\}$ be the connected components of $S\setminus \mathcal{D}'$. Observe that if $c\in\Dmc'$, then the Dehn twist $D_c\in G_\Dmc$ about $c$ acts as the identity on $\prod_{j=1}^k\Cmc(S_j')$. Also, if $c\in\Dmc\setminus\Dmc'$, then $c$ lies in the interior of $S_j'$ for some $j=1,\dots,k$. In that case, $D_c$ acts as the identity on $\Cmc(S_i')$ for all $i\neq j$, and its action on $\Cmc(S_j')$ induces an action on $\mathrm{hol}(\Cmc(S_j'))$, which we have described explicitly in terms of the coordinates on $\mathrm{hol}(\Cmc(S_j'))$ in Proposition \ref{prop:Dehntwistcoords}.

Proposition \ref{prop:Dehntwistcoords} implies that aside from the twist coordinates in $\{t_c:c\in\Dmc\setminus\Dmc'\}$, all the other coordinate functions on $\mathrm{hol}(X_{\Dmc'})$ descend to well-defined functions on $\mathrm{hol}(W_{\Dmc',\Dmc})$, which in turn give well-defined functions on $W_{\Dmc',\Dmc}$ by precomposing with the holonomy map $\mathrm{hol}:W_{\Dmc',\Dmc}\to\mathrm{hol}(W_{\Dmc',\Dmc})$. Although $t_c$ does not descend to a well-defined function on $\mathrm{hol}(W_{\Dmc',\Dmc})$ for all $c\in\Dmc\setminus\Dmc'$, we may replace $t_c$ with the map $\theta_c:\mathrm{hol}(X_{\Dmc'})\to\Sbbb:= \Rbbb/(2\pi\cdot\Zbbb)$ defined by
\begin{equation}\label{eqn:theta}
\theta_c:\mathrm{hol}(\mu)=(\mathrm{hol}(\mu_1),\dots,\mathrm{hol}(\mu_k))\mapsto \frac{2\pi t_c(\mathrm{hol}(\mu_j))}{\ell_{1,c}(\mathrm{hol}(\mu_j))+\ell_{2,c}(\mathrm{hol}(\mu_j))}.
\end{equation}
By Proposition \ref{prop:Dehntwistcoords}, this map descends to a well-defined map
$\theta_c:\mathrm{hol}(W_{\Dmc',\Dmc})\to \Sbbb$, so we may think of it as a map from $W_{\Dmc',\Dmc}$ to $\Sbbb$ by pre-composing with $\mathrm{hol}:W_{\Dmc',\Dmc}\to\mathrm{hol}(W_{\Dmc',\Dmc})$.

If $c\in \Dmc\setminus\Dmc'$, we may define the functions $g_{1,c},\dots,g_{4,c}:\mathrm{hol}(W_{\Dmc',\Dmc})\to\Rbbb$ by
\begin{eqnarray*}
g_{1,c}(\cdot)&:=&\ell_{1,c}(\cdot) \ell_{2,c}(\cdot) \cos\left(\frac{\pi}{2}f(r_c(\cdot))\right)\cos(\theta_c(\cdot)),\\
g_{2,c}(\cdot)&:=&\ell_{1,c}(\cdot) \ell_{2,c}(\cdot) \cos\left(\frac{\pi}{2}f(r_c(\cdot))\right)\sin(\theta_c(\cdot)),\\
g_{3,c}(\cdot)&:=&\ell_{1,c}(\cdot)\ell_{2,c}(\cdot) \sin\left(\frac{\pi}{2}f(r_c(\cdot))\right),\\
g_{4,c}(\cdot)&:=&\ell_{1,c}(\cdot)^2-\ell_{2,c}(\cdot)^2.
\end{eqnarray*}
In the above formulas, $f:\Rbbb\to\Rbbb$ is the smooth function given by $f(s)=\frac{e^s-1}{e^s+1}$. Again, for $i=1,\dots,4$ and $c\in \Dmc\setminus\Dmc'$, $g_{i,c}$ can also be viewed as functions on $W_{\Dmc',\Dmc}$ by pre-composing with $\mathrm{hol}:W_{\Dmc',\Dmc}\to\mathrm{hol}(W_{\Dmc',\Dmc})$.

On the other hand, if $c\in\Dmc'$ or $c$ is a puncture of $S$, define $g_{1,c},\dots,g_{4,c}:W_{\Dmc',\Dmc}\to\Rbbb$ by
\begin{eqnarray*}
g_{1,c}(\cdot)&:=&0\\
g_{2,c}(\cdot)&:=&0\\
g_{3,c}(\cdot)&:=&\left\{\begin{array}{ll}
0&\text{if }\mu\text{-type of } c\text{ is parabolic or quasi-hyperbolic}\\
\ell_{1,c}(\cdot)\ell_{2,c}(\cdot)&\text{if }\mu\text{-type of } c\text{ is bulge }+\infty\\
-\ell_{1,c}(\cdot)\ell_{2,c}(\cdot)&\text{if }\mu\text{-type of } c\text{ is bulge }-\infty
\end{array}\right.,\\
g_{4,c}(\cdot)&:=&\ell_{1,c}(\cdot)^2-\ell_{2,c}(\cdot)^2.
\end{eqnarray*}
Note that if $c\in\Dmc'$ or $c$ is a puncture of $S$, $g_{i,c}$ is a function on $\mathrm{hol}(W_{\Dmc',\Dmc})$ for $i=1,2,4$. This is not so for $g_{3,c}$, but its absolute value $|g_{3,c}(\cdot)|=\ell_{1,c}(\cdot)\ell_{2,c}(\cdot)$ is a function on $\mathrm{hol}(W_{\Dmc',\Dmc})$. Moreover, in the cases in which the reparametrized bulge parameter $r_c(\cdot) = \pm\infty$, we have simply extended the formulas above by identifying $f(\infty)=1$ and $f(-\infty)=-1$. 

\begin{notation}\label{not:standard}
Let $\Pmc\setminus\Dmc=:\{c_1,\dots,c_{3g-3+n-m}\}$, $\Dmc\setminus\Dmc'=:\{e_{m'+1},\dots,e_m\}$, $\Dmc'=:\{e_1,\dots,e_{m'}\}$, and $\{d_1,\dots,d_n\}$ be the punctures of $S$. Note that $m:=|\Dmc|$, $m':=|\Dmc'|\leq m$.
\end{notation}
With this notation, define
\[\Theta_{\Dmc',\Dmc}:\mathrm{hol}(W_{\Dmc',\Dmc})\to(\Rbbb^2)^{3g-3+n-m}\times(\Rbbb_+^2)^{3g-3+n-m}\times(\Rbbb^4)^m\times(\Rbbb^2)^n\times(\Rbbb^2)^{2g-2+n}\]
by
\begin{eqnarray*}
\Theta_{\Dmc',\Dmc}&\:=&\prod_{i=1}^{3g-3+n-m}(r_{c_i},t_{c_i})\prod_{i=1}^{3g-3+n-m}(\ell_{1,c_i},\ell_{2,c_i})\prod_{i=m'+1}^{m}(g_{1,e_i},g_{2,e_i},g_{3,e_i},g_{4,e_i})\\
&&\hspace{1cm}\prod_{i=1}^{m'}(g_{1,e_i},g_{2,e_i},|g_{3,e_i}|,g_{4,e_i})\prod_{i=1}^{n}(|g_{3,d_i}|,g_{4,d_i})\prod_{j=1}^{2g-2+n}(i_{1,P_j},i_{2,P_j}).
\end{eqnarray*}

\begin{lem}\label{lem:besthol}
Let $E_1:=(\Rbbb^2\setminus(0,0))\times \Rbbb^2\subset\Rbbb^4$ and $E_2:=(0,0)\times\Rbbb_{\geq 0}\times\Rbbb\subset\Rbbb^4$. The map $\Theta_{\Dmc',\Dmc}$ is a homeomorphism onto its image, which is
\[(\Rbbb^2)^{3g-3+n-m}\times(\Rbbb_+^2)^{3g-3+n-m}\times E_1^{m-m'}\times E_2^{m'}\times(\Rbbb_{\geq 0}\times\Rbbb)^n\times(\Rbbb^2)^{2g-2+n}.\]
\end{lem}

\begin{proof}
Recall from Remark \ref{rem: coordinates} that the map
\[\Phi_{\Dmc'}:\mathrm{hol}(X_{\Dmc'})\to(\Rbbb^2)^{3g-3+n-m'}\times(\Rbbb_+^2)^{3g-3+n-m'}\times(\Rbbb_{\geq 0}^2)^{m'+n}\times(\Rbbb^2)^{2g-2+n}\]
given by
\begin{eqnarray*}
\Phi_{\Dmc'}&:=&\left(\prod_{i=1}^{3g-3+n-m}(r_{c_i},t_{c_i})\prod_{i=m'+1}^{m}(r_{e_i},t_{e_i})\right)\\
&&\hspace{1.5cm}\left(\prod_{i=1}^{3g-3+n-m}(\ell_{1,c_i},\ell_{2,c_i})\prod_{i=m'+1}^{m}(\ell_{1,e_i},\ell_{2,e_i})\right)\\
&&\hspace{3cm}\prod_{i=1}^{m'}(\ell_{1,e_i},\ell_{2,e_i})\prod_{i=1}^{n}(\ell_{1,d_i},\ell_{2,d_i})\prod_{j=1}^{2g-2+n}(i_{1,P_j},i_{2,P_j}).
\end{eqnarray*}
is a homeomorphism. From this and the definition of the $G_\Dmc$-action on $\mathrm{hol}(X_{\Dmc'})$ defined above, we see that
\begin{eqnarray*}
\overline{\Phi}_{\Dmc',\Dmc}:\mathrm{hol}(W_{\Dmc',\Dmc})&\to&(\Rbbb^2)^{3g-3+n-m}\times(\Rbbb_+^2)^{3g-3+n-m}\times(\Rbbb\times\Sbbb\times\Rbbb_+^2)^{m-m'}\times\\
&&\hspace{4cm}(\Rbbb_{\geq 0}^2)^{m'}\times(\Rbbb_{\geq 0}^2)^n\times(\Rbbb^2)^{2g-2+n}
\end{eqnarray*}
given by
\begin{eqnarray*}
\overline{\Phi}_{\Dmc',\Dmc}&:=&\prod_{i=1}^{3g-3+n-m}(r_{c_i},t_{c_i})\prod_{i=1}^{3g-3+n-m}(\ell_{1,c_i},\ell_{2,c_i})\prod_{i=m'+1}^{m}(r_{e_i},\theta_{e_i},\ell_{1,e_i},\ell_{2,e_i})\\
&&\hspace{3cm}\prod_{i=1}^{m'}(\ell_{1,e_i},\ell_{2,e_i})\prod_{i=1}^{n}(\ell_{1,d_i},\ell_{2,d_i})\prod_{j=1}^{2g-2+n}(i_{1,P_j},i_{2,P_j}).
\end{eqnarray*}
is also a homeomorphism.

From the definition of $(g_{1,e_i},\dots,g_{4,e_i})$, one sees that to finish the proof, it is sufficient to prove that the maps $F_1:\Rbbb\times\Sbbb\times\Rbbb^2_+\to E_1$ and $F_2:(\Rbbb_{\geq 0})^2\to E_2$ defined by
\begin{eqnarray*}
F_1(a_1,a_2,a_3,a_4)&:=&\left(a_3a_4\cos\left(\frac{\pi}{2}f(a_1)\right)\cos(a_2),\right.\\
&&a_3a_4\cos\left(\frac{\pi}{2}f(a_1)\right)\sin(a_2),\left.a_3a_4\sin\left(\frac{\pi}{2}f(a_1)\right),a_3^2-a_4^2\right),\\
F_2(a_3,a_4)&:=&(0,0,a_3a_4,a_3^2-a_4^2)
\end{eqnarray*}
are homeomorphisms. But this can be verified easily by writing down explicit continuous formulas for the inverse maps for both $F_1$ and $F_2$.
\end{proof}

\subsection{$\Psi_\Dmc$ is a bijection}
Next, we will explicitly describe the map $\Psi_\Dmc$ in Theorem \ref{thm: main} and show that it is a bijection. Define
\[\Psi_{\Dmc',\Dmc}:W_{\Dmc',\Dmc}\to(\Rbbb^2)^{3g-3+n-m}\times(\Rbbb_+^2)^{3g-3+n-m}\times(\Rbbb^4)^m\times (\Rbbb^2)^n\times(\Rbbb^2)^{2g-2+n},\]
by
\begin{eqnarray*}
\Psi_{\Dmc',\Dmc}&:=&\prod_{i=1}^{3g-3+n-m}(r_{c_i},t_{c_i})\prod_{i=1}^{3g-3+n-m}(\ell_{1,c_i},\ell_{2,c_i})\prod_{i=1}^m(g_{1,e_i},g_{2,e_i},g_{3,e_i},g_{4,e_i})\\
&&\hspace{5cm}\prod_{i=1}^n(g_{3,d_i},g_{4,d_i})\prod_{j=1}^{2g-2+n}(i_{1,P_j},i_{2,P_j}).
\end{eqnarray*}

\begin{lem}\label{lem:bijection on strata}
Let $E_1:=(\Rbbb^2\setminus(0,0))\times \Rbbb^2\subset\Rbbb^4$ and $E_3:=\Rbbb^4\setminus E_1=(0,0)\times\Rbbb^2\subset\Rbbb^4$. The map $\Psi_{\Dmc',\Dmc}$ is a bijection onto its image, which is
\[(\Rbbb^2)^{3g-3+n-m}\times(\Rbbb_+^2)^{3g-3+n-m}\times E_1^{m-m'}\times E_3^{m'}\times(\Rbbb^2)^n\times(\Rbbb^2)^{2g-2+n}.\]
\end{lem}

\begin{proof}
For any $\mu\in X_{\Dmc'}$, let $\Amc_{\mathrm{hol}(\mu)}\subset\{d_1,\dots,d_n\}$ be the punctures of $S$ whose $\mu$-type is bulge $\pm\infty$, and let $\Dmc'_{\mathrm{hol}(\mu)}:=\{c\in\Dmc':\mu\text{-type of c is bulge }\pm\infty\}$. By Corollary \ref{cor: hol}, we see that $\mathrm{hol}^{-1}(\mathrm{hol}(\mu))\subset X_{\Dmc'}$ has $2^{|\Dmc'_{\mathrm{hol}(\mu)}\cup\Amc_{\mathrm{hol}(\mu)}|}$ elements, each of which corresponds to the choice of whether the $\mu$-type of $c$ is bulge $+\infty$ or bulge $-\infty$ for each $c\in\Dmc'_{\mathrm{hol}(\mu)}\cup\Amc_{\mathrm{hol}(\mu)}$. The same is true for $\mathrm{hol}^{-1}(\mathrm{hol}[\mu])\subset W_{\Dmc',\Dmc}$ as well, because the only element in $G_\Dmc$ that sends $\mathrm{hol}^{-1}(\mathrm{hol}(X_{\Dmc'}))$ to itself is the identity.

Note that by replacing the coordinate functions $|g_{3,c}|$ of $\Theta_{\Dmc',\Dmc}$ with $g_{3,c}$ for $c\in\{e_1,\dots,e_{m'}\}\cup\{d_1,\dots,d_n\}$ allows us to distinguish whether the $\mu$-type of $c$ is bulge $+\infty$ or bulge $-\infty$. The lemma follows immediately from this observation.
\end{proof}

As defined, the target of $\Psi_{\Dmc',\Dmc}$ does not depend on $m'$, but only on $m$. Since
\[U_\Dmc=\bigcup_{\Dmc'\subset\Dmc}W_{\Dmc',\Dmc},\]
is a disjoint union, we may define
\[\Psi_\Dmc:U_\Dmc\to(\Rbbb^2)^{3g-3+n-m}\times(\Rbbb_+^2)^{3g-3+n-m}\times(\Rbbb^4)^m\times (\Rbbb^2)^n\times(\Rbbb^2)^{2g-2+n}\]
by $\Psi_\Dmc[\mu]:=\Psi_{\Dmc',\Dmc}[\mu]$ if $[\mu]\in W_{\Dmc',\Dmc}$. As a consequence of Lemma \ref{lem:bijection on strata}, we have the following proposition.

\begin{prop}
The map $\Psi_\Dmc$ is a bijection.
\end{prop}

\begin{proof}
Simply note that $E_1\cup E_3=\Rbbb^4$ is a disjoint union, where $E_1$ and $E_3$ are as defined in the statement of Lemma \ref{lem:bijection on strata}.
\end{proof}

\subsection{$\Psi_\Dmc$ is a homeomorphism}
To finish the proof of Theorem \ref{thm: main}, we need to show that the bijection $\Psi_\Dmc$ is a homeomorphism. Recall that $\mathcal C(S)^{\rm aug}$ is first countable (see (3) of Remark \ref{rem: top}). Hence, it is sufficient to show that if $\mu=[\mu]\in X_\Dmc=W_{\Dmc,\Dmc}$ and $\{[\mu^j]\}_{j=1}^\infty$ is a sequence in $U_\Dmc$, then $\lim_{j\to\infty}[\mu^j]=[\mu]$ in $U_\Dmc$ if and only if $\lim_{j\to\infty}\Psi_\Dmc[\mu^j]=\Psi_\Dmc[\mu]$.

Let $S_1,\dots,S_k$ be the connected components of $S\setminus\Dmc$. Since $\Dmc$ is a finite set, by considering the subsequences of $\{[\mu^j]\}_{j=1}^\infty$ that lie in different strata separately, we may further assume that $\{[\mu^j]\}_{j=1}^\infty\subset W_{\Dmc',\Dmc}$ for some fixed $\Dmc'\subset\Dmc$. Also, note that the map $\mathrm{Pull}_{\Dmc',\Dmc}:X_{\Dmc'}\to X_\Dmc$ descends to $\mathrm{Pull}_{\Dmc',\Dmc}:W_{\Dmc',\Dmc}\to W_{\Dmc,\Dmc}$. Then (1) of Remark \ref{rem: top} implies that it is sufficient to prove the following theorem.

\begin{thm}\label{thm: special case}
Let $\Dmc$ be an oriented multi-curve on $S$, let $S_1,\dots,S_k$ be the connected components of $S\setminus\Dmc$, and let $\mu=(\mu_1,\dots,\mu_k)\in X_\Dmc$. Also, let $\Dmc'\subset\Dmc$, let $\{[\mu^j]\}_{j=1}^\infty\subset W_{\Dmc',\Dmc}$, and let $\mu^j_l\in\Cmc(S_l)$ so that $\mathrm{Pull}_{\Dmc',\Dmc}[\mu^j]=(\mu^j_1,\dots,\mu^j_k)$.
\begin{enumerate}
\item If $\lim_{j\to\infty}\mu^j_l=\mu_l$ for all $l=1,\dots,k$, then $\lim_{j\to\infty}\Psi_{\Dmc',\Dmc}[\mu^j]=\Psi_{\Dmc,\Dmc}(\mu)$.
\item If $\lim_{j\to\infty}\Psi_{\Dmc',\Dmc}[\mu^j]=\Psi_{\Dmc,\Dmc}(\mu)$, then $\lim_{j\to\infty}\mu^j_l=\mu_l$ for all $l=1,\dots,k$.
\end{enumerate}
\end{thm}

In the above theorem, we again choose an oriented pants decomposition $\Pmc\supset\Dmc$, a bridge across every closed edge of $\Tmc_\Pmc$, and an orientation about every puncture of $S$ to define $\Psi_\Dmc$.

We first prove (1) of Theorem \ref{thm: special case}.
\begin{proof}[Proof if (1) of Theorem \ref{thm: special case}]
First, observe that if $\lim_{j\to\infty}\mu^j_l=\mu_l$ for all $l=1,\dots,k$, then $\lim_{j\to\infty}\mathrm{hol}(\mu^j_l)=\mathrm{hol}(\mu_l)$ for all $l=1,\dots,k$. So Lemma \ref{lem:besthol} implies that
\begin{enumerate}[(i)]
\item $\displaystyle\lim_{j\to\infty}(i_{1,P_i}[\mu^j],i_{2,P_i}[\mu^j])=(i_{1,P_i}(\mu),i_{2,P_i}(\mu))$ for all $i=1,\dots,2g-2+n$,
\item $\displaystyle\lim_{j\to\infty}(r_{c_i}[\mu^j],t_{c_i}[\mu^j])=(r_{c_i}(\mu),t_{c_i}(\mu))$ for all $i=1,\dots,3g-3+n-m$,
\item $\displaystyle\lim_{j\to\infty}(\ell_{1,c_i}[\mu^j],\ell_{2,c_i}[\mu^j])=(\ell_{1,c_i}(\mu),\ell_{2,c_i}(\mu))$ for all $i=1,\dots,3g-3+n-m$,
\item $\displaystyle\lim_{j\to\infty}(\ell_{1,e_i}[\mu^j],\ell_{2,e_i}[\mu^j])=(\ell_{1,e_i}(\mu),\ell_{2,e_i}(\mu))$ for all $i=1,\dots,m$,
\item $\displaystyle\lim_{j\to\infty}(\ell_{1,d_i}[\mu^j],\ell_{2,d_i}[\mu^j])=(\ell_{1,d_i}(\mu),\ell_{2,d_i}(\mu))$ for all $i=1,\dots,n$.
\end{enumerate}
Thus, to prove (1) of Theorem \ref{thm: special case}, it is sufficient to prove that
\begin{equation}\label{eqn:g3}\lim_{j\to\infty}(g_{1,e_i}[\mu^j],g_{2,e_i}[\mu^j],g_{3,e_i}[\mu^j],g_{4,e_i}[\mu^j])=(g_{1,e_i}(\mu),g_{2,e_i}(\mu),g_{3,e_i}(\mu),g_{4,e_i}(\mu))\end{equation}
for all $i=1,\dots,m$, and
\begin{equation}\label{eqn:g3'}\lim_{j\to\infty}(g_{1,e_i}[\mu^j],g_{2,e_i}[\mu^j],g_{3,e_i}[\mu^j],g_{4,e_i}[\mu^j])=(g_{1,e_i}(\mu),g_{2,e_i}(\mu),g_{3,e_i}(\mu),g_{4,e_i}(\mu))\end{equation}
for all $i=1,\dots,n$.
We will only give the proof of (\ref{eqn:g3}); a special case of the same argument also works for (\ref{eqn:g3'}).

Let $a\in\{1,\dots,k\}$ so that $S_a$ is the connected component of $S\setminus \Dmc$ that lies on the left of $e_i$. Also, let $p$ be the puncture of $S_a$ that corresponds to $e_i$. If the $\mu_a$-type of $p$ is quasi-hyperbolic or parabolic, then either $\ell_{1,e_i}(\mu)=0$ or $\ell_{2,e_i}(\mu)=0$. By (iv), one of $\ell_{1,e_i}[\mu^j]$ or $\ell_{2,e_i}[\mu^j]$ converges to $0$, while the other is bounded. This, together with the definition of $(g_{1,e_i},g_{2,e_i},g_{3,e_i},g_{4,e_i})$, imply that (\ref{eqn:g3}) holds.

On the other hand, if the $\mu_a$-type of $p$ is bulge $\pm\infty$, then $\ell_{1,e_i}(\mu),\ell_{2,e_i}(\mu)>0$, which implies that there is some $C>1$ so that
\[\frac{1}{C}<\ell_{1,e_i}[\mu^j],\ell_{2,e_i}[\mu^j]<C\]
for all $j$. If $i=1,\dots,m'$, this implies that the $\mu_a^j$-type of $p$ is bulge $\pm\infty$ as well. Since we assumed that $\lim_{j\to\infty}\mu^j_a=\mu_a$, the $\mu_a^j$-type of $p$ must agree with the $\mu_a$-type of $p$. From the definition of $(g_{1,e_i},g_{2,e_i},g_{3,e_i},g_{4,e_i})$, it follows that (\ref{eqn:g3}) holds. If $i=m'+1,\dots,m$, by Proposition \ref{prop:Dehntwistcoords}, we may choose representatives $\mu^j\in[\mu^j]$ so that 
\[-\frac{\ell_{1,e_i}(\mu^j)+\ell_{2,e_i}(\mu^j)}{2}\leq t_{e_i}(\mu^j)<\frac{\ell_{1,e_i}(\mu^j)+\ell_{2,e_i}(\mu^j)}{2},\] 
in which case
$$
-C< t_{e_i}(\mu^j)< C.
$$
Lemma \ref{lem:bulge} then implies that $\lim_{j\to\infty}r_{e_i}(\mu^j)=\pm\infty$ if the $\mu_a$-type of $p_L$ is bulge $\pm\infty$, so $r_{e_i}[\mu^j]=r_{e_i}(\mu^j)$ is positive (resp. negative) for sufficiently large $j$ if the $\mu_a$-type of $p$ is $+\infty$ (resp. $-\infty$). A straightforward calculation then proves that (\ref{eqn:g3}) holds.
\end{proof}

For any convex real projective structure $\mu\in\Cmc(S)$, let $\rho\in\mathrm{hol}(\mu)$, and let $\Omega$ be the $\rho$-equivariant developing  image of $\mu$. To prove (2) of Theorem \ref{thm: special case}, we need to define two properly convex domains $\widehat{\Omega}$ and $\widecheck{\Omega}$ so that
\begin{itemize}
\item $\rho(\pi_1(S))$ acts properly discontinuously on both $\widehat{\Omega}$ and $\widecheck{\Omega}$,
\item $\widehat{\Omega}$ and $\widecheck{\Omega}$ depend only on $\rho$, and
\item $\widecheck{\Omega}\subset\Omega\subset\widehat{\Omega}$.
\end{itemize}

First, we define $\widehat{\Omega}$. Proposition \ref{prop:transverse1} states that $\xi_{\rho}^{(2)}(x)$ and $\Omega$ do not intersect for any $x\in\widetilde{\Vmc}$. Also, Proposition \ref{prop:transverse} implies that $\xi_{\rho}^{(1)}(y)$ does not lie in $\xi_{\rho}^{(2)}(x)$ for all distinct $x,y\in\widetilde{\Vmc}$, which in particular implies that $\xi_{\rho}^{(2)}(x)\neq\xi_{\rho}^{(2)}(y)$. Since $\xi^{(1)}_{\rho}(\widetilde{\Vmc})\subset\partial\Omega$, this means that one of the two connected components of $\Rbbb\Pbbb^2\setminus(\xi_{\rho}^{(2)}(x)\cup\xi_{\rho}^{(2)}(y))$, denoted $H(x,y)$, contains $\xi_\rho^{(1)}(\widetilde{\Vmc}\setminus\{x,y\})$. With this, we can define $\widehat{\Omega}$ to be the interior of
\[\bigcap_{\mathrm{distinct }\,x,y\in\widetilde{\Vmc}}H(x,y).\]
It is clear that $\widehat{\Omega}$ is open, $\Omega\subset\widehat{\Omega}$, and $\rho(\pi_1(S))$ acts on $\widehat{\Omega}$. In particular, $\xi_\rho^{(1)}(\widetilde{\Vmc})\subset\partial\widehat{\Omega}$.

\begin{lem}\
\begin{enumerate}
\item The open set $\widehat{\Omega}\subset\Rbbb\Pbbb^2$ is properly convex.
\item The action of $\rho(\pi_1(S))$ on $\widehat{\Omega}$ is properly discontinuous.
\end{enumerate}
\end{lem}

\begin{proof}
(1) Let $x,y,z\in\widetilde{\Vmc}$ be a pairwise distinct triple of points. It is straightforward to check that since $t_{x,y,z}>0$, $\Rbbb\Pbbb^2\setminus(\xi_\rho^{(2)}(x)\cup\xi_\rho^{(2)}(y)\cup\xi_\rho^{(2)}(z))$ is a union of four properly convex (open) triangles. By definition, $\widehat{\Omega}$ has to lie in one of these four triangles, so $\overline{\widehat{\Omega}}$ does not contain an entire projective line in $\Rbbb\Pbbb^2$. With this, it is clear from the definition of $\widehat{\Omega}$ that $\widehat{\Omega}$ is properly convex.

(2) Since $\widehat{\Omega}$ is properly convex, we can define the Hilbert metric $d_{\widehat{\Omega}}$ on $\widehat{\Omega}$ (see proof of Proposition \ref{prop:devimage}), which is invariant under the $\rho(\pi_1(S))$ action on $\widehat{\Omega}$. Recall that $(\widehat{\Omega},d_{\widehat{\Omega}})$ is a proper path metric space. Hence, $\rho(\pi_1(S))$ acts properly discontinuously on $\widehat{\Omega}$ because $\rho(\pi_1(S))$ is a discrete subgroup of the isometry group of the Hilbert metric.
\end{proof}

Next, we define $\widecheck{\Omega}$ to be the interior of the convex hull of $\xi_{\rho}^{(1)}(\widetilde{\Vmc})$ in $\widehat{\Omega}$. Since $\Omega$ is properly convex and $\xi_{\rho}^{(1)}(\widetilde{\Vmc})\subset\overline{\Omega}$, we see that $\xi_{\rho}^{(1)}(\widetilde{\Vmc})$ is not contained in a projective line in $\Rbbb\Pbbb^2$. Thus, $\widecheck{\Omega}$ is non-empty. Furthermore, since $\xi_{\rho}^{(1)}(\widetilde{\Vmc})\subset\partial\Omega$, the convexity of $\Omega$ implies that $\widecheck{\Omega}\subset \Omega$. In particular, $\widecheck{\Omega}$ is a non-empty, properly convex subset of $\Rbbb\Pbbb^2$, on which $\rho(\pi_1(S))$ acts properly discontinuously. 

\begin{remark}\label{rem:cont}
By Proposition \ref{prop:conv-flag}, we see that $\widehat{\Omega}$ and $\widecheck{\Omega}$ depend only on $\rho$, and vary continuously with $\rho$.
\end{remark}

With this, we can prove (2) of Theorem \ref{thm: special case}
\begin{proof}[Proof of (2) of Theorem \ref{thm: special case}]
Since $\lim_{j\to\infty}\Psi_{\Dmc',\Dmc}[\mu^j]=\Psi_{\Dmc,\Dmc}(\mu)$, it is clear that $\lim_{j\to\infty}\Theta_{\Dmc',\Dmc}(\mathrm{hol}[\mu^j])=\Theta_{\Dmc,\Dmc}(\mathrm{hol}(\mu))$. It then follows from Remark \ref{rem: coordinates} that $\lim_{j\to\infty}\mathrm{hol}(\mu^j_i)=\mathrm{hol}(\mu_i)$ for all $i=1,\dots,k$. Then by Proposition \ref{prop:Dehntwistcoords}, we can find a representative $\rho^j_i\in\mathrm{hol}(\mu^j_i)$ and a representative $\rho_i\in\mathrm{hol}(\mu_i)$ so that $\lim_{j\to\infty}\rho^j_i=\rho_i$, and 
\[-\frac{\ell_{1,e}(\mu^j)+\ell_{2,e}(\mu^j)}{2}\leq t_{e}(\mu_i^j)<\frac{\ell_{1,e}(\mu^j)+\ell_{2,e}(\mu^j)}{2}\] 
for all $e\in\Dmc\setminus\Dmc'$. In particular, there is some $C>0$ so that $-C\leq t_e(\mu^j_i)\leq C$ for all $j,i$.

Since $\lim_{j\to\infty}\rho^j_i=\rho_i$, we see that $\lim_{j\to\infty}\xi_{\rho^j_i}=\xi_{\rho_i}$ uniformly. Also, by Remark \ref{rem:cont}, $\lim_{j\to\infty}\widehat{\Omega}^j_i=\widehat{\Omega}_i$ and $\lim_{j\to\infty}\widecheck{\Omega}^j_i=\widecheck{\Omega}_i$. Since $\widecheck{\Omega}^j_i\subset\Omega^j_i\subset\widehat{\Omega}^j_i$ and $\widecheck{\Omega}_i\subset\Omega_i\subset\widehat{\Omega}_i$, it follows from Theorem \ref{thm: hol} and Lemma \ref{lem:bulge} that $\lim_{j\to\infty}\Omega_i^j=\Omega_i$. (2) of Theorem \ref{thm: special case} follows.
\end{proof}

\section{Convex real projective structures via cubic differentials} \label{sec: conv cubic}

For a closed oriented surface $S$ of genus at least 2, Labourie \cite{Labourie07} and the first author \cite{Loftin2001} independently showed that a convex $\rp^2$ structure on $S$ is equivalent to a pair $(X,U)$, where $X$ is a complex structure on $S$ and $U$ is a holomorphic cubic differential on $X$. This correspondence was later extended to regular convex $\rp^2$ structures on the one hand and pairs $(X,U)$, where $X$ is a noded, connected Riemann surface and $U$ is a regular cubic differential over $X$ \cite{BenHul13,Nie15,Loftin2004,Loftin}. The notion of a regular $k$-differential is due to Bers \cite{Bers74}, while the geometric and analytic foundation of the relationship between cubic differentials and convex $\rp^2$ structures follows largely from deep work on hyperbolic affine spheres of Cheng-Yau \cite{ChengYau77,ChengYau86}.

To formally state this result, we recall some standard terminology from the theory of Riemann surfaces. Let $\bar X$ be a compact, noded Riemann surface. A neighborhood of each node of $\bar X$ is biholomorphic to a neighborhood of the origin in $\{(z,w)\in\Cbbb^2: zw=0\}$. We refer to $z$ and $w$ here as \emph{local coordinates} near the node. Let $P$ be a (possibly empty) finite collection of points in $\bar X$ that are not nodes, and let $X=\bar X \setminus P$. We refer to the points in $P$ as \emph{punctures} of $X$. Also let $\mathring{X}$ denote the complement of the nodes in $X$.
The \emph{normalization of $X$} is a smooth (possibly disconnected) Riemann surface equipped with a projection map to $X$ which is a biholomorphism restricted to the preimage of  $\mathring X$ and is two-to-one over each node.

\begin{definition}
A \emph{regular cubic differential} on $X$ is a meromorphic section $\sigma$ of the third tensor power of the holomorphic cotangent bundle over the normalization of $\bar X$ with the following properties:
\begin{itemize}
\item $\sigma$ is holomorphic on $\mathring{X}$,
\item $\sigma$ has poles of order at most $3$ at each node and puncture of $\bar{X}$.
\item the {\em residues} of $\sigma$ sum to 0 at each node, i.e. in terms $z,w$ local coordinates near the nodes, the third-order terms of the cubic differential are $R\,dz^3/z^3$ and $-R \, dw^3/w^3$, for a complex constant $R$.
\end{itemize}
\end{definition}

The set of regular cubic differentials on $X$ is naturally a finite dimensional complex vector space.

Via the unifomization theorem, the Teichm\"uller space $\Tmc(S)$ can be thought of as the deformation space of complex structures on $S$. From this point of view, the augmented Teichm\"uller space $\mathcal T(S)^{\rm aug}$ is a stratified space with strata $\mathcal T(S,\mathcal D)^{\rm aug}$ enumerated by the multi-curves $\mathcal D$ on $S$. Each stratum $\mathcal T(S,\mathcal D)^{\rm aug}$ is the deformation space of marked noded, compact Riemann surfaces with punctures $P$, with the property that the marking $f:S\setminus\Dmc\to\mathring{X}$ identifies
\begin{itemize}
\item neighborhoods of the punctures of $\mathring{X}$ with neighborhoods of the punctures of $S$,
\item neighborhoods of the nodes of $\mathring{X}$ to neighborhoods in $S$ of the curves in $\Dmc$.
\end{itemize}

Now for a fixed multi-curve $\mathcal D$, define
$$\mathcal T(S)^{\rm aug,\mathcal D} =  \bigcup_{\mathcal D'\subset \mathcal D} \mathcal T(S,\mathcal D')^{\rm aug},$$
with the subspace topology induced from that on $\mathcal T(S)^{\rm aug}$. Recall $G_{\mathcal D}$ is the subgroup of the mapping class group generated by Dehn twists around loops in $\mathcal D$, and define $\mathcal Q_{\mathcal D}:=\mathcal T^{\rm aug,\mathcal D}(S)/ G_{\mathcal D}$. Let $\Xmc_{\mathcal D}$ be the proper flat family of noded Riemann surfaces parametrized by $\mathcal Q_{\mathcal D}$, and let $\mathcal K^{\rm reg}_\Dmc$ be the complex vector bundle of regular cubic differentials over $\Qmc_\Dmc$. In other words, the fiber of $\Xmc_{\Dmc}$ above the point $X\in\Qmc_\Dmc$ is $X$ (see e.g.\ \cite{HubKoch14} for a full discussion), and the fiber of $\mathcal K^{\rm reg}_\Dmc$ above the point $X\in\Qmc_\Dmc$ is the vector space of regular cubic differentials on $X$.

With this notation, we can state the following theorem.
\begin{thm}\label{thm: cx-manifold}
$\mathcal Q_{\mathcal D}$ carries the natural structure of a complex manifold, and $\mathcal K^{\rm reg}_\Dmc$ is a holomorphic vector bundle over $\Qmc_\Dmc$. In particular, the total space of $\mathcal K^{\rm reg}_\Dmc$ has the structure of a complex manifold.
\end{thm}
Hubbard-Koch \cite{HubKoch14} construct the complex structure on $\mathcal Q_{\mathcal D}$. See
\cite{Loftin} and \cite{HubKoch14} for a proof that $\mathcal K^{\rm reg}_\Dmc$ is a holomorphic vector bundle.

In the setting when $S$ is a closed surface, the first author \cite{Loftin2001} and Labourie  \cite{Labourie07} independently established the following theorem.
\begin{thm}[Labourie, Loftin]\label{LL-thm}
Let $S$ be a closed connected oriented surface of genus at least two.
Then there is a canonical bijective correspondence between $\Cmc(S)$ and the total space of the vector bundle over $\mathcal T(S)$ whose fibers over a point $X\in\Tmc(S)$ is the space of cubic differentials on $X$. In particular, this defines a canonical complex structure on $\Cmc(S)$.
\end{thm}

The first author later extended this theorem to $\Cmc(S)^{\rm aug}$. More precisely, he proved the following (see Theorem 4.3.1 and Section 5.1 of \cite{Loftin}).

\begin{thm}\label{thm: Loftin}
There is a canonical continuous bijection $\Xi$ from the total space of $\mathcal K^{\rm reg}_{\Dmc}$ to the quotient space $\mathcal C(S)^{\rm aug,\mathcal D} / G_{\mathcal D}$. \end{thm}

See (1) of Remark \ref{rem: top} for the definition of $\mathcal C(S)^{\rm aug,\mathcal D} $.

\begin{remark}
The first author worked in the setting when $S$ is a closed surface, but his arguments show that Theorem \ref{thm: Loftin} holds for finite type surfaces with negative Euler characteristic as well.
\end{remark}

Thus Theorem \ref{thm: main}, together with the above theorem shows that $\Xi$ is a continuous one-to-one correspondence between manifolds of the same dimension. Then Brouwer's Invariance of Domain Theorem shows $\Xi$ is a homeomorphism. Corollary \ref{cor:Loftin} follows immediately.

\appendix

\section{The Proof of Theorem \ref{thm: hol}}
In this appendix, we give a proof of Theorem \ref{thm: hol}. We define an \emph{interval} in $\partial\Omega$ to be a subset of $\partial \Omega$ homeomorphic to an interval in $\mathbb R$. Note intervals need not be straight line segments. The key step in this proof is summarized in the following proposition.

\begin{prop} \label{rule-out-gaps}
Let $\mu\in\Cmc(S)$, let $\rho\in\mathrm{hol}_\mu$, and let $\Omega$ be the $\rho$-equivariant developing image of $\mu$. Also, let
\begin{eqnarray*}
\mathscr{E}(\rho)=\mathscr{E}&:=&\{\text{saddle fixed points of hyperbolic elements in }\rho(\pi_1(S))\},\\
\mathscr{F}(\rho)=\mathscr{F}&:=&\{\text{fixed points of non-identity elements in }\rho(\pi_1(S))\}\setminus\Emc,\\
\Gamma(\rho)=\Gamma&:=&\{\gamma\in\pi_1(S):\gamma\text{ is peripheral and }\rho(\gamma)\text{ is hyperbolic or quasi-hyperbolic}\},\\
\mathscr{J}(\rho)=\mathscr{J}&:=&\{\langle\rho(\gamma)\rangle:\gamma\in\Gamma\}.
\end{eqnarray*}
For each $H=\langle\rho(\gamma)\rangle\in\mathscr{J}$, let $I_H\subset\partial\Omega$ be the unique interval that does not contain any points in $\mathscr F$, and whose endpoints are:
\begin{itemize}
\item the attracting and repelling fixed points of $\rho(\gamma)$ if $\rho(\gamma)$ is hyperbolic,
\item the two fixed points of $\rho(\gamma)$ if $\rho(\gamma)$ is quasi-hyperbolic.
\end{itemize}
Then
\[\partial \Omega \setminus\overline{\mathscr{F}}=\bigcup_{H\in\mathscr{J}}I_H.\]
\end{prop}

The interval $I_H$ in Proposition \ref{rule-out-gaps} exists because every element in $H$ is peripheral. Assuming Proposition \ref{rule-out-gaps}, we will now prove Theorem \ref{thm: hol}. 

\begin{proof}[Proof of Theorem \ref{thm: hol}]Let $\gamma\in\pi_1(S)\setminus\{\id\}$. It is easy to observe from the dynamics of the $\langle\rho_i(\gamma)\rangle$ action on $\Rbbb\Pbbb^2$ that if $\rho_i(\gamma)$ has a fixed point $p$ that does not lie in $\partial\Omega_i\cap\partial\Omega_i'$, then $\rho_i(\gamma)$ is necessarily hyperbolic and $p$ is the saddle fixed point of $\rho_i(\gamma)$. Since $\partial\Omega_i\cap\partial\Omega_i'\subset\Rbbb\Pbbb^2$ is closed, it follows that $\overline{\mathscr F}\subset\partial\Omega_i\cap\partial\Omega_i'$. Proposition \ref{rule-out-gaps} then implies that
\[\partial \Omega_i\setminus\overline{\mathscr{F}(\rho_i)}=\bigcup_{H\in\mathscr{J}(\rho_i)}I_H\,\,\,\text{ and }\,\,\,\partial \Omega_i'\setminus\overline{\mathscr{F}(\rho_i)}=\bigcup_{H\in\mathscr{J}(\rho_i)}I_H',\]
where $I_H\subset\partial\Omega_i$ and $I_H'\subset\partial\Omega_i'$ are the unique open intervals that do not contain any points in $\mathscr{F}(\rho_i)$, and whose endpoints are
\begin{itemize}
\item the attracting and repelling fixed points of $\rho_i(\gamma)$ if $\rho_i(\gamma)$ is hyperbolic,
\item the two fixed points of $\rho_i(\gamma)$ if $\rho_i(\gamma)$ is quasi-hyperbolic.
\end{itemize}

If $\rho_i(\gamma)$ is quasi-hyperbolic, then observe from the dynamics of the $H:=\langle\rho_i(\gamma)\rangle$ action on $\Rbbb\Pbbb^2$ that $I_H=I_H'$ is the straight line segment between the two fixed points of $\rho_i(\gamma)$, or else the convexity of $\Omega$ is violated. (See for example \cite{Loftin2004}).

On the other hand, if $\rho_i(\gamma)$ is hyperbolic, and $H:=\langle\rho_i(\gamma)\rangle$, then the admissibility of $\mu_i$ implies that $I_H$ is either the (open) edge of the principal triangle of $\rho_i(\gamma)$ whose endpoints are $\{\rho_i(\gamma)^+,\rho_i(\gamma)^-\}$ (bulge $-\infty$), or the union of $\rho_i(\gamma)^0$ with the two edges of the principal triangle of $\rho_i(\gamma)$ that have $\rho_i(\gamma)^0$ as a common vertex (bulge $+\infty$). In either case, the endpoints of $I_H$ are $\{\rho_i(\gamma)^+,\rho_i(\gamma)^-\}$. The admissibility of $\mu_i'$ implies the same for $I_H'$. Thus, if $I_H\neq I_H'$, then $I_H\cup I_H'$ is the union of $\rho_i(\gamma)^0$ with the three edges of the principal triangle of $\rho_i(\gamma)$. See Figure \ref{admissible-hyperbolic-figure}.

\begin{figure}
\begin{center}
\includegraphics[scale=1.1]{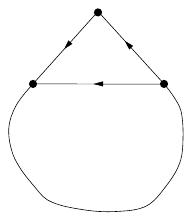}
\small
\put (-102, 71){$\rho(\gamma)^+$}
\put (-10, 71){$\rho(\gamma)^-$}
\put (-54, 113){$\rho(\gamma)^0$}
\end{center}
\caption{Principal triangle}
\label{admissible-hyperbolic-figure}
\end{figure}

It follows immediately from this that the interior of the symmetric difference $\Omega_i\, \triangle\, \Omega_i'$ is the union of a $\pi_1(S_i)$-invariant subset of triangles in $\Gmc_{\rho_i}$.
\end{proof}

It is thus sufficient to prove Proposition \ref{rule-out-gaps}. To do so, we need the notion of a limit set. First of all, recall that $\widecheck{\Omega}\subset\Omega$ is the interior of the convex hull in $\Omega$ of $\overline{\mathscr{F}}$, and that $\rho(\pi_1(S))$ acts on $\widecheck{\Omega}$ freely properly discontinuously.  

\begin{definition}
For $p\in \widecheck{\Omega}$, define
the \emph{limit set} $\Lambda_p=\Lambda_p(\rho)$ of the $\rho(\pi_1(S))$-action on $\widecheck{\Omega}$ to be the set of accumulation points in $\partial\widecheck{\Omega}$ of  $\rho(\pi_1(S))\cdot p$.
\end{definition}

Next, we want to prove that $\Lambda_p=\overline{\mathscr{F}}$ for all $p\in\widecheck{\Omega}$. 

\begin{lem} \label{limit-set-lemma}
For every $p\in\widecheck{\Omega}$, $\Lambda_p=\overline{\mathscr{F}}$. In particular, $\Lambda_{p_1}=\Lambda_{p_2}$ for all $p_1,p_2\in\widecheck{\Omega}$.
\end{lem}

\begin{proof}
The proof of this lemma in large part follows Kuiper \cite{Kuiper54}, p.\ 208.  We recall the argument for the reader's convenience. 

It is clear that $\overline{\mathscr{F}}\subset\Lambda_p$. Let $m$ be any maximal open subinterval of $\partial\widecheck{\Omega}\setminus\overline{\mathscr{F}}$. By our definition of $\widecheck{\Omega}$, $m$ is a projective line segment. We will now prove that no points in $m$ can lie in $\Lambda_p$. This immediately implies that $\Lambda_p\subset\overline{\mathscr{F}}$.

Suppose for contradiction that there is some $q\in m$ so that $q\in\Lambda_p$. This means that there is a sequence $\{\gamma_i\}_{i=1}^\infty\subset\pi_1(S)$ so that $\rho(\gamma_i)\cdot p\to q$ as $i\to\infty$. The sequence $\{\rho(\gamma_i)\}_{i=1}^\infty$ lies in $\PGL(3,\Rbbb)\subset\Pbbb(\mathrm{End}(\Rbbb^3))$, and $\Pbbb(\mathrm{End}(\Rbbb^3))$ is compact. Thus (possibly passing to a subsequence), we may assume  this sequence $\rho(\gamma_i)$ has a limit $g_\infty\in\Pbbb(\mathrm{End}(\Rbbb^3))$, which is the projectivization of a linear endomorphism $L_\infty:\Rbbb^3\to\Rbbb^3$ of rank $1$ or $2$. In other words, $g_\infty$ is a projective map 
\[g_\infty:\Rbbb\Pbbb^2\setminus\Pbbb(\ker L_\infty)\to\Rbbb\Pbbb^2\] 
whose image is $\Pbbb(\im \,L_\infty)$. Moreover, $\rho(\gamma_i)\to g_\infty$ uniformly on compact subsets of $\rp^2 \setminus \Pbbb(\ker L_\infty)$.  

Consider a geodesic ball $B$ of radius $\epsilon>0$ centered at $p$ with respect to the Hilbert metric on $\widecheck{\Omega}$. Then
\[g_\infty(B) =\lim_{i\to\infty} \rho(\gamma_i)\cdot B,\] 
so the proper discontinuity of $\rho$ implies that $g_\infty(B) \subset \partial\widecheck{\Omega}$. Since $\rho(\gamma_i)$ is an isometry for all $i$, by the definition of the Hilbert metric and the fact that $q\in\partial\widecheck{\Omega}$ is in the interior of a line segment in the boundary, $g_\infty(B)$ is not a point. This implies that $L_\infty$ has rank $2$, and that $g_\infty(B)$ has to be an open subsegment of $m$. In particular, $t:=\Pbbb(\ker L_\infty)\in\partial\widecheck\Omega$ is a single point and $\ell:=\Pbbb(\im\, L_\infty)$ is the projective line such that $\overline{m}=\partial\widecheck\Omega\cap\ell$. 

Using this, we will now show that $\widecheck{\Omega}$ is an open triangle. This will be a contradiction because it is easy to see that the projective automorphism group of such a triangle is virtually Abelian. Thus, since $S$ has negative Euler characteristic,  there is no injective representation $\pi_1(S)\to {\rm Aut}(\widecheck{\Omega})$. 

First, observe that for any point $s\in\ell$, $g_\infty^{-1}(s)$ is a line through $t$ with $t$ removed. Conversely, every line through $t$ with $t$ removed is sent to a single point in $\ell$ by $g_\infty$. Furthermore, 
\[g_\infty\left(\overline{\widecheck\Omega}\right)=\lim_{i\to\infty} \rho(\gamma_i)\cdot \overline{\widecheck\Omega} \subset\overline{\widecheck\Omega}\cap\ell=\partial\widecheck\Omega\cap\ell=\overline m.\]
Since $g_\infty\left(\overline{\widecheck\Omega}\right)\supset g_\infty(\overline m)=\overline m$, it follows that $g_\infty\left(\overline{\widecheck\Omega}\right)=\overline{m}$.
As such, if we set 
\[X = \{ x\in \rp^2\setminus\{t\} : \overline{tx} \cap \overline{\widecheck{\Omega}}\neq \emptyset\},\] 
where $\overline{tx}$ is the projective line in $\rp^2$ passing through $t$ and $x$, then $g_\infty(X)=\overline{m}$. Since $\widecheck{\Omega}$ is properly convex, there are a pair of distinct lines $\ell_1$ and $\ell_2$ that intersect at $t$, such that $X\cup\{t\}$ is the closure of one of the two connected components of $\rp^2\setminus(\ell_1\cup\ell_2)$. Note then that $g_\infty(\ell_1)$, $g_\infty(\ell_2)$ are the endpoints of $\overline{m}$, and all the points in the interior of $X$ are mapped by $g_\infty$ to $m$.

Since $\overline{\widecheck\Omega}$ lies in $X\cup\{t\}$ and $t\in\partial\widecheck\Omega$, it follows that $a_1:=\partial\widecheck\Omega\cap\ell_1$ and $a_2:=\partial\widecheck\Omega\cap\ell_2$ are proper subsegments of $\ell_1$ and $\ell_2$ with $t$ as a common vertex (it is possible that $a_i=\{t\}$). Let $a_3:=\partial\widecheck\Omega\setminus(a_1\cup a_2)$. To show that $\widecheck\Omega$ is an open triangle in $\rp^2$, it now suffices to show that $a_3$ is a subsegment of a projective line.

If $y$ lies in the interior of $a_3$, then it lies in the interior of $X$, so $g_\infty(y) \in m$. Since $\rho(\gamma_i)$ leaves $\partial \widecheck{\Omega}$ invariant for all $i$, we see that $\rho(\gamma_i)\cdot y \in m$ for large $i$.  Thus, there is a neighborhood $I_y$ of $y$ in $\partial\widecheck{\Omega}$ such that $\rho(\gamma_i)\cdot I_y\subset m$ for sufficiently large $i$. This means that $I_y$ is a projective line segment for all $y$ in the interior of $a_3$, so $a_3$ is a projective line segment. 
\end{proof}

\begin{remark}
In the previous lemma, we can weaken the restriction that $p\in\widecheck{\Omega}$.  In fact, $\Lambda_p = \overline{ \mathcal F}$ for any $p\in\Omega$, by Proposition \ref{rule-out-gaps} and basic facts about the dynamics of hyperbolic elements on the principal triangle. 
\end{remark}

With this, we can now finish the proof of Proposition \ref{rule-out-gaps}.

\begin{proof}[Proof of Proposition \ref{rule-out-gaps}]
It is clear from the definitions that
\[\partial \Omega \setminus\overline{\mathscr F}\supset\bigcup_{H\in\mathscr J}I_H,\]
so it is sufficient to prove the other inclusion, i.e. every maximal open interval in $\partial\Omega\setminus\overline{\mathscr F}$ is of the form $I_H$ for some $H\in\mathscr J$. Let $\widecheck{\Omega}\subset \Omega$ be the convex hull in $\Omega$ of $\overline{\mathscr F}$. Since $\partial\Omega$ and $\partial\widecheck\Omega$ are both oriented topological circles and the cyclic ordering on $\overline{\mathscr F}$ induced by both $\partial\Omega$ and $\partial\widecheck\Omega$ agree, we see that there is a canonical bijection between connected components of $\partial\Omega\setminus\overline{\mathscr{F}}$ and $\partial\widecheck\Omega\setminus\overline{\mathscr{F}}$. It is thus sufficient to prove the proposition for $\widecheck\Omega$ in place of $\Omega$.

Let $\Sigma := \widecheck{\Omega}/\rho(\pi_1(S))$ and let $\pi:\widecheck{\Omega}\to\Sigma$ be the projection map. Choose any $p\in\widecheck{\Omega}$, and let $B_r(p)\subset\widecheck{\Omega}$ denote the open ball (with respect to the Hilbert metric) centered at $p$ with radius $r$. Consider $r$ large enough so that $\Sigma\setminus \pi(\overline{B_r(p)})$ is a disjoint collection of open cylinders, one for each end of $S$. Then observe that
\[D := \bigcup_{\gamma\in\pi_1(S)} \rho(\gamma) \cdot B_r(p),\]
is connected. Furthermore, our choice of $r$ ensures that each connected component of $\widecheck{\Omega}\setminus \overline{D}$ is the image of the developing map of a connected component of $\Sigma\setminus \overline{\pi(B_r(p)})$, which is a cylinder.

Let $m$ be a maximal open line segment in $\partial \widecheck{\Omega}\setminus\overline{\mathscr F}$. By Lemma \ref{limit-set-lemma}, every $q\in m$ is not a limit point of the $\rho(\pi_1(S))$ action on $\widecheck{\Omega}$. This implies that there is a connected component $N$ of $\widecheck{\Omega}\setminus\overline{D}$ that contains $m$ in its boundary. Furthermore, since $D$ is connected and both endpoints of $m$ lie in $\overline{D}$, we see that $N\cap\partial\widecheck{\Omega}=m$.  See Figure \ref{classify-ends-figure}.

\begin{figure}
\begin{center}
\includegraphics{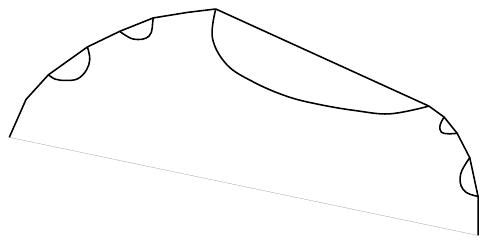}
\Large
\put (-80, 92){$m$}
\put (-120, 42){$D$}
\put (-100, 77){$N$}
\end{center}
\caption{A gap in the limit set}
\label{classify-ends-figure}
\end{figure}

Since $N$ is the image of the developing map of a convex $\Rbbb\Pbbb^2$ cylinder, there is an infinite cyclic subgroup $H=\langle \rho(\gamma)\rangle\subset\rho(\pi_1(S))$ that preserves $N$. The subgroup $H$ also preserves $\partial\widecheck{\Omega}$, so $H$ preserves $\overline{m}$, which implies that it fixes both endpoints of $m$ and preserves $m$. This means that $\rho(\gamma)$ is either hyperbolic or quasi-hyperbolic. The fact that $m$ does not contain any points in $\overline{\mathscr F}$ immediately implies that $\gamma\in\pi_1(S)$ is peripheral, and so $m=I_H$ for some $H\in\mathscr J$.
\end{proof}

\end{document}